\documentclass[10pt,reqno]{amsart}
\usepackage{amssymb,mathrsfs,color}
\usepackage{pinlabel}

\usepackage{enumerate}

\usepackage{graphicx}
\usepackage{xcolor} 
\usepackage{tensor}
\usepackage{slashed}
\usepackage{cite}
\definecolor{green}{rgb}{0,0.5,0} 

\newcommand{\nrm}[1]{\Vert#1\Vert}
\newcommand{\brk}[1]{\langle#1\rangle}
\newcommand{\set}[1]{\{#1\}}

\newcommand{\tr}{\textrm{tr}}

\newcommand{\aeq}{\sim}
\newcommand{\aleq}{\lesssim}

\newcommand{\ud}{\mathrm{d}}
\newcommand{\rd}{\partial}
\newcommand{\nb}{\nabla}

\newcommand{\bb}{\Big}

\newcommand{\alp}{\alpha}
\newcommand{\bt}{\beta}
\newcommand{\gmm}{\gamma}
\newcommand{\Gmm}{\Gamma}
\newcommand{\dlt}{\delta}
\newcommand{\Dlt}{\Delta}
\newcommand{\eps}{\epsilon}

\newcommand{\lmb}{\lambda}

\newcommand{\Sgm}{\Sigma}
\newcommand{\tht}{\theta}

\newcommand{\bfd}{{\bf d}}

\newcommand{\bfh}{{\bf h}}

\newcommand{\bfD}{{\bf D}}

\newcommand{\bfR}{{\bf R}}

\newcommand{\bbD}{\mathbb D}

\newcommand{\bbH}{\mathbb H}

\newcommand{\bbN}{\mathbb N}

\newcommand{\bbR}{\mathbb R}
\newcommand{\bbS}{\mathbb S}

\newcommand{\bbZ}{\mathbb Z}

\newcommand{\calE}{\mathcal E}

\newcommand{\calN}{\mathcal N}

\newcommand{\calP}{\mathcal P}

\newcommand{\calR}{\mathcal R}
\newcommand{\calS}{\mathcal S}

\newcommand{\calX}{\mathcal X}
\newcommand{\calY}{\mathcal Y}

\newcommand{\arctanh}{\mathrm{arctanh} \, }

\usepackage{wrapfig}
\usepackage{tikz}
\usetikzlibrary{arrows,calc,decorations.pathreplacing}
\definecolor{light-gray1}{gray}{0.90}
\definecolor{light-gray2}{gray}{0.80}

\definecolor{deepgreen}{cmyk}{0.1,0.7,0.8,0.1}

\newcommand{\E}{\mathcal{E}}

\newcommand{\LL}{\mathcal{L}}
\newcommand{\NN}{\mathcal{N}}
\newcommand{\MM}{\mathcal{M}}
\newcommand{\cS}{\mathcal{S}}

\newcommand{\RR}{\mathcal{R}}

\newcommand{\Hp}{\mathbb{H}}
\newcommand{\N}{\mathbb{N}}
\newcommand{\R}{\mathbb{R}}
\newcommand{\Sp}{\mathbb{S}}

\newcommand{\D}{\mathbb{D}}

\newcommand{\g}{\mathbf{g}}
\newcommand{\h}{\mathbf{h}}
\newcommand{\m}{\mathbf{m}}

\newcommand{\etab}{\boldsymbol{\eta}}

\newcommand{\Db}{\mathbf{D}}

\newcommand{\al}{\alpha}
\newcommand{\be}{\beta}
\newcommand{\ga}{\gamma}
\newcommand{\de}{\delta}

\newcommand{\om}{\omega}
\newcommand{\la}{\lambda}
\newcommand{\te}{\theta}
\newcommand{\s}{\sigma}

\newcommand{\De}{\Delta}

\newcommand{\Ga}{\Gamma}
\newcommand{\La}{\Lambda}
\newcommand{\Sig}{\Sigma}

\newcommand{\p}{\partial}
\newcommand{\na}{\nabla}

\makeatletter

\newcommand{\Rmnum}[1]{\expandafter\@slowromancap\romannumeral #1@}
\makeatother

\newcommand{\I}{\infty}

\newcommand{\ti}{\widetilde}

\newcommand{\ang}[1]{\left\langle{#1}\right\rangle}
\newcommand{\abs}[1]{\left\lvert{#1}\right\rvert}
\newcommand{\ds}{\displaystyle}

\newcommand{\ali}[1]{\begin{align}\begin{split} #1 \end{split}\end{align}}
\newcommand{\ant}[1]{\begin{align*}\begin{split} #1 \end{split}\end{align*}}
\newcommand{\EQ}[1]{\begin{equation}\begin{split} #1 \end{split}\end{equation}}

\newcommand{\Del}[1]{}

\numberwithin{equation}{section}

\newtheorem{thm}{Theorem}[section]
\newtheorem{cor}[thm]{Corollary}
\newtheorem{lem}[thm]{Lemma}
\newtheorem{prop}[thm]{Proposition}

\newtheorem{conj}{Conjecture}

\theoremstyle{remark}
\newtheorem{rem}{Remark}
\newtheorem{defn}{Definition}

\newcommand{\mand}{{\ \ \text{and} \ \  }}

\newcommand{\mif}{{\ \ \text{if} \ \ }}
\newcommand{\mfor}{{\ \ \text{for} \ \ }}
\newcommand{\mas}{{\ \ \text{as} \ \ }}

\newcommand{\dvol}{\operatorname{dvol}}

\newcommand{\zero}{\mathbf{0}}

\newcommand{\Ls}{L_{\frac{ds}{s}}}

\renewcommand{\ds}{\frac{ds}{s}}

\newcommand{\dsp}{\frac{ds^\prime}{s^\prime}}

\newcommand{\rest}{\!\!\restriction}

\begin{document}

\title[Wave maps on hyperbolic space]{The Cauchy problem for wave maps on hyperbolic space in dimensions $d \ge 4$}

\author{Andrew Lawrie}
\author{Sung-Jin Oh}
\author{Sohrab Shahshahani}

\begin{abstract}
We establish global well-posedness  and scattering for wave maps from $d$-dimensional hyperbolic space into Riemannian manifolds  of bounded geometry for initial data that is small in the critical Sobolev space for $d \geq 4$. The main theorem is proved using the moving frame  approach introduced  by Shatah and Struwe. However, rather than imposing the Coulomb gauge  we formulate the wave maps problem in Tao's caloric gauge, which is constructed using the harmonic map heat flow. 
In this setting the caloric gauge has the remarkable property that the main `gauged' dynamic equations reduce to a system of nonlinear \emph{scalar} wave equations on $\Hp^d$ that are amenable to Strichartz estimates  rather than \emph{tensorial} wave equations (which arise in other gauges such as the Coulomb gauge) for which useful dispersive estimates are not known.  This last point makes the heat flow approach crucial in the context of wave maps on curved domains. 
\end{abstract}

\thanks{Support of the National Science Foundation (NSF), DMS-1302782 and NSF 1045119 for the first and third authors, respectively, is gratefully acknowledged. The second author is a Miller Research Fellow, and acknowledges support from the Miller Institute. The authors were also supported by the NSF under Grant~No.0932078000 while in residence at the MSRI in Berkeley, CA during Fall 2015.}

\maketitle

\section{Introduction}

Let $(\MM, \etab)$ be a Lorentzian manifold with metric $\etab$ given by the product $\MM = \R \times  \Sigma$ where $\Sigma = (\Sigma, \h)$ is a Riemannian manifold of dimension $d$ with metric $\h$. Let $(\NN, \g)$ be a complete Riemannian manifold without boundary of dimension $n$. A wave map $u: \MM \to \NN$ is a formal critical point of the Lagrangian action 
\EQ{ \label{lag} 
\LL(u) := \int_{\MM}  \ang{ du, du}_{  \etab^{-1}\otimes u^* \g} \, \dvol_{\etab}. 
}
Above we view the differential $du$ as a section of the vector bundle $T^* \MM \otimes u^* T \NN$, which is endowed with the metric $ \etab^{-1}\otimes u^* \g$. Here $u^*T \NN$ denotes the pullback of $T\NN$ by the map $u$ and $u^*\g$ is the pullback metric. In coordinates $\{x^{\al}\}_{\al  =0}^d$ on  $\MM$ and $\{u^j\}_{j=1}^n$ on $\NN$ the Lagrangian density becomes 
\EQ{
\ang{ du, du}_{  \etab^{-1}\otimes u^* \g} \, \dvol_{\etab}
= \etab^{\al \be} \g_{jk}(u) \p_\al u^j \p_\be u^k  \, \sqrt{ \abs{ \etab}} \, dx
}
The Euler-Lagrange equations for~\eqref{lag} are given  by 
\EQ{ \label{wmi} 
\etab^{\alp \bt} D_{\alp} \rd_{\bt} u = 0
}
where $D$ is the pullback covariant derivative on $u^* T\NN \otimes T^{\ast} \MM$. Expanding the above we obtain following system of nonlinear wave equations on $\MM = \R \times \Sigma$: 
\ant{ 
\Box_{\Sigma} u^k =  \etab^{\al \be} \, {}^{(\NN)} \Ga^k_{ij}(u) \p_{\al}u^i \p_{\be} u^j
}
where  $\Box_{\Sig} = -\p_t^2 + \Delta_{\Sigma}$ is the D'Alembertian on $\MM$, $\Delta_{\Sigma}$ is the Laplace-Beltrami operator on $\Sigma$, and the ${}^{(\NN)} \Ga^k_{ij}$ are the Christoffel symbols associated to the metric connection on $T\NN$. 

Equivalently, by the Nash embedding theorem we can view $\NN$ as an isometrically embedded submanifold of $(\R^N, \ang{ \cdot, \cdot})$. In this case we write $u = (u^1, \dots, u^N)$ where $u: \MM \to  \NN \hookrightarrow \R^N$. A wave map $u$ is then a solution to 
\EQ{ \label{wme}
 \Box_{\Sigma} u  \perp T_u \NN
 }
which can be expressed in coordinates as
\EQ{ \label{wme-2}
\Box_{\Sigma} u^K  =  \etab^{\al \be} S^K_{IJ}(u)\p_\al u^I  \p_\be u^J 
}
where $S = S^K_{IJ}$ is the second fundamental form of the embedding $\NN \hookrightarrow  \R^N$. 

In this paper we consider  the domain 
\ant{
\Sigma  = \Hp^d, \quad d \ge 4
}
where $\Hp^d$ is $d$-dimensional hyperbolic space. To simplify the exposition we restrict to dimension $d =4$ and $\Sigma =  \Hp^4$, noting that the results and methods used here apply equally well to domains $\Sigma = \Hp^d$ for all  $d \ge 4$. To avoid technical distractions, we also assume that $\NN$ is a \emph{closed} manifold (i.e., compact without boundary), which ensures uniform boundedness of the second fundamental form $S(u)$ and its derivatives in the Nash embedding. We refer to Remark~\ref{rem:noncpt} for a discussion on the case when $\NN$ is non-compact.

We study the Cauchy problem for~\eqref{wme} for smooth finite energy initial data $(u,  \p_t u)\!\!\restriction_{t=0} = (u_0, u_1)$ given by maps 
\EQ{ \label{eq:data} 
u_0:  \Hp^4 \to \NN , \quad 
u_1:  \Hp^4 \to u_0^* T\NN , \quad u_1(x) \in  T_{u_0(x)}  \NN
} 
This means that we can think of the pair of  initial data $(u_0, u_1)$ as a map into the tangent space of the target  $T\NN$ and   we will also write $\vec u(0) :=(u_0, u_1): \Hp^4 \to  T\NN$.  We will further assume that  we can find a point $u_{\infty} := u_0(\infty) \in \NN$ and a compact subset $K \Subset \Hp^4$ so that 
\EQ{ \label{eq:data1} 
(u_0(x), u_1(x)) = (u_0( \infty), 0) \in T_{u_0(\infty)} \NN \mif x \in \Hp^4 - K, 
}
that is, the data is a constant map in the exterior of a compact subset of $\Hp^4$. 

A simple version of our main result is as follows. 




\begin{thm} \label{t:main} 
There exists an $\eps_0>0$ small enough so that for all smooth  $(u_0, u_1)$ as in~\eqref{eq:data}, \eqref{eq:data1} with 
\ant{
\| (d u_0, u_1)\|_{H^1 \times H^1( \Hp^4; T\NN)} < \eps_0
}
there exists a unique, global smooth solution $u: \R \times \Hp^4 \to \NN \subset \R^N$ to~\eqref{wme} with initial data $(u, \p_t u) \!\! \restriction_{t=0}  = (u_0, u_1)$ and  satisfying 
\ant{
\sup_{t \in \R} \| (du(t), \p_t u(t)) \|_{H^1 \times H^1( \Hp^4; T\NN)}  \lesssim \eps_0. 
}
Moreover we have the following qualitative pointwise decay 
\ant{
 \| u(t, \cdot) - u_0({\I}) \|_{L^{\infty}_x} \to 0 \mas t \to \pm \infty.
}
\end{thm}

Theorem~\ref{t:main} is a global well-posedness and scattering result for small critical Sobolev data, which is of similar flavor to the results of \cite{Tao1, KR01, SS02, NSU} in the flat high dimensional case ($\Sgm = \bbR^{d}$ with $d \geq 4$).  A major motivation for the latter results was to develop techniques to approach the two dimensional case $\Sgm = \bbR^{2}$, which is challenging due to weaker dispersion but particularly interesting in view of energy criticality of the wave maps equation. In fact, based on the small critical Sobolev data results \cite{Tao2, Kri04, Tat05}, a satisfactory large data result on $\bbR \times \bbR^{2}$ was recently established in \cite{ST1, ST2, KS, Tao37}. By the same token, we view Theorem~\ref{t:main} as a first step towards the interesting case $\Sgm = \bbH^{2}$, whose study was initiated by the authors \cite{LOS1, LOS3, LOS4} under symmetry assumptions. We refer to Section~\ref{s:motivate} for a more detailed discussion of the history and motivation of the problem.

A particularly simple proof of the flat space version of Theorem~\ref{t:main} was given by Shatah and Struwe \cite{SS02}, which relied on the use of the derivative formulation of the wave maps equation in the Coulomb gauge (see also \cite{NSU} for a similar approach). However the analysis of \cite{SS02} does not extend to our case, due to tensoriality of the main equations. Instead, we formulate the wave map problem in Tao's caloric gauge, which leads to equations with scalar principal part with nonlinearity amenable to perturbative analysis. In view of the effectiveness of the caloric gauge on $\bbR \times \bbR^{2}$ \cite{Tao37}, we expect the framework developed in this paper to be useful in the important case of $\Sgm = \bbH^{2}$ as well. In Section~\ref{s:main-idea} below, we will give a further explanation of the main difficulties of our problem, and how Tao's caloric gauge provides their resolution.

\begin{rem}
In Section~\ref{s:outline} we will give a more precise formulation of Theorem~\ref{t:main} that includes the boundedness of appropriate Strichartz-type norms in the caloric gauge as well as a notion of scattering. Note that the norm $H^1( \Hp^4; T\NN)$ is a Sobolev space, which will be defined in \eqref{eq:extr-TN-nrm} in Section~\ref{s:fs}. 
\end{rem} 

\begin{rem}
Theorem~\ref{t:main} can be easily extended to wave maps $u: \R \times \Hp^d \to \NN$ for all $d \ge 4$ using arguments identical to those contained in this paper.  We have chosen to present the proof only for dimension $d =4$ to keep the exposition  as  simple as possible. Unfortunately,  the techniques presented here do not suffice to obtain the same results in dimensions $d=2,3$, due to the failure of favorable Strichartz estimates (see Theorem~\ref{t:str}).
 \end{rem} 

\begin{rem}\label{rem:noncpt}
The restriction to a compact target manifold $\NN$ is not expected to be essential, and Theorem~\ref{t:main} should hold for any smooth Riemannian manifold $\NN$ with bounded geometry (i.e., the curvature tensor and its covariant derivatives are uniformly bounded). One way to extend our proof to the noncompact case would be to replace the a priori estimates we proved using the isometric embedding by their intrinsic counterparts, as was done in the case of the flat domain $\Sgm = \bbR^{d}$ \cite{Tao37, Sm}. To avoid distracting technicalities, we will not address this issue in detail.
\end{rem} 
 
 \subsection{Main ideas}  \label{s:main-idea}
A key difficulty for establishing well-posedness of the wave maps equation at the critical Sobolev regularity is the presence of a quadratic nonlinearity with derivatives (henceforth referred to as a `nonperturbative part' of the nonlinearity) which precludes a direct iteration argument based on estimates for $\Box_{\Sgm}$; see \cite{DAG04}. 

In the breakthrough work of Tao \cite{Tao1, Tao2}, it was realized that the gauge structure of the wave maps equation can be used to renormalize away the nonperturbative part of the nonlinearity. This idea has its roots in the harmonic maps literature, e.g., see H\'elein \cite{Hel}. The work of Shatah and Struwe \cite{SS02} and Nahmod, Stefanov and Uhlenbeck \cite{NSU} further clarified the picture by exhibiting a particular representation of the wave maps equation, namely the \emph{derivative formulation} in the \emph{Coulomb gauge}, in which the whole nonlinearity is essentially cubic or higher. It was shown in \cite{SS02, NSU} that in this formulation, the wave maps equation is amenable to perturbative analysis based on $\Box_{\Sgm}$ in the small critical Sobolev data setting on $\bbR \times \bbR^{d}$ when $d \geq 4$. We begin by summarizing the main ideas of this approach, as it will provide a starting point for our discussion.

\subsubsection*{Derivative formulation and the Coulomb gauge}
In the derivative formulation, one works with the differential $d u$ instead of the map $u$ itself. Being a section of the vector bundle $u^{\ast} T \NN$ (pullback of the tangent bundle $T \NN$ by $u$), the differential $d u$ has the advantage of being linear. With respect to a global field $e$ of orthonormal frames of $u^{\ast} T \NN$, which exists since the domain $\bbR \times \Sgm$ is  contractible,  $d u$ can be conveniently written as an $\bbR^{n}$-valued 1-form $\Psi = \psi_{\alp} \, d x^{\alp}$. Denoting by $\bfD_{\alp}$ the pullback covariant derivative of $u^{\ast} T \NN$ in the frame $e$, the wave maps equation can be rephrased as the div-curl system\footnote{We adopt the usual index notation for tensors, where the greek letters $\alp, \bt, \ldots$ correspond to the global space-time coordinates $x^{\alp} = (t = x^{0}, x^{1}, \ldots, x^{d})$ and the latin letters $a, b, \ldots$ represent the space coordinates $x^{a} = (x^{1}, \ldots, x^{d})$. For further explanation of the notation, we refer to Section~\ref{s:prelim}.}
\begin{equation} \label{eq:psi-div-curl}
\etab^{\alp \bt} \bfD_{\alp} \psi_{\bt} = 0, \quad
\bfD_{\alp} \psi_{\bt} - \bfD_{\bt} \psi_{\alp} = 0
\end{equation}
The nonlinearity in this formulation is encoded in $\bfD_{\alp}$, which takes the form $\rd_{\alp} + A_{\alp}$ for some real skew-symmetric $n \times n$-matrix valued 1-form $A_{\alp} d x^{\alp}$ (connection 1-form). Considering the action of the commutator $[\bfD_{\alp}, \bfD_{\bt}]$ on sections of $u^{\ast} T \NN$, it follows that $A_{\alp}$ obeys the curl equation
\begin{equation} \label{eq:A-curl}
	\nb_{\alp} A_{\bt} - \nb_{\bt} A_{\alp} + [A_{\alp}, A_{\bt}] = \bfR(u)(\psi_{\alp}, \psi_{\bt})
\end{equation} 
where $\bfR(u)$ is the pullback Riemann curvature. However, $A_{\alp}$ is otherwise undetermined. The gauge invariance of the wave maps equation manifests as the freedom of the choice of $e$, which in turn determines $A_{\alp}$.

The Coulomb gauge is given by a frame $e$ with a connection 1-form $A$ whose restriction to the constant-$t$ hypersurfaces $\set{t} \times \Sgm$ has zero divergence, i.e., 
\begin{equation} \label{eq:A-div}
	\bfh^{ab} \nb_{a} A_{b} = 0
\end{equation}
Such a frame always exists in the small critical data setting by the classical argument of Uhlenbeck \cite{Uhl}. 
By the div-curl system \eqref{eq:A-curl}, \eqref{eq:A-div}, it follows that $A$ can be recovered from $\psi_{\alp}$ by solving the nonlinear elliptic equations
\begin{equation} \label{eq:A-lap}
	\bfh^{bc} \nb_{b} \nb_{c} A_{\alp} - \bfh^{bc} [\nb_{b}, \nb_{\alp}] A_{c} = \bfh^{bc} \nb_{b} \bb( \bfR(u) (\psi_{c}, \psi_{\alp}) - [A_{c}, A_{\alp}] \bb)
\end{equation}
On the other hand, from the div-curl system \eqref{eq:psi-div-curl}, we may derive the nonlinear wave equation
\begin{equation} \label{eq:psi-wave}
\begin{aligned}
	\etab^{\bt \gmm} \nb_{\bt} \nb_{\gmm} \psi_{\alp} - \etab^{\bt \gmm} [\nb_{\bt}, \nb_{\alp}] \psi_{\gmm} = &- 2 \etab^{\bt \gmm} A_{\bt} \nb_{\gmm} \psi_{\alp} - (\etab^{\bt \gmm} \nb_{\bt} A_{\gmm}) \psi_{\alp} \\
	& - \etab^{\bt \gmm} A_{\bt} A_{\gmm} \psi_{\alp} + \etab^{\bt \gmm} \bfR(u) (\psi_{\bt}, \psi_{\alp}) \psi_{\gmm} 
\end{aligned}
\end{equation}
In the flat case $\Sgm = \bbR^{d}$, the left-hand sides of \eqref{eq:A-lap} and \eqref{eq:psi-wave} reduce to well-understood linear operators, namely the (componentwise) scalar Laplacian and d'Alembertian, respectively. As a consequence, when $d \geq 4$ the coupled system \eqref{eq:A-lap} and \eqref{eq:psi-wave} can be used to establish appropriate critical regularity a priori bounds for the smooth wave map $u$ with the given smooth initial data, via a direct perturbative analysis based on $L^{p}$ estimates for $A_{\alp}, \nb A_{\alp}$ and global Strichartz estimates for $\psi_{\alp}$\footnote{We remark that a similar scheme also works in dimensions $d =2, 3$, but the proof is far more complex as global Strichartz estimates are insufficient; it is now necessary to reveal the null structure of the wave maps equation in the Coulomb gauge, and utilize multilinear null form estimates. See \cite{Kri2, Kri04} for details.}. Global existence and scattering of $u$ can then be inferred from these a priori bounds, which completes the proof of Theorem~\ref{t:main} in the case $\Sgm = \bbR^{d}$ and $d \geq 4$. We refer to \cite{SS02, NSU} for more details.

\subsubsection*{Issue of tensoriality on a curved background}
This type of analysis, however, does not immediately extend to curved $\Sgm$ in general. The main issue is \emph{tensoriality} of $\psi_{\alp}$ and $A_{\alp}$, which is inconspicuous in the flat case but becomes pronounced on a curved background. First of all, observe that the left-hand side of \eqref{eq:psi-wave}, which is the Hodge-de Rham d'Alembertian on 1-forms, does not coincide with the (componentwise) scalar d'Alembertian on $\bbR \times \Sgm$ in general. In particular, unlike the scalar d'Alembertian, no useful dispersive estimates, such as \emph{global} Strichartz estimates, seem to be known for this linear operator in our setting $\Sgm = \bbH^{d}$. To make matters worse, the left-hand side of~\eqref{eq:A-lap}, which is the Hodge-de Rham Laplacian on 1-forms, is also distinct from the scalar Laplacian in general. Even in the special case $\Sgm = \bbH^{d}$, much less is known about $L^{p}$ estimates for the Hodge-de Rham Laplacian compared to the scalar Laplacian; e.g., see \cite{Br05}.

Nevertheless, there is a reason to believe that this obstruction is not essential, on the ground that the wave map itself is \emph{not} tensorial, i.e., its transformation laws under coordinate changes of the domain are those of a scalar. Indeed, we will show in this paper that a geometrically natural realization of $u$ as a linear \emph{scalar} field with (in fact, more) favorable nonlinearity can be obtained by a different choice of gauge, namely the \emph{caloric gauge} introduced by Tao \cite{Tao04, Tao37}.

\subsubsection*{Idea of the caloric gauge}
The main component, as well as the namesake, of the caloric gauge is the \emph{harmonic map heat flow}, a celebrated geometric PDE in its own right. In simplest terms, the idea of the caloric gauge can be phrased as follows. Given a smooth wave map $u(t, x)$ defined on a time interval $I$, we introduce a new time parameter $s \in \bbR^{+}$ and solve the harmonic map heat flow in $s$ with initial data $u(t, x)$. Since the initial data $u(t, x)$ is suitably small,  the heat flow can be shown to converge to a fixed point $u_{\infty}$ on $\NN$ as $s \to \infty$.  This now gives a canonical choice of frame at $s = \infty$, namely the \emph{same set of orthonormal vectors} $e( \infty, t, x) = e_{\infty}$  for each point $(t, x) \in  I \times \Sigma$. The \emph{caloric gauge} determined by the choice $e_{\infty}$ is then constructed  by parallel transporting the frame $e_\infty = e( \infty, t, x)$  from $s= \infty$  back along the harmonic map flow. 


The reason why the caloric gauge is so effective can be understood clearly by making an analogy between the harmonic map heat flow resolution and the Littlewood-Paley theory. One way to define Littlewood-Paley theory on a general complete Riemannian manifold $\Sgm$ is to introduce a new heat-time parameter $s \in \bbR^{+}$ and to use the linear heat equation $(\rd_{s} - \Dlt_{\Sgm}) f = 0$; see, for instance, \cite{St70Book, KR06}. The main ideas can be sketched as follows. Given a function $g$ on $\Sgm$ and the solution $f$ to the linear heat equation with $f \rest_{s = 0} = g$, we interpret $f(s)$ as the projection of $g$ to frequencies less than $s^{-1/2}$, which may be justified in the case $\Sgm = \bbR^{d}$ by observing that the Fourier transform of $f(s)$ is simply that of $g$ multiplied by a rescaled Gaussian adapted to the ball $\set{\abs{\xi} \simeq s^{-1/2}}$. A Littlewood-Paley projection to frequencies comparable to $s^{-1/2}$ is defined to be $f^{(k)}(s) := s^{k} \rd_{s}^{k} f(s) = s^{k} \Dlt^{k} f(s)$ for a fixed $k \geq 1$, the idea being that $s^{k} \Dlt^{k}$ damps the frequencies much lower than $s^{-1/2}$. Using the fact that $f \to 0$ as $s \to \infty$ and performing Taylor expansion at $s = \infty$, one obtains the reconstruction formula $g = \int_{0}^{\infty} \frac{(-1)^{k}}{(k-1)!} f^{(k)} (s) \, \ds$. It can be readily seen that the main results of Littlewood-Paley theory (e.g., the square function theorem) extends to this heat flow definition.

In analogy with this formulation of Littlewood-Paley theory, we may interpret the harmonic map heat flow resolution of a map $u$ as a geometric Littlewood-Paley decomposition of $u$. The analogue of the Littlewood-Paley projection $f^{(1)}(s)$ is $s \rd_{s} u(s)$, which is a section of $u(s)^{\ast} T \NN$. Given a frame $e(s)$ of this bundle, $s \rd_{s} u(s)$ is expressed as an $\bbR^{n}$-valued function $s \psi_{s}(s)$, which we call the \emph{heat-tension field}. According to this interpretation, it is natural to view the `geometric Littlewood-Paley projection' $\psi_{s}$ as the main dynamic variable. More precisely, in order to establish an appropriate a priori estimate at the critical regularity for $u$ (as in the Coulomb gauge case), we will analyze the dynamic equation obeyed by $\psi_{s}$. To make this scheme work, we need an analogue of the reconstruction formula, i.e., a means to reconstruct the harmonic map $u$ and other variables of the derivative formulation such as $\psi_{\alp}$ and $A_{\alp}$ in terms of $\psi_{s}$.

This is exactly where Tao's caloric gauge condition enters. As alluded to earlier, the caloric gauge condition corresponds to taking $e (\infty) = e_{\infty}$ for a fixed constant frame $e_{\infty}$ and parallel transporting this frame, i.e., imposing $D_{s} e(s) = 0$ for $0 < s < \infty$. As a consequence, it follows that
\begin{equation} \label{eq:cal-int}
\begin{gathered}
	\rd_{s} \psi_{\alp} = \bfD_{\alp} \psi_{s}, \quad
	\rd_{s} A_{\alp} = \bfR(u(s)) (\psi_{s}, \psi_{\alp}) \quad \hbox{ for } 0 < s < \infty \\
	\psi_{\alp} \to 0, \quad A_{\alp} \to  0 \quad \hbox{ as } s \to \infty
\end{gathered}
\end{equation}
Therefore, $\psi_{\alp}$ and $A_{\alp}$ can be recovered from $\psi_{s}$ (under mild bootstrap assumptions) by simply integrating to $s$ from $\infty$. 

It is now clear how the caloric gauge resolves the difficulties faced by the Coulomb gauge. The main dynamic variable is now $\psi_{s}$ instead of $\psi_{\alp}$, which is evidently scalar. Indeed, $\psi_{s}$ obeys a nonlinear wave equation\footnote{We remark that the nonlinearity of this equation involves $\psi_{s}$ at different $s$, resembling the equation satisfied by standard Littlewood-Paley projection of $u$ in, say, the extrinsic formulation \eqref{wme-2}. The difference, of course, is that the nonlinearity for $\psi_{s}$ enjoys more favorable structure.} \eqref{eq:wmp} with the scalar d'Alembertian as the principal part, for which global Strichartz estimates are known \cite{MT12, AP14}.  From $\psi_{s}$, the other variables $\psi_{\alp}$ and $A_{\alp}$ are recovered by a simple integration from $s = \infty$. In effect, elliptic theory for the tensor $A_{\alp}$ in the Coulomb gauge is replaced by well-posedness theory for the harmonic map heat flow, which only relies on parabolic theory for the scalar heat equation. 

\subsubsection*{Further discussion on the caloric gauge}
So far we have emphasized the \emph{scalar} nature of the main dynamic variable $\psi_{s}$ as a key advantage of the caloric gauge over the Coulomb gauge, whose effectiveness becomes evident on a curved background such as ours. 

Another important feature of the caloric gauge, applicable to any background, is that it is well-adapted to the large data setting. The construction of the caloric gauge relies on the harmonic map heat flow, whose large data theory is well-understood in settings where one may hope to prove a large data result for the wave map. For further discussion, see Smith \cite{Sm}, where the caloric gauge was constructed for any time-dependent map on $\bbR \times \bbR^{2}$ with sub-ground state energy. We also note the work \cite{Oh14, Oh15} of the second author on the extension of the caloric gauge to more general gauge theories by replacing the harmonic map heat flow by the Yang-Mills heat flow, where the key advantage over the Coulomb gauge seems to be applicability in the large data setting.

Finally, as already highlighted by Tao \cite{Tao04, Tao37}, perhaps the deepest advantage of the caloric gauge is that it leads to a more favorable nonlinearity than the Coulomb gauge.  More precisely, the worst interaction of $\psi_{\alp}$ in the equation for $A$ in the Coulomb gauge, namely the high-high to low interaction, does not occur in the caloric gauge when suitably interpreted using the Littlewood-Paley analogy above\footnote{More precisely, \eqref{eq:cal-int} implies that $A_{\alp}(s)$ can be recovered just from $s' \psi_{s}(s')$ at larger $s'$, which corresponds to lower frequency projections according to the Littlewood-Paley analogy discussed above.}. This point was used crucially in the works \cite{BIKT11, Sm, DS15} on Schr\"odinger maps, where such an extra cancellation appears to be essential. In our paper, this improved structure is reflected by the sufficiency of the usual Strichartz estimates to close the estimates; in particular, refined Strichartz estimates as in \cite{SS02, NSU} are not used.

\begin{rem}  \label{r:null}
Another idea that played a key role in the understanding the wave maps is that of null structure, introduced by Klainerman and Machedon \cite{KlaMac93} in this context, which refers to a special cancellation structure in the nonlinearity of the wave maps equation. We note that the null structure was also observed under symmetry in the work \cite{CTZcpam}. In this paper, the null structure of the wave maps does not play any role, since we work in high dimensions where the linear wave equation enjoys strong enough dispersion. Nevertheless, the null structure of the wave maps equation is expected to be essential in the case $\Sgm = \bbH^{2}$, which is the most interesting case.
\end{rem}

 \subsection{History and motivation} \label{s:motivate}

 The Cauchy problem for wave maps arises as a model problem in both general relativity, see e.g.,~\cite{W90}, and in particle physics, see e.g.,~\cite{MS, GML}, where it is called a nonlinear $\s$-model. From a purely mathematical point of view, the wave maps equation is perhaps the simplest  and most natural geometric wave equation  -- the nonlinearity arises from the geometry of the target manifold -- and is the hyperbolic analogue of the elliptic harmonic maps equation and the parabolic harmonic map heat flow equation.

\subsubsection*{History} 
At this point there is quite a lot known about the critical Cauchy problem for wave maps on $\R^{1+d}_{t, x}$, with a robust understanding in the very interesting energy critical dimension $d=2$. To start, 
 a detailed critical well-posedness theory was given in the classical works by 
 Christodoulou, Tahvildar-Zadeh~\cite{CTZduke, CTZcpam} for spherically symmetric wave maps and  by Shatah, Tahvildar-Zadeh \cite{STZ92, STZ94} for equivariant wave maps. Previously, Shatah~\cite{Shatah} proved the existence of self-similar blow up solutions for energy supercritical wave maps into the $3$-sphere. 

Outside of spherical symmetry the sharp subcritical local well-posedness theory was developed by Klainerman, Machedon~\cite{KlaMac93, KlaMac95, KlaMac97} and Klainerman, Selberg~\cite{KlaSel97, KlaSel02} by making essential use of the~\emph{null form} structure in the nonlinearity, see Remark~\ref{r:null}. 
 The very difficult critical small data problem was partially settled by Tataru~\cite{Tat01}, who proved global well-posedness for data with small critical Besov norm  $\dot B^{\frac{d}{2}, 2}_1 \times \dot B^{\frac{d}{2} -1, 2}_1$ , and then completed by Tao~\cite{Tao1, Tao2} for initial data that is small in $\dot{H}^{\frac{d}{2}} \times \dot{H}^{\frac{d}{2}-1}$ for targets $\NN = \Sp^{N-1}$.  The seminal Tataru, Tao theory of small data wave maps was then extended to more general targets by Krieger~\cite{Kri04, Kri2}, Klainerman, Rodnianski~\cite{KR01}, Tataru~\cite{Tat05},  Shatah, Struwe~\cite{SS02}, and Nahmod, Stefanov, and Uhlenbeck~\cite{NSU}; see Section~\ref{s:main-idea} for more on these works. 

The bubbling analysis of Struwe~\cite{Struwe} linked singularity formation for energy critical equivariant wave maps to the presence of nontrivial finite energy harmonic maps. This made apparent the decisive role that the geometry of the target plays in the energy critical problem. Later, explicit blow up constructions demonstrating the bubbling of a harmonic map were achieved in breakthrough works of Krieger, Schlag, Tataru~\cite{KST}, Rapha\"el, Rodnianski~\cite{RR}, and Rodnianski, Sterbenz~\cite{RS} in the positively curved case $\NN = \bbS^{2}$. Since there are no nontrivial finite energy harmonic maps from  $\R^2$ into negatively curved manifolds it was long conjectured that all smooth, decaying finite energy data lead to a global and scattering wave map evolution for maps $\R^{1+2}_{t, x} \to \NN$ in the case that all the sectional curvatures of $\NN$ are negative. This theorem was proved in 2009 in three remarkable and independent works by Krieger, Schlag~\cite{KS}, Sterbenz, Tataru~\cite{ST1, ST2}, and Tao~\cite{Tao37}. Note that Krieger and Schlag ~\cite{KS} also developed a `twisted' nonlinear profile decomposition in the spirit of the work of Bahouri and G\'erard~\cite{BG}. We also remark that in~\cite{ST1, ST2} Sterbenz and Tataru proved the more general Threshold Theorem, which states that any smooth data with energy less than that of the minimal energy nontrivial harmonic map lead to a globally regular and scattering solution. For a further refinement of the threshold theorem, see \cite{LO1}, and for more detailed study of singularity formation for equivariant wave maps, see \cite{CKLS1, CKLS2, Cote13, JK, BKT}. 

There are fewer works concerning wave maps on curved domains. Shatah, Tahvildar-Zadeh  showed the existence  and orbital stability of equivariant time-periodic wave maps, $\R \times \Sp^2 \to \Sp^2$. This was extended by the third author \cite{Shah2} to allow for maps $\R \times \Sigma \to  \Sp^2$ where $\Sigma$ is diffeomorphic to $\Sp^2$ and admits an $SO(2)$ action. The first author established a critical small data global theory for wave maps on  small asymptotically flat perturbations of $\R^4$~\cite{L}  using the linear estimates of Metcalfe, Tataru~\cite{MT12}. Also, we note the recent work of D'Ancona, Zhang~\cite{DZ} who proved global existence for small equivariant wave maps on rotationally symmetric spacetimes for $d \ge 3$.

\subsubsection*{Motivation} 

There are several reasons to consider the wave maps equation on a curved spacetime $\MM$ as in~\eqref{wme}. For one, note that such a wave map arises as a nonlinear model for Einstein's equations. To be more specific, the Ernst potential of a Killing vector field $X$ on a $1+3$ dimensional  Einstein-vacuum manifold $\MM$ can be viewed as a wave map $\Phi: \ti \MM \to \Hp^2$, where $\ti \MM = \MM/ X$ is now a $1+2$ dimensional spacetime; see for example~\cite{W90} and the recent work of Ionescu and Klainerman~\cite{IK}, who addressed the stability of the Kerr Ernst potential.   We also refer the reader to the recent work of Andersson, Gudapati, and Szeftel~\cite{AGS}, who considered the wave map-Einstein system in an equivariant setting, and also to the work of Luk~\cite{Luk13} who proved a global existence result for (3+1)-dimensional semilinear equations with null form structure on a Kerr spacetime.


Our main motivation for introducing curvature to the domain is that we view the treatment of wave maps on curved backgrounds as a natural extension of the detailed theory on flat spacetime. A natural starting point is the following  question: what effects does the geometry of the domain have on the global dynamics of the wave maps problem, especially for large data? In other words,  are there new \emph{nonlinear} phenomena that arise in the curved setting? Surely, one can affect the dynamics of the wave map equation by  say, studying the problem on a compact domain or on manifolds $\MM$ that allow for the trapping of bicharacteristic rays.  However, many of the most pressing  issues that arise in such settings are already present at the level of the scalar linear wave equation on $\MM$. There is a large body of literature devoted to the long time dispersive behavior of linear waves on curved spacetimes, and many interesting open questions remain; see for example~\cite{MT12} for one interesting direction. 

To narrow the focus on the interplay between the curved geometry of the domain and the nonlinear structure of the wave maps equation, one can consider the constant sectional curvature model $\Sgm = \Hp^d$ as a natural starting place.  The global dispersive theory of linear waves $\Box_{\Hp^d} v = 0$ is now well understood, and in fact dispersion is stronger than for free waves on $\R^{1+d}$ due to the exponential volume growth of balls in $\Hp^d$; see~\cite{MT11, AP14, MTay12, B07, IS} and~\cite{BCS, BCD, BD, IPS} for nonlinear applications. 

To view the small data result in Theorem~\ref{t:main} (or more precisely, Theorem~\ref{t:main1}) and the techniques developed in this paper in this broader context, let us first describe our recent work on the energy critical wave maps equation on hyperbolic space, and the interesting nonlinear dynamics that arise in $2$-dimensions. In~\cite{LOS1, LOS2, LOS3} we considered \emph{equivariant} wave maps from $\R \times \Hp^2$ into two model targets, the $2d$ hyperbolic space, $\NN = \Hp^2$ and the $2$-sphere, $\NN = \Sp^2$. In these settings the $1$-equivariant wave map equation can be expressed entirely in terms of the polar coordinate $\psi$ on the target, that is,  $ u(t, r, \theta) = ( \psi(t, r), \theta)$ and~\eqref{wmi} reduces to
\EQ{ \label{eq:ewm} 
\psi_{tt} - \frac{1}{ \sinh r} \partial_r( \sinh r \partial_r \psi) + \frac{g(\psi) g'(\psi)}{ \sinh^2 r} = 0
}
Here the target is endowed with the metric $ds^2 =  d \psi^2 + g^2(\psi) d \theta^2$ where $g(\psi) = \sinh \psi$ in the case $\NN = \Hp^2$ and $g(\psi) =  \sin \psi$ if $\NN =  \Sp^2$. 
 
By combining two well-known geometric facts -- the conformal invariance of $2d$ harmonic maps and their energy, and the conformal equivalence between $\Hp^2$ and the unit disc $\D^2$ -- one finds explicit continuous $1$-parameter families of finite energy harmonic maps for both targets $\NN = \Hp^2$ and $\NN =  \Sp^2$. 

In the case of the target $\NN = \Hp^2$ the family of finite energy harmonic maps and their energies $\E(P_{\la}, 0)$ are 
\EQ{\label{eq:pla} 
 P_{\la}(r):= 2 \arctanh( \la \arctanh(r/2)),\,   \la \in [0, 1), \, \,  \E(P_\la, 0) =  4 \pi  \frac{\la^2}{1- \la^2}.
}
The very existence of these maps is in stark contrast to the corresponding Euclidean problem; there are no finite energy harmonic maps $\bbR^{2} \to \bbH^{2}$ due to the well-known Eells-Sampson theory~\cite{ES}. In~\cite{LOS1} we proved that every $P_\la$ is asymptotically stable under small finite energy equivariant perturbations. Moreover, there can be no finite time blow up for wave maps $\R \times \Hp^2 \to \Hp^2$ since otherwise, by the same Struwe bubbling analysis as above,  there would have to exist a nontrivial finite energy Euclidean harmonic map into $\Hp^2$, yielding a contradiction. 
Next, note that the harmonic map $P_{\la}(r)$ satisfies 
\ant{
P_\la(r)  \to  2 \arctanh( \la)  \mas r \to \infty.
}
In fact, the space of equivariant finite energy data  naturally splits into disjoint classes $\E_\la$ for $\la \in [0, 1)$ where each $(\psi_0, \psi_1) \in \E_\la$ is characterized by $\psi_0(r) \to  2 \arctanh( \la)$ as $r \to \infty$. Moreover, the classes $\E_\la$ are preserved by the smooth wave map flow and $(P_\la, 0)$ uniquely minimizes the energy in $\E_\la$; see~\cite{LOS1, LOS3}.  One is led to the following conjecture regarding the asymptotic behavior of solutions in the equivariant setting.  
\begin{conj}[Soliton resolution for equivariant wave maps $\R \times \Hp^2 \to \Hp^2$] \label{c:ESR}  
Consider the Cauchy problem~\eqref{eq:ewm} with finite energy initial data $(\psi_{0}, \psi_{1})$. Let 
\ant{
\la :=\tanh (\psi_{0}(\infty)/2) \in [0, 1).
} 
Then~\eqref{eq:ewm} is globally well-posed, and the solution scatters to $P_{\la}$ as $t \to \pm \infty$.
\end{conj}

In~\cite{LOS3} the conjecture was verified for all data with endpoints $\la \le \La_0$ where $\La_0 \ge 0.57$, leaving open the range of endpoints $\la \in [\La_0, 1)$. 

More ambitiously, one can formulate the soliton resolution conjecture in the nonequivariant setting. 
In this case, using conformal equivalence of $\bbH^{2}$ and $\bbD^{2}$, the space of finite energy maps splits into disjoint classes $\calE_{\rd u}$ parametrized by boundary data $\rd u$ at spatial infinity that are consistent with finite energy. These classes are preserved by the wave map flow, and the minimizer $\calP_{\rd u}$ of the Dirichlet energy in each class is a harmonic map. The soliton resolution conjecture may then be phrased in terms $\calP_{\rd u}$ as in Conjecture~\ref{c:ESR}.  

An interesting milestone in this direction would be the global asymptotic stability of $\calP_\la(r, \te):=( P_\la(r),  \te)$ under arbitrarily large non-equivariant perturbations, where $\calP_{\la}$ has the advantage of being explicit. 
\begin{conj} \label{c:SRP} 
Let $(u_0, u_1)$ be smooth finite energy initial data for~\eqref{wme} for $\MM = \R \times \Hp^2$ and $\NN = \Hp^2$. Let $\la \in [0, 1)$ be fixed  and assume that outside some compact set $K \Subset \Hp^2$ we have 
\EQ{ 
u_0(x)  
= \calP_\la(x)  \mfor   x  \in  \Hp^2- K,  \label{2dbc} 
}
Then the unique wave map evolution $(u, \p _t u)(t)$ associated to the data $(u_0, u_1)$ is globally regular  and scatters to $(\calP_\la, 0)$ as $t \to \pm \infty$. 
\end{conj} 


 Aside from establishing the small data scattering theory in high dimensions, this paper can be viewed as a first step towards the non-equivariant theory for $\Sigma = \Hp^2$, and in particular Conjecture~\ref{c:SRP}. 

 

\begin{rem} 
There are even more challenges  and unexpected phenomena in the other case considered in~\cite{LOS1}, namely when $\NN = \Sp^2$.  The equivariant harmonic maps from $\Hp^2 \to \Sp^2$ and their energies $\E(Q_{\lmb}, 0)$ are given 
 by 
\EQ{ \label{eq:qla} 
Q_{\la}(r):= 2 \arctan( \la \arctanh(r/2)),  \, \la \in [0, \infty), \,  \E(Q_\la, 0) =  4 \pi  \la^2(1+ \la^2)^{-1}.
} 
The stability of $Q_{\la}$ is investigated in~\cite{LOS1}.  We showed that there exists $\la_0>1$ so that each $Q_{\la}$ for $\la < \la_0$ is asymptotically stable under small finite energy equivariant perturbations. However there is more interesting behavior for $\la \gg 1$. In particular,  for $\la$ large enough the  Schr\"odinger operatror $\LL_\la$ obtained by linearizing about $Q_\la$  has an eigenvalue in its spectral gap $(0, 1/4)$. This gap eigenvalue leads to a solution to the linearized equation that does not decay as $t \to \pm \infty$. On the other hand, one can show that any small finite energy perturbation of $Q_\la$ must lead to a globally regular nonlinear evolution. This is a consequence of the fact that any finite time blow up in this energy critical, equivariant setting must occur by energy concentration at $r = 0$ and thus the global geometry of the domain is irrelevant. 
In fact blow up can be characterized by the bubbling  of a \emph{Euclidean} harmonic map exactly as in~\cite{Struwe}, and this requires more energy than a sufficiently small perturbation of $Q_\la$;  see~\cite{LOS3}. These  observations suggest that in fact all of the $Q_\la$ are asymptotically stable in the  energy space, with the stability of the $Q_\la$ for large $\lambda$ occurring via a completely nonlinear mechanism, such as the \emph{radiative damping} described in the work of Soffer, Weinstein~\cite{SW99}. 
\end{rem}



%

\section{Geometric  and analytic preliminaries } \label{s:prelim}

\subsection{Notation} We will try to maintain the following conventions for indices: We reserve greek lowercase indices $\al, \be, \ga,$ \dots $ \in\{ 0, 1, 2, 3, 4\}$ for coordinates on $\MM = \R \times \Hp^4$. We use  Latin indices $a, b, c,$  \dots $ \in \{1, 2, 3, 4\}$ when restricting to coordinates on $\Hp^4$. 
These indices are raised and lowered using the appropriate domain metrics introduced below.
Sometimes we will use the bold letters $\mathbf{a}, \mathbf{b}, \ldots$ for coordinates $(s, t, x^{1}, \ldots, x^{4})$ on $\bbR^{+} \times \MM$.
Next, we will use lowercase Latin indices beginning at $i$, namely $i, j,k, \ell, m$ for components in the tangent space of $\NN$; hence they run from $1, \dots, n$ where $n$ is the dimension of $\NN$. The capital letters $I, J, K, L$ will be used for indices in $\bbR^{N}$, which arises as the ambient space to which $\NN$ is embedded. Finally, throughout the paper we adopt the standard convention of silently summing over repeated upper and lower indices.



 \subsection{Geometry of the domain} Let $\R^{d+1}$ denote the $(d+1)$-dimensional Minkowski space  with rectilinear coordinates  $\{y^1,\dots, , y^d, y^0\}$ and metric $\m$ given in these coordinates by $\m = \textrm{diag}(1, \dots, 1,  -1)$.  
One can then define the $d$-dimensional hyperbolic space $\Hp^d$ as
\begin{align*}
\Hp^d := \{ y \in \R^{d+1} \mid  = 1, (y^{0})^{2} - \sum_{a=1}^{d} (y^{a})^{2} = 1\, \, y^0>0\}
\end{align*}
The Riemannian structure on $\Hp^d$ is obtained by pulling back the metric  $\m$ on $\R^{d+1}$ via the inclusion map $ \iota: \Hp^d \hookrightarrow \R^{d+1}$. In particular the Riemannian metric $\h$ on $\Hp^d$ is given by 
$
\h =  \iota^* \m. 
$

We note that $\Hp^d$ admits global coordinates systems. Given a global coordinate system $\{x^a\}$ on $\Hp^d$ denote by $\h_{ab}$ the components of $\h$ and $\h^{ab}$ the components of  the inverse matrix $\h^{-1}$, i.e., 
\ant{
\h_{a b} :=  \h\left( \frac{\p}{\p x^a}, \frac{\p}{\p x^b} \right), \quad \h^{ab}:= \h^{-1}(dx^a, dx^b)
}
The metric $\etab = (\etab_{\al \be})$ on the Lorentzian manifold $\MM = \R \times \Hp^d$ is then given in coordinates by $\etab_{00} = -1$, $\etab_{0 a} = 0$ and $\etab_{a b} = \h_{ab}$. We denote the Christoffel symbols on $\R \times \Hp^d$ by 
\ant{
\Ga_{ \al \be}^{\ga}  = \frac{1}{2} \etab^{ \ga \mu}( \p_{\al} \etab_{\mu \be} + \p_\be \etab_{\al \mu} - \p_\mu \etab_{\al \be})
} 
Given a vector field $X = X^\al \frac{\p}{\p x^\al}$ in $\Ga(T\MM)$  or a $1$-form $\om = \om_\al dx^\al \in \Ga( T^*\MM)$ the metric covariant derivative $\na$  is defined  in coordinates by 
\ant{
\na_\al X = (\p_{\al} X^\be + \Ga^{\be}_{\al \ga} X^\ga ) \p_\be
, \quad \na_\al \om  = (\p_\al \om_\be  - \Gamma^\ga_{\al \be} \om_\ga ) dx^\be
}
From these definitions one can generalize the definition of   $\na$ to $(p,q)$-tensor fields of arbitrary rank by requiring that: 
$\nb_{\al} f = \rd_{\alp} f$ for $(0, 0)$-tensors (i.e., functions), $\nb_{\al}$ agrees with the preceding definition for $(1, 0)$- or $(0, 1)$-tensors (i.e., vector fields or 1-forms) and $\nb$ obeys the Leibniz rule with respect to tensor products.
We note that $\na_0 = \p_t$ (i.e., $\Gmm_{0 \bt}^{\alp} = 0$) and that $\na_a$ for a Latin index $a$ will denote a covariant derivative over $\Hp^d$. Next, we will denote by $\RR$ the Riemann curvature tensor on $\Hp^d$. Recall that for a vector field $\xi^a\frac{\p}{\p x^a}$ in $\Gmm(T\Hp^d)$ we have the formula
\EQ{
[ \na_a,  \na_b] \xi^c   = \RR^c_{ \, \, e a b} \xi^e
}
as well as 
\EQ{
 [ \na_a,  \na_b] \xi^a   = \RR^a_{ \, \, c a b} \xi^c = \RR_{c b} \xi^c
}
where $\RR_{c b}$ denote the components of the Ricci curvature tensor. Since $\Hp^d$ has constant sectional curvature $ \mathcal{K} = -1$ we note that 
\EQ{\label{Ricci formula}
\RR_{a b} = -(d-1) \h_{a b}
}
Finally, we observe that the Riemann curvature tensor on $\MM = \bbR \times \bbH^{d}$ vanishes when contracted with $\rd_{t}$ in any slot, and it agrees with $\tensor{\calR}{^{a}_{e}_{ab}}$ when contracted with vectors tangent to $\bbH^{d}$. For this reason, we will slightly abuse notation by referring to the Riemann curvature tensor on $\MM$ also as $\calR$.

In Section~\ref{s:caloric} and beyond, we will work exclusively in the case $d = 4$.


\subsection{Geometry of maps  $u: \MM \to \NN$}  \label{s:maps} 

As we saw in the introduction the wave maps equation~\eqref{wmi} is defined using the geometric structure of the pullback vector bundle $u^* T \NN$ for  maps $u:  \MM \to \NN$. In this subsection, we develop the differential-geometric formalism for analysis on the bundle $u^{*} T \NN$.

Let $\MM$ be a smooth manifold equipped with a torsion-free connection $\nb$ on the tensor bundle, and assume furthermore that $\MM$ is contractible. A specific example to have in mind is of course $\MM = \bbR \times \bbH^{4}$ with $\nb$ the Levi-Civita connection. Such a general formulation will be useful later in Section~\ref{s:caloric}, where we consider the harmonic map heat flow and construct the caloric gauge. 

Let $u$ be a smooth map from $\MM$ to $\NN$, where $\NN$ obeys the hypothesis of Theorem~\ref{t:main}. The pullback bundle $u^{\star} T \NN$ is the vector bundle over $\MM$ whose fiber at $p \in \MM$ is the tangent space $T_{u(p)} \NN$. Smooth sections of $u^* T\NN$ are maps $\phi : \MM  \to  T \NN$ so that $\phi(p) \in T_{u(p)} \NN$ for each $p \in \MM$. The Riemannian connection on $T\NN$ induces the pullback covariant derivative $D$ on $u^* T \NN$. In local coordinates $\{x^{\al}\}_{\al  = 0}^d$ on $\MM$ and  $ \{u^j\}_{j = 1}^n$ on $\NN$ we have 
\EQ{
D_{\al} \phi =  \big(\p_{\al}  \phi^j + {}^{(\NN)}\Gamma^{j}_{ik}(u)   \phi^i \p_{\al} u^k \big) \frac{\p}{\p u^j} 
}
where ${}^{(\NN)}\Gamma^{j}_{ik}$ are the Christoffel symbols on $T \NN$. The pullback covariant derivative $D_{\alp}$ is naturally extended to $u^{\ast} T \NN$-valued tensors (i.e., sections of $T^{\otimes p} \MM \otimes (T^{\ast})^{\otimes q} \MM \otimes u^{\ast} T \NN$) using the connection $\nb_{\alp}$ on the tensor bundle of $\MM$.

We will study the wave maps system by expressing the equation satisfied by the differential $du$ in terms of a global moving frame (or in short, simply a \emph{frame}) on $u^{*}T \NN$. This refers to a collection $e = (e_{1}, \ldots, e_{n})$ of global (i.e., defined on all of $\MM$) sections of $u^{\ast} T \NN$ that form an orthonormal basis of $T_{u(p)} \MM$ on each $p \in \MM$. By the assumption that the base manifold $\MM$ is contractible, at least one such a frame exists. Expressing each fiber $T_{u(p)} \NN$ in terms of the orthonormal basis $(e_{1}, \ldots, e_{n})(p)$, we see that $e$ defines a global trivialization of $u^{\ast} T \NN$, i.e.,
\begin{equation*}
	e : E_{n} \to T \NN, \quad v = (p; v^{1}, \ldots, v^{n}) \mapsto e v = e_{k}(p) v^{k}
\end{equation*}
where $E_{n}$ refers to the trivial bundle $E_{n} = \MM \times \bbR^{n}$. Moreover, $e$ induces an isomorphism between the spaces of sections
\begin{equation*}
	e : \Gmm(E_{n}) \to \Gmm(u^{\ast} T \NN), \quad \varphi = (\varphi^{1}, \ldots, \varphi^{n}) \mapsto e \varphi = e_{k} \varphi^{k}
\end{equation*}
Note that $\Gmm(E_{n})$ is simply the space of $\bbR^{n}$-valued functions on $\MM$. This isomorphism extends naturally to $u^{\ast} T \NN$-valued tensors; we make a slight abuse of notation and refer to all these isomorphisms by $e$. 

Given a map $B: \MM \to SO(n)$ and a frame $e$, we may obtain a new frame $e'$ by setting
\begin{equation} \label{eq:cf}
	e \mapsto e B =: e', \quad \hbox{ or } e'_{j} := e_{i} B^{i}_{j}
\end{equation}
Conversely, observe that any two frames $e, e'$ with the same orientation must be related to each other by $e' = e B$ for some $B : \MM \to SO(n)$. In other words, the wave maps system in the global moving frame formulation has an $SO(n)$ gauge structure.

Let $e$ be a frame. Associated to the map $u$ we have the differential $d u$, which is a $u^{\ast} T \NN$-valued 1-form. We denote by $\Psi(\rd_{\alp}) = \psi_{\alp}$ the description of $d u$ with respect to $e$, 
\begin{equation}
	d u = e \Psi, \quad \rd_{\alp} u = e \psi_{\alp} 
\end{equation}

The pullback covariant derivative $D_{\alp}$ admits a matrix description in the global trivialization $E_{n}$. We first introduce a connection $\bfD_{\alp}$ on $E_{n}$ induced by $D_{\alp}$, characterized by the identity
\begin{equation}
	D_{\alp} e \varphi = e \bfD_{\alp} \varphi
\end{equation}
Using the connection $\nb$ on $E_{n}$ (defined component-wisely) as a reference, we may express $\bfD$ in the form
\begin{equation} \label{eq:conn-1-form}
	\bfD = \nb + A
\end{equation}
where $A$ is an $\mathfrak{so}(n)$-valued (i.e., anti-symmetric traceless matrix) 1-form on $\MM$. The latter fact is evident from the following alternative characterization of $A$:
\begin{equation*}
	A = e^{-1} D e, \quad \hbox{ or } A^{i}_{j, \alp} = \brk{D_{\alp} e_{j}, e_{i}}
\end{equation*}
We call $A$ the \emph{connection 1-form} associated to $\bfD$. Going back to $\ud u = e \Psi$, by the torsion-free property inherited from both the domain and target connections,  note that
\begin{equation}
	\bfD_{\alp} \psi_{\bt} = \bfD_{\bt} \psi_{\alp}
\end{equation}

If $u$ is a wave map and $e \psi_{\alp}$ is the representation of $\rd_{\alp} u$ in the frame $e$, the wave maps equation can be rewritten succinctly as
\EQ{ \label{eq:wm-psi-al}
\etab^{\al \be}(\Db_\al \psi)^j_\be = 0 
}
or in coordinate free notation 
\EQ{
\tr_{\etab}\Db \Psi = 0
}

Next, we introduce the \emph{curvature 2-form} associated to $\bfD$, which is defined to be the commutator 
\begin{equation} \label{eq:curv-2-form}
	F_{\alp \bt} \varphi = (\bfD_{\alp} \bfD_{\bt} - \bfD_{\bt} \bfD_{\alp}) \varphi
\end{equation}
Note that $F$ is an $\mathfrak{so}(n)$-valued 2-form on $\MM$. Using \eqref{eq:conn-1-form}, $F$ can be expressed in terms of $A$ as
\begin{equation} 
F = d A + A \wedge A,   \quad  F_{\alp \bt} = \rd_{\alp} A_{\bt} - \rd_{\bt} A_{\alp} + [A_{\alp}, A_{\bt}]
\end{equation}
The curvature 2-form $F$ is related through $u$ to the Riemann curvature $R$ of the target manifold $\NN$ as follows:
\begin{equation} \label{eq:curv-R}
	e F_{\alp \bt} \varphi = R(u) (\rd_{\alp} u, \rd_{\bt} u ) e \varphi = R(u)(e \psi_{\alp}, e \psi_{\bt}) e \varphi
\end{equation}

In the sequel we will have to perform  changes of frame as in~\eqref{eq:cf} and we record here the effect of a change of basis on the relevant objects, $F, A, \psi$.  For $B: \MM \to SO(n) $ we have the \emph{gauge symmetry} relative to the map $u: \MM \to \NN$: 
\EQ{ \label{eq:gauge-symm}
u \mapsto u;  \quad  e \mapsto   eB;  \quad  \psi_\al \mapsto B^{-1} \psi_\al; \\
F_{\al \be} \mapsto B^{-1} F_{\al \be} B; \quad A_\al \mapsto B^{-1} \p_\al B + B^{-1}A_\al B
}


Finally, we make a simple remark on the computation of the commutator $[\bfD_{\alp}, \bfD_{\bt}]$. For a section $\varphi \in \Gmm(E_{n})$, this commutator is precisely the matrix-valued 2-form $F_{\alp \bt}$. For a $\bbR^{n}$-valued tensor $T = \tensor{T}{^{\mu_{1} \cdots \mu_{p}}_{\nu_{1} \cdots \nu_{q}}}$, the domain curvature tensor $\calR$ associated to $\nb$ arises, i.e.,
\begin{equation}  \label{eq:commDD}
\begin{aligned}
	{[\bfD_{\alp}, \bfD_{\bt}]} \tensor{T}{^{\mu_{1} \cdots \mu_{p}}_{\nu_{1} \cdots \nu_{q}}}
	= & F_{\alp \bt} \tensor{T}{^{\mu_{1} \cdots \mu_{p}}_{\nu_{1} \cdots \nu_{q}}} \\
	& + \tensor{\calR}{^{\mu_{1}}_{\gmm \alp \bt}} \tensor{T}{^{\gmm \cdots \mu_{p}}_{\nu_{1} \cdots \nu_{q}}} + \cdots 
	+ \tensor{\calR}{_{\nu_{q}}^{\gmm}_{\alp \bt}} \tensor{T}{^{\mu_{1} \cdots \mu_{p}}_{\nu_{1} \cdots \gmm}} 
\end{aligned}
\end{equation}

\subsubsection{Extrinsic representation}
Using the isometric embedding $\NN \hookrightarrow \bbR^{N}$, the map $u$ admits a representation as an $\bbR^{N}$-valued map $u = (u^{K})_{K=1, \ldots N}$, which will be referred to as the \emph{extrinsic representation}. This point of view is convenient for analysis when the delicate gauge structure described in \eqref{eq:gauge-symm} need not be exploited. See, for instance, our analysis of the harmonic map heat flow in Section~\ref{s:caloric}. 

In this viewpoint, any section $\phi$ of $u^{\ast} T \NN$ admits a representation as an $\bbR^{N}$-valued function on $\MM$. Since $\NN \hookrightarrow \bbR^{N}$ is isometric, we have
\begin{equation} \label{eq:extr-phi}
	\bb( \sum_{K=1}^{N} \abs{\phi^{K}}^{2} \bb)^{1/2} = \abs{\phi}_{u^{\ast} T \NN} = \abs{\varphi}_{\bbR^{n}}
\end{equation}
where $e \varphi = \phi$. We will abuse the notation a bit and write $\abs{\phi}_{\bbR^{N}}$ as a shorthand for the left-hand side.
In particular, the differential $\ud u$ is represented by an $\bbR^{N}$-valued 1-form $\ud u = (\ud u^{K})_{K=1, \ldots, N}$ on $\MM$ and we have
\begin{equation} \label{eq:extr-du}
	\abs{\ud u(X)}_{\bbR^{N}} = \abs{\ud u(X)}_{u^{\ast} T \NN} = \abs{\Psi(X)}_{\bbR^{n}}
\end{equation}
for any vector $X$ tangent to $\MM$. The pullback covariant derivative $D$  on $u^*T\NN$ has the following natural interpretation: Given a section $\phi$ of $ u^* T \NN$, its covariant derivatives  $D_{\al}\phi$ over a point $x \in \MM$ are the standard coordinate derivatives $\p_\al \phi$  of $\phi$ projected orthogonally onto  $T_{u(x)} \NN$. This leads to the formula
\EQ{ \label{eq:extr-D}
D_\al \phi^K =  \p_{\al} \phi^K  + S^K_{I J}(u)  \phi^I \p_\al u^J 
} 
where $S$ is the second fundamental form. 

\subsection{Differentiation of the curvature 2-form} \label{s:curv}

Let $u : \MM \to \NN$ be a map and $e$ be a given frame. Let $F$ be the curvature 2-form of $\bfD$ defined above in \eqref{eq:curv-2-form}. 
We introduce the covariant derivative of $F$ by the formula
\begin{equation}
	\bfD_{\mu} F_{\alp \bt} := \nb_{\mu} F_{\alp \bt} + [A_{\mu}, F_{\alp \bt}],
\end{equation}
where $\nb_{\mu}$ is to be understood in the component-wise sense.
This definition is natural\footnote{Geometrically, $F$ is an $\mathrm{ad} \, P$-valued 2-form, where $\mathrm{ad} \, P$ is the vector bundle with fiber $\mathfrak{so}(n)$, structure group $SO(n)$ acting by the adjoint action $(g, a) \mapsto g a g^{-1}$, and the same structure coefficients as $u^{\ast} T \NN$. Then $\bfD$ defined as above is simply the induced covariant derivative on $\mathrm{ad} \, P$.} in view of the gauge symmetry of $F$ in \eqref{eq:gauge-symm}. Observe also that $\bfD F$ is characterized by the Leibniz rule
\begin{equation} \label{eq:d-curv-leibniz}
	\bfD (F \varphi) = (\bfD F) \varphi + F \bfD \varphi,
\end{equation}
for any $\bbR^{n}$-valued function $\varphi$. 

Our goal now is to give formulae for differentiating $F$. Recall that $F$ is related to the Riemann curvature $R$ on $\NN$ by the formula \eqref{eq:curv-R}. To have a description of $F$ entirely in terms of objects defined relative to the frame $e$, we introduce the pullback curvature $\bfR$ by
\begin{equation*}
	e \bfR(\psi^{(1)}, \psi^{(2)}) \varphi = R(u)(e \psi^{(1)}, e \psi^{(2)}) e \varphi
\end{equation*}  
for $\bbR^{n}$-valued functions $\psi^{(j)}$ $(j=1,2)$ and $\varphi$. Observe that $\bfR(\psi^{(1)}, \psi^{(2)})$ takes values in $\mathfrak{so}(n)$. The identity \eqref{eq:curv-R} may be rewritten as
\begin{equation}  \label{eq:Fdu}
	F_{\alp \bt} = \bfR(\psi_{\alp}, \psi_{\bt}).
\end{equation} 
To facilitate the computation of derivatives of $F$, we introduce the pullback differentiated curvature (for $m = 1, 2, \ldots$)
\begin{equation*}
	\bfR^{(m)} (\psi^{(1)}, \ldots, \psi^{(m)}; \varphi^{(1)}, \varphi^{(2)})
	= (\nb^{(m)} R)(u) (e \psi^{(1)}, \ldots, e \psi^{(m)}; e \varphi^{(1)}, e \varphi^{(2)}) e,
\end{equation*}
under the convention $\nb^{(m)} R (X^{(1)}, \ldots, X^{(m)}; Y, Z) = \nb^{(m)}_{X^{(1)} \cdots X^{(m)}} R(Y, Z)$. In the case $m = 0$, we take $\bfR^{(0)} = \bfR$. From the definition, we may immediately derive the differentiation formula
\begin{equation} \label{eq:DR}
	(\bfD_{\mu} \bfR^{(m)})(\psi^{(1)}, \ldots, \psi^{(m)}; \varphi^{(1)}, \varphi^{(2)})
	= \bfR^{(m+1)} (\psi_{\mu}, \psi^{(1)}, \ldots, \psi^{(m)}; \varphi^{(1)}, \varphi^{(2)})
\end{equation}

Combined with the Leibniz rule, it is now straightforward to compute covariant derivatives of $F$ of any order. For instance, in order to compute $\bfD F$, we first apply the Leibniz rule to obtain
\begin{align*}
	\bfD_{\mu} F_{\alp \bt} 
	=(\bfD_{\mu} \bfR^{(0)}) (\psi_{\alp}, \psi_{\bt}) + \bfR^{(0)} (\bfD_{\mu} \psi_{\alp}, \psi_{\bt}) + \bfR^{(0)} (\psi_{\alp}, \bfD_{\mu} \psi_{\bt})
\end{align*}
Then by \eqref{eq:DR}, we arrive at the formula
\begin{equation} \label{eq:DF}
\bfD_{\mu} F_{\alp \bt} 
= \bfR^{(1)} (\psi_{\mu} ; \psi_{\alp}, \psi_{\bt}) + \bfR^{(0)} (\bfD_{\mu} \psi_{\alp}, \psi_{\bt}) + \bfR^{(0)} (\psi_{\alp}, \bfD_{\mu} \psi_{\bt})
\end{equation}

Finally, by the bounded geometry assumptions, observe that the following pointwise bound holds for each $\bfR^{(m)}$:
\begin{equation} \label{eq:DR-ptwise}
	\abs{\bfR^{(m)}(\psi^{(1)}, \ldots, \psi^{(m)}; \varphi^{(1)}, \varphi^{(2)})}
	\aleq \abs{\psi^{(1)}} \cdots \abs{\psi^{(m)}} \abs{\varphi^{(1)}} \abs{\varphi^{(2)}}
\end{equation}
where we use the Hilbert-Schmidt norm for matrices on the left-hand side.

\subsection{Function Spaces on $\Hp^d$}\label{s:fs} 

With the Riemannian structure we can define the relevant function spaces on $\Hp^d$. Let $f: \Hp^d \to \R$ be a smooth function. The $L^p( \Hp^d)$ spaces are defined for $1 \le p < \infty$ by 
\begin{align*}
\|f \|_{L^p(\bbH^d)} :=  \left( \int_{\bbH^d} | f(x) |^p\,  \dvol_{\h} \right)^{\frac{1}{p}}, \quad \| f \|_{L^{\infty}( \bbH^d)} :=  \sup_{x \in \Hp^d} \abs{f(x)}
\end{align*} 
The Sobolev norm $W^{k, p}( \Hp^d; \R) $ are defined as follows, see for example~\cite{Heb}: 
\EQ{ \label{sobh4} 
\| f\|_{W^{k, p}(\Hp^d; \R)}:=    \sum_{\ell=0}^k  \left( \int_{\Hp^d} \abs{ \na^{(\ell)} f}^p  \, \dvol_\h \right)^{\frac{1}{p}}
}
where here $\na^{(\ell)} f$ is the $\ell$th covariant derivative of $f$ with the convention that $\na^{(0)} f = f$. We note that in coordinates the components of $\na f$ are given by $(\na f)_a =  \p_a f$, and those of $\na^{(2)} f$ are then 
\EQ{
(\na^{2} f)_{ab} =  \p_a \p_b f - \Gamma^c_{a b} \p_c f
}
and so on. We then have 
\EQ{ \label{eq:abs-nb-ell-f}
\abs{\na^{(\ell)} f}^p = \left(\h^{a_1 b_1} \cdots \h^{a_\ell b_\ell} (\na^{(\ell)} f)_{a_1 \cdots a_\ell}  (\na^{(\ell)} f)_{b_1 \cdots b_\ell} \right)^{\frac{p}{2}}
}
One then defines the Sobolev space~$W^{k, p}(\Hp^d; \R)$ to be the completion of all smooth compactly supported functions $f \in C^{\infty}_0( \Hp^d; \R)$ under the norm~\eqref{sobh4}; see~\cite[Theorem $2.8$]{Heb}. We note that for $p=2$ we write $W^{k, 2}(\Hp^d; \R) =: H^k(\Hp^d; \R)$. The Sobolev spaces $W^{s, p}$ for $s \ge 0$ are then defined by interpolation.

Alternatively one can define Sobolev spaces on $\Hp^d$ using the spectral theory of the Laplacian $-\Delta_{\Hp^d}$, which we will now denote simply by  $-\Delta$.  Given $s \in \bbR$, the fractional Laplacian $(-\Delta)^{\frac{s}{2}}$ is a well-defined operator on, say $C^{\infty}_{0}(\Hp^{d})$ via the spectral theory of $-\Delta$.  Given a function $f \in C^{\I}_{0}(\Hp^d)$, we set
\begin{align*}
\|f \|_{\ti W^{s, p}(\Hp^d)} = \|(- \Delta)^{\frac{s}{2}} f\|_{L^p(\Hp^d)}
\end{align*}
and define $\ti W^{s, p}( \Hp^d)$ to be the completion of $C^{\infty}_0( \Hp^d)$ under the  norm above.

The fractional Laplacian $(- \De)^{\frac{s}{2}}$ is bounded on $L^p(\Hp^d)$ for all $s \le 0$ and all $p \in (1, \infty)$. Using this fact one can show that $\ti W^{s, p} = W^{s, p}$ with equivalent norms, where the latter space is defined using the Riemannian structure as above; see for example~\cite{Tat01hyp} for a proof, and also Lemma~\ref{l:riesz} below.
In the sequel we will often use these two definitions interchangeably, with expressions like 
\EQ{
 \| \na (-\De)^{\ell} f\|_{L^2} \simeq  \|  \na^{(2 \ell +1)} f \|_{L^2} \simeq    \| f\|_{H^{2 \ell +1 }} 
}
where the last $\simeq$ is due to the Poincar\'e inequality; see Lemma~\ref{l:pi}. We will often use the notation 
\EQ{
\ang{f \mid g}_{L^2} := \int_{\Hp^d} f(x) g(x) \, \dvol_{\h}
}
to denote the $L^2(\Hp^d)$ inner product of $f, g$. 

We state some basic functional inequalities that hold on $\bbH^{d}$.
\begin{lem}[Poincar\'e inequality] \label{l:pi} 
Let $f \in C^{\infty}_{0}(\bbH^{d})$. Then the following inequalities hold.
\begin{align}
	\nrm{f}_{L^{2}} \leq & \bb(\frac{d-1}{2}\bb)^{-1} \nrm{\nb f}_{L^{2}} \\
	\nrm{f}_{L^{2}} \leq & \bb( \frac{d-1}{2} \bb)^{-2}\nrm{\Dlt f}_{L^{2}}
\end{align}
\end{lem} 
Both inequalities are equivalent to the well-known fact that $-\De$ on $\bbH^{d}$ has a spectral gap of $(\frac{d-1}{2})^{2}$ (see~\cite{Bray}). 

\begin{lem}[Gagliardo-Nirenberg inequality]\label{l:gn}  
Let $f \in C^{\infty}_{0}(\bbH^{d})$. Then for any $1 < p < \infty$, $p \leq q \leq \infty$ and $0 < \tht < 1$ such that
\ant{
\frac{1}{q} =  \frac{1}{p} - \frac{\te}{d}
}
we have
\EQ{
\| f\|_{L^q(\bbH^{d})} \lesssim \| f\|_{L^p(\bbH^{d})}^{1-\te} \| \na f\|_{L^p(\bbH^{d})}^{\te}
}

\end{lem}
This lemma can be proved using Littlewood-Paley projections defined (say) by the linear heat equation (see~\cite{IPS, LOS2}). In Lemma~\ref{l:gn}, we note that $\nrm{\nb f}_{L^{p}} \simeq \nrm{(-\Dlt)^{\frac{1}{2}} f}_{L^{p}}$ (see~\cite{Tat01hyp}). We omit the routine details.

\begin{lem}[Sobolev embedding]\label{l:se} 
Let $f \in C^{\infty}_{0}(\bbH^{d})$. Then for any $1 < p < d$ and $p \leq q < \infty$ with 
\begin{equation*}
	\frac{1}{q} = \frac{1}{p} - \frac{1}{d}
\end{equation*}
we have
\begin{equation} \label{eq:se}
	\nrm{f}_{L^{q}(\bbH^{d})} \aleq \nrm{\nb f}_{L^{p}(\bbH^{d})}
\end{equation}
\end{lem} 
For a proof, we refer to \cite[Theorem~3.2]{Heb}.

We conclude with some statements related to the fractional Laplacian $(-\Dlt)^{\alp}$.
\begin{lem}[$L^{2}$ interpolation inequalities] \label{l:L2-int}
Let $f \in C^{\infty}_{0}(\bbH^{d})$. Then for $0 \leq \bt \leq \alp$ we have
\begin{equation*}
	\nrm{(-\Dlt)^{\bt} f}_{L^{2}(\bbH^{d})} \aleq \nrm{f}_{L^{2}(\bbH^{d})}^{1-\frac{\bt}{\alp}} \nrm{(-\Dlt)^{\alp} f}_{L^{2}(\bbH^{d})}^{\frac{\bt}{\alp}}
\end{equation*}
Moreover, for $\frac{d}{2} < \alp$ we have
\begin{equation*}
	\nrm{f}_{L^{\infty}(\bbH^{d})} \aleq \nrm{f}_{L^{2}(\bbH^{d})}^{1 - \frac{d}{2 \alp}} \nrm{(-\Dlt)^{\frac{\alp}{2}} f}_{L^{2}(\bbH^{d})}^{\frac{d}{2 \alp}}
\end{equation*}
\end{lem}

\begin{lem}[Boundedness of Riesz transform] \label{l:riesz}
Let $f \in C^{\infty}_{0}(\bbH^{d})$. Then for $1 < p < \infty$ we have
\begin{equation*}
	\nrm{\nb f}_{L^{p}} \simeq \nrm{(-\Dlt)^{\frac{1}{2}} f}_{L^{p}}
\end{equation*}
\end{lem}

Lemma~\ref{l:L2-int} can be proved using Littlewood-Paley projections as in Lemma~\ref{l:gn}. For a proof of Lemma~\ref{l:riesz}, we refer to \cite[Theorem~6.1]{strichartz}.


\subsubsection{Norms for tensors and sections}
To extend the norms $L^{p}(\bbH^{d}; \bbR)$, $W^{k, p}(\bbH^{d}; \bbR)$ and $H^{k}(\bbH^{d}; \bbR)$ to vector-valued functions, we need to specify the pointwise norm for the vector. In the cases of interest, i.e., when the vector space is $\bbR^{n}$ or $\mathfrak{so}(n)$, we use 
\begin{equation} \label{eq:op-norm}
	\abs{\varphi}_{\bbR^{n}}^{2} = \brk{\varphi, \varphi}_{\bbR^{n}} := \sum_{k=1}^{n} (\varphi^{k})^{2}, \quad
	\abs{M}_{\mathfrak{so}(n)}^{2} = \brk{M, M}_{\mathfrak{so}(n)} := \sum_{i, j = 1}^{n} M^{2}_{ij}
\end{equation}
Observe that each norm arises from an inner product. When it is clear which one is being used from context, we will often omit the subscript $\bbR^{n}$ and $\mathfrak{so}(n)$. 
Using these pointwise norms, we define the function norms $L^{p}(\bbH^{d}; V)$, $W^{k, p}(\bbH^{d}; V)$ and $H^{k}(\bbH^{d}; V)$ ($V = \bbR, \bbR^{n}$ or $\mathfrak{so}(n)$)  in the componentwise fashion. Functional inequalities in Lemma~\ref{l:pi}--Lemma~\ref{l:se} as well as those in Lemma~\ref{l:L2-int} and Lemma~\ref{l:riesz} involving the fractional Laplacian still hold in this setting.

These norms may be further extended to $V$-valued tensors using the metric $\h$, in a similar fashion to \eqref{eq:abs-nb-ell-f}. For instance, for a $\bbR^{n}$-valued covariant tensor $\varphi_{a_{1} \cdots a_{\ell}}$, we define these spaces using the pointwise norm
\begin{equation*}
	\abs{\varphi}
	= \bb( \bfh^{a_{1} b_{1}} \cdots \bfh^{a_{\ell} b_{\ell}} \brk{\varphi_{a_{1} \cdots a_{\ell}}, \varphi_{b_{1} \cdots b_{\ell}}}\bb)^{\frac{1}{2}}
\end{equation*}

The basic functional inequalities stated in Lemma~\ref{l:pi}, Lemma~\ref{l:gn} and Lemma~\ref{l:se} remain valid in the tensorial case, thanks to the following result.

\begin{lem}[Diamagnetic inequality] \label{l:diamag}
Let $T$ be a smooth real-, vector- or matrix-valued tensor on $\bbH^{d}$. Then we have
\begin{equation} \label{eq:diamag}
	\abs{\nb \abs{T}} \leq \abs{\nb T}
\end{equation}
in the distributional sense.
\end{lem}
More precisely, \eqref{eq:diamag} holds when tested against smooth non-negative test functions. The proof proceeds by justifying the following formal computation:
\begin{equation*}
	\abs{\nb \abs{T}} = \abs{\nb \sqrt{\abs{T}^{2}}} = \abs{\frac{\brk{T, \nb T}}{\abs{T}}} \leq \abs{\nb T}
\end{equation*}
It is clear that the key structural property we rely on is the fact that the pointwise norm arises from an inner product, which is parallel with respect to the covariant derivative $\nb$ (i.e., $\nb \brk{X, Y} = \brk{\nb X, Y} + \brk{X, \nb Y}$). We omit the straightforward details for an actual proof.


Next, we consider objects defined on $\MM = \bbR \times \bbH^{d}$, or more generally on its subsets of the form $\MM_{I} = I \times \bbH^{d}$. Here the Lorentzian metric $\etab$ is unsuitable for defining pointwise norms for ($V$-valued) tensors, since it is not positive definite. To rectify this issue, we introduce a Riemannian metric $\tilde{\etab} = (\tilde{\etab}_{\alp \bt})$ defined by $\tilde{\etab}_{00} = 1$, $\tilde{\etab}_{0a} = 0$ and $\tilde{\etab}_{ab} = \bfh_{ab}$. Using this auxiliary metric $\tilde{\etab}$, we define appropriate pointwise norms in the same fashion as above. For example, if $\Psi = (\psi_{\alp})$ is a $\bbR^{n}$-valued 1-form on $\MM_{I}$, then by the pointwise norm $\abs{\Psi}$ of $\Psi$ we mean
\begin{equation} \label{eq:tilde-eta}
	\abs{\Psi} = \bb(  \tilde{\etab}^{\alp \bt} \brk{\psi_{\alp}, \psi_{\bt}} \bb)^{\frac{1}{2}} = \bb( \brk{\psi_{0}, \psi_{0}} + \bfh^{ab} \brk{\psi_{a}, \psi_{b}}\bb)^{\frac{1}{2}}
\end{equation}
Since the metric $\etab$ is independent of $t$ in the $(t, x^{a})$ coordinates on $\MM_{I}$, the auxiliary metric $\tilde{\etab}$ is parallel with respect to the Levi-Civita connection. Hence the diamagnetric inequality (Lemma~\ref{l:diamag}) holds for the above pointwise norms on $\MM_{I}$ as well, and 
Lemma~\ref{l:pi}, Lemma~\ref{l:gn} and Lemma~\ref{l:se} extend to norms of tensors on $\MM_{I}$.

\subsubsection{Norms for maps $ \Hp^d \to \NN \hookrightarrow \bbR^N$ } 
The main objects of study in the paper are smooth maps $u: \Hp^d \to \NN$ that tend to a fixed point $u_{\infty} \in \NN$ on the target at infinity, i.e.,
 \begin{equation} \label{eq:u-infty-intr}
	\bfd_{\calN}(u(x), u_{\infty}) \to 0 \mas \bfd(x, \zero) \to \infty
\end{equation}
where $\bfd$, respectively  $\bfd_{\NN}$, is the geodesic distance function on $\Hp^d$, respectively $\NN$.  For technical purposes, it will be convenient to define a Sobolev norm for maps $u: \Hp^{d} \to \NN$ using the extrinsic representation. Assuming that $\NN \hookrightarrow \bbR^{N}$ isometrically and viewing $u$ as an $\bbR^{N}$-valued function, \eqref{eq:u-infty-intr} is equivalent to
\ant{
 \abs{ u(x) - u_{\infty}}_{\R^N} \to 0 \mas \bfd(x, \zero) \to \infty
 }
 We will say that such a $u$ has finite $W^{s, p}$ norm with respect to $u_{\infty}$ if 
\EQ{
\| u- u_{\infty}\|_{W^{s, p}( \Hp^d; \R^N)} < \infty
}
where the norm $W^{s, p}(\Hp^{d}; \R^{N})$ for $\bbR^{N}$-valued functions is defined as above.
We write 
\EQ{
\| u\|_{W^{s, p}( \Hp^d; (\NN, u_{\infty}))}  := \| u- u_{\infty}\|_{W^{s, p}(\Hp^d; \R^N)}
}
The point of this notation is that one can make sense of the $L^p$ spaces along with various functional inequalities relative to the fixed base point $u_{\infty}$ on the target.  In particular,  appropriate versions of the inequalities in Lemma~\ref{l:pi}--Lemma~\ref{l:riesz} extend to smooth $\NN$-valued maps which tend to $u_{\infty}$ at infinity, as discussed above. 

Given a $u^{\ast} T \NN$-valued tensor $\phi$, which we view extrinsically as an $\bbR^{N}$-valued tensor, we also introduce the notation
\begin{equation} \label{eq:extr-TN-nrm}
	\nrm{\phi}_{W^{s, p}(\bbH^{d}; T \NN)} := \nrm{\phi}_{W^{s, p}(\bbH^{d}; \bbR^{N})}
\end{equation}
For any positive $s \in \bbN$, we have
\begin{equation}
	\nrm{d u}_{W^{s-1, p}(\bbH^{d}; T \NN)} \simeq \nrm{u}_{W^{s, p}(\bbH^{d}; (\NN, u_{\infty}))}
\end{equation}

\subsection{The linear heat flow on $\Hp^d$}

Our analysis will make use of several heat flows on hyperbolic space. The most basic is the linear heat equation, 
\EQ{\label{eq:heat} 
(\p_s - \De) f = 0
}
where, for the moment, $f$ is a real-valued function on $\bbH^{d}$.
Denote by $p(x, y; s)$ the \emph{heat kernel} on $\Hp^d$, that is, the kernel of the heat semi-group operator $e^{s \Delta}$ acting on real valued functions on $\Hp^d$, i.e., 
\EQ{ \label{eq:hk} 
e^{s \De} f(x) = \int_{\Hp^d} p(x, y; s) f(y) \dvol_{\h}
}
Let $J \subset \bbR^{+}$ be an interval extending from $0$, $f_{0} \in L^{2}(\bbH^{d})$ and $G \in L^{1}(J; L^{2}(\bbH^{d}))$. The unique solution in $C([0, \infty); L^{2}(\bbH^{d}))$ to the inhomogeneous equation 
\EQ{ \label{eq:heat-ih}
\p_s f - \De f = G , \quad 
f \!\!\restriction_{s=0} = f_0
}
is given by  Duhamel's formula, 
\EQ{
f(s) = e^{s\De} f_0 + \int_0^s e^{(s-s') \De}G(s') \, ds'
}

\subsubsection{$L^{2}$ regularity theory} 

The starting point for the linear heat equation~\eqref{eq:heat-ih} is the $L^2$-based regularity theory. We first sketch the energy integral approach for \eqref{eq:heat-ih}, which only relies on integration by parts. Based on this approach, we establish a more refined estimate for the homogeneous case \eqref{eq:hk} and present a simple extension to the inhomogeneous case \eqref{eq:heat-ih}, which will be used in Section~\ref{s:caloric} to establish well-posedness theory for the harmonic map heat flow.

Let $f(s, \cdot)$ be a solution to \eqref{eq:heat-ih}, where we assume both $f$ and $G$ to be smooth and decaying sufficiently fast at the spatial infinity; the precise assumptions needed will be clear for each estimate we state below. Multiplying \eqref{eq:heat-ih} by $f$ and integrating by parts over $\bbH^{d}$, we obtain
\begin{align*}
	\brk{G \mid f}
	= \brk{\rd_{s} f \mid f}
	+\brk{-\Dlt f \mid f}
	= \frac{1}{2} \rd_{s} \nrm{f}_{L^{2}}^{2}
	+ s^{-1} \nrm{s^{\frac{1}{2}} \nb f}_{L^{2}}^{2}
\end{align*}
Integrating over the $s$-intervals $(0, s)$ for $s \in J$, using H\"older's inequality and taking the square root, we arrive at
\begin{equation} \label{eq:ih-en}
	\nrm{f}_{L^{\infty}_{\ds}(J; L^{2})}
	+ \nrm{s^{\frac{1}{2}} \nb f}_{L^{2}_{\ds}(J; L^{2})}
	\aleq \nrm{f_{0}}_{L^{2}} + \nrm{G}_{L^{1}_{\ds}(J; L^{2})}
\end{equation}
This bound is useful from the point of view of Cauchy problem, since the right-hand side depends only on the data $f_{0}$ and $G$. 

On the other hand, multiplying \eqref{eq:heat-ih} by $s^{\alp} (-\Dlt)^{\alp} f$ with an appropriate $\alp \neq 0$, we obtain the following sharp $L^{2}$ parabolic regularity bound.
\begin{lem} \label{lem:ih-reg}
Let $J \subset \bbR^{+}$ be an interval extending from $0$. Let $f$ be a solution to \eqref{eq:heat-ih} with $f_{0} \in L^{2}$ and
\begin{equation*}
	s^{1+k} (-\Dlt)^{k} G \in L^{2}_{\ds}(J; L^{2}(\bbH^{d})), \quad
	s^{k} (-\Dlt)^{k} f  \in L^{2}_{\ds}(J; L^{2}(\bbH^{d}))
\end{equation*}
for some $k \geq 0$. Then $f$ obeys the bound
\begin{equation} \label{eq:ih-reg}
\begin{aligned}
& \hskip-2em
	\nrm{s^{k+\frac{1}{2}} (-\Dlt)^{k+\frac{1}{2}} f}_{L^{\infty}_{\ds}(J; L^{2})}
	+ \nrm{s^{k+1} (-\Dlt)^{k+1} f}_{L^{2}_{\ds}(J; L^{2})}  \\
	\aleq & \nrm{s^{k} (-\Dlt)^{k} f}_{L^{2}_{\ds}(J; L^{2})}
	+ \nrm{s^{k+1} (-\Dlt)^{k} G}_{L^{2}_{\ds}(J; L^{2})}
\end{aligned}
\end{equation}
\end{lem}
\begin{proof} 
Multiplying the equation \eqref{eq:heat-ih} by $s^{2 k+1} (-\Dlt)^{2k+1} f$, we have
\begin{align*}
	\brk{\rd_{s} f \mid s^{2 k+1} (-\Dlt)^{2k+1} f}
	+\brk{-\Dlt f \mid s^{2 k+1} (-\Dlt)^{2k+1} f} = \brk{G \mid s^{2 k+1} (-\Dlt)^{2k+1} f}
\end{align*}
Using functional calculus of the self-adjoint operator $-\Dlt$, the previous identity can be rewritten as
\begin{align*}
& \hskip-2em
	\frac{1}{2} \rd_{s} \nrm{s^{k+\frac{1}{2}} (-\Dlt)^{k+\frac{1}{2}} f}_{L^{2}}^{2}
	+ s^{-1}\nrm{s^{k+1} (-\Dlt)^{k+1} f}_{L^{2}}^{2}  \\
	=& \bb(k+\frac{1}{2} \bb) s^{-1} \brk{s^{k} (-\Dlt)^{k} f \mid s^{k+1} (-\Dlt)^{k+1} f} \\
	& + s^{-1} \brk{s^{k+1} (-\Dlt)^{k} G \mid s^{k+1} (-\Dlt)^{k+1} f}
\end{align*}
We may write $J = (0, s_{0})$ for some $s_{0} > 0$. Integrating over the $s$-interval $(s', s_{0})$ and taking $s' \to 0$, we obtain
\begin{align*}
& \hskip-2em
	\frac{1}{2} \nrm{s^{k+\frac{1}{2}} (-\Dlt)^{k+\frac{1}{2}} f}_{L^{\infty}_{\ds}(J; L^{2})}^{2}
	+ \nrm{s^{k+1} (-\Dlt)^{k+1} f}_{L^{2}_{\ds}(J; L^{2})}^{2} \\
	 \le &  \Big|\int_{J} \brk{( k+\frac{1}{2} ) s^{k} (-\Dlt)^{k} f + s^{k+1} (-\Dlt)^{k} G \mid s^{k+1} (-\Dlt)^{k+1} f} \, \ds\Big|
\end{align*}
Applying Cauchy-Schwarz to the right-hand side and absorbing $s^{k+1} (-\Dlt)^{k+1} f$ on the left-hand side, the desired inequality \eqref{eq:ih-reg} follows.
\end{proof}

We now return to the study of $L^{2}$ theory of \eqref{eq:heat-ih} from the point of view of the Cauchy problem.
The starting point is a thorough understanding of the homogenous case, where $G = 0$ and $f(s) = e^{s \Dlt} f_{0}$. To simplify the notation, we will often drop the subscript $0$ in $f_{0}$. Starting from \eqref{eq:ih-en} and applying \eqref{eq:ih-reg} inductively for $k = \frac{1}{2}, \frac{3}{2}, \ldots$, the following $\Ls^{\infty}$ estimate holds.
 \begin{lem} \label{lem:hh-L2:infty}
Let $f \in L^{2}(\bbH^{d})$ and $\alp \geq 0$. Then
 \begin{equation} \label{eq:hh-L2:infty}
	\nrm{s^{\alp} (-\Dlt)^{\alp} e^{s \Dlt} f}_{L^{2}(\bbH^{d})} \aleq_{\alp} \nrm{f}_{L^{2}(\bbH^{d})}
\end{equation}
\end{lem}
By the same argument, we obtain appropriately $s$-weighted $\Ls^2$ estimates for the $L^2(\Hp^d)$ norm of all derivatives $e^{s\De}f$, but  not for the $L^2$ norm of $e^{s\De}f$ itself. To compensate for this in the sequel we will also require the following $s$-weighted $\Ls^2$ estimate on the $L^2$ norm of $\alp$ derivatives of $e^{s\De}f$ for any $\alp > 0$. 
\begin{lem} \label{lem:hh-L2:2} Let $f \in L^2(\Hp^d)$ and $\alp > 0$. Then, 
\EQ{ \label{eq:hh-L2:2}
\bb( \int_0^\infty\| s^{\alp} (-\De)^{\alp} e^{s \De} f\|_{L^2(\Hp^d)}^2 \, \ds \bb)^{1/2} \lesssim_{\alp} \| f\|_{L^2(\Hp^d)}
}

\end{lem}
\begin{proof} Let $Y= Y( \R^+ \times \Hp^d)$, and $Y^* = Y^*(\R \times \Hp^d)$  denote the Banach spaces defined by the norms 
\ant{
\| g\|_{Y}^2 :=  \int_0^\infty \| s^{\alp} (-\De)^{\alp} g(s)\|_{L^2}^2 \, \ds, \quad \| F\|_{Y^*}^2 =  \int_0^\infty \| s^{-\alp} (-\De)^{-\alp} F(s)\|_{L^2}^2 \, \ds
}
Define  $T f := e^{s \De} f$. Our goal is to show that $T: L^{2}(\Hp^d) \to Y(\R^+ \times \Hp^d)$ with bounded norm. By the $TT^*$ method it suffices to prove the dual estimate 
\EQ{ \label{eq:Td} 
\| T^* F\|_{L^2(\Hp^d)} \lesssim \| F\|_{Y^*( \R^+ \times \Hp^d)}
}
where $T^* F := \int_0^\I e^{s \De} F(s) \, \ds$. To prove~\eqref{eq:Td} we write $G = s^{-\alp} (-\De)^{\alp} F$ and square the left-hand side above.  We have 
\ant{
\Big\| \int_0^\infty &e^{s \De} F(s) \, \ds  \Big\|_{L^2}^2    \\
 &= \int_0^\infty \int_0^\infty \ang{ s^{\alp}e^{s\De} (-\De)^{\alp} G(s) \mid (s')^{\alp} e^{s'\De}(-\De)^{\alp} G(s')}_{L^2} \, \ds \dsp \\
 & = \int_0^\infty \int_0^\infty s^{\alp} (s')^{\alp}\ang{  G(s) \mid (-\De)^{2 \alp} e^{(s+s')\De} G(s')}_{L^2} \, \ds \dsp \\
 & = 2\int_0^\infty \int_s^\infty s^{\alp} (s')^{\alp}\ang{  G(s) \mid (-\De)^{2 \alp} e^{(s+s')\De} G(s')}_{L^2} \, \ds \dsp 
}
 where the last line follows by symmetry in $s$, $s'$. Using Cauchy-Schwarz and the estimate~\eqref{eq:hh-L2:infty} we can bound the last line above by 
 \ant{
 2\int_0^\infty \int_s^\infty &s^{\alp} (s')^{\alp}\ang{  G(s) \mid (-\De)^{2 \alp}e^{(s+s')\De} G(s')}_{L^2} \, \ds \dsp  \\
 & \lesssim  \int_0^\infty \int_s^\infty \frac{s^{\alp} (s')^{\alp}}{ (s+s')^{2 \alp}}  \| G(s)\|_{L^2} \| G(s')\|_{L^2} \,  \ds \dsp \\
 & \lesssim  \int_0^\infty \int_s^\infty\left( \frac{s}{s'}\right)^{\alp}   \| G(s)\|_{L^2} \| G(s')\|_{L^2} \,  \ds \dsp \\
 & \lesssim \|G\|_{\Ls^2 L^2}^2 \simeq \| s^{-\alp} (-\De)^{-\alp} F\|_{\Ls^2 L^2}^2
 }
where the last line follows from Schur's test. This completes the proof. 
\end{proof} 

The above homogeneous $L^{2}$ estimates extend to the inhomogeneous case \eqref{eq:heat-ih} via the following estimates for the Duhamel integral.
\begin{lem} \label{lem:ih}
Let $J \subset \bbR^{+}$ be an interval extending from $0$, i.e., $J = (0, s_{0}]$ for some $s_{0} > 0$. Let $G$ satisfy $s G \in L^{1}_{\ds}(J; L^{2}(\bbH^{d}))$.  Then for $0 \leq \alp < 1$, we have
\begin{align*}
	\nrm{s^{\alp} (-\Dlt)^{\alp} \int_{0}^{s} s' e^{(s - s') \De} G(s') \, \frac{d s'}{s'}}_{L^{\infty}_{\ds}(J; L^{2}(\bbH^{d}))}
	\aleq_{\alp} & \nrm{s G}_{L^{1}_{\ds} \cap L^{\infty}_{\ds} (J; L^{2}(\bbH^{d}))} 
\end{align*}
For $0 < \alp < 1$, we have 
\begin{align*}
	\nrm{s^{\alp} (-\Dlt)^{\alp} \int_{0}^{s} s' e^{(s - s') \De} G(s') \, \frac{d s'}{s'}}_{L^{2}_{\ds}(J; L^{2}(\bbH^{d}))}
	\aleq_{\alp} & \nrm{s G}_{L^{1}_{\ds} \cap L^{2}_{\ds} (J; L^{2}(\bbH^{d}))}
\end{align*}
Finally, corresponding to the case $\alp = 1$ we have
\begin{align*} 
	& \nrm{s \rd_{s} \int_{0}^{s} s' e^{(s - s') \De} G(s') \, \dsp}_{L^{2}_{\ds}(J; L^{2}(\bbH^{d}))}	\\
	& + \nrm{s \Dlt \int_{0}^{s} s' e^{(s - s') \De} G(s') \, \dsp}_{L^{2}_{\ds}(J; L^{2}(\bbH^{d}))} 
	\aleq  \nrm{s G}_{L^{1}_{\ds} \cap L^{2}_{\ds} (J; L^{2}(\bbH^{d}))}
\end{align*}
\end{lem}

\begin{proof} 
The case $\alp = 1$ follows by the energy integral method, or more precisely, by the combination of \eqref{eq:ih-en} and the following variant of \eqref{eq:ih-reg} in the case $k = 0$: 
\begin{equation*}
	\nrm{s^{\frac{1}{2}} (-\Dlt)^{\frac{1}{2}} f}_{L^{\infty}_{\ds}(J; L^{2})}
	+ \nrm{s \Dlt f}_{L^{2}_{\ds}(J; L^{2})}  
	\aleq \nrm{s^{\frac{1}{2}} (-\Dlt)^{\frac{1}{2}} f}_{L^{2}_{\ds}(J; L^{2})}
	+ \nrm{s G}_{L^{2}_{\ds}(J; L^{2})}
\end{equation*}
We skip the proof, which is similar to that of \eqref{eq:ih-reg}. Henceforth, we fix either $r = \infty$ and $0 \leq \alp < 1$, or $r = 2$ and $0 < \alp < 1$. In the former case, the summation $(\sum (\cdot)^{r})^{1/r}$ is to be interpreted as a supremum as usual.

Extending $G$ by zero outside $J$, we may assume that $J =\bbR^{+}$. For any $k \in \bbZ$,
we split
\begin{equation*}
	G_{k}(s, x) = 1_{\set{2^{k} < s \leq 2^{k+1}}} G(s, x)
\end{equation*}
and accordingly, we define
\begin{align*}
H_{k} (s, x) 
=&   \int_{0}^{s}  s' e^{(s-s') \Dlt} G_{k}(s', x) \, \dsp
\end{align*} 
Then by the triangle inequality, we have
\begin{align*}
	\nrm{s^{\alp} (-\Dlt)^{\alp} \int_{0}^{s} s' e^{(s-s') \Dlt} G(s') \dsp}_{L^{r}_{\ds}(J; L^{2})} 
	= &\nrm{\sum_{k} s^{\alp} (-\Dlt)^{\alp} H_{k}}_{L^{r}_{\ds} L^{2}}  \\
	\leq & \nrm{\sum_{k} 1_{0 < s \leq 2^{k+2}}  \, s^{\alp} (-\Dlt)^{\alp} H_{k}}_{L^{r}_{\ds}L^{2}} \\
	& + \sum_{k} \nrm{s^{\alp} (-\Dlt)^{\alp} H_{k}}_{L^{r}_{\ds}((2^{k+2}, \infty) ; L^{2})} \\
	=: & I + II
\end{align*}

We first treat the term $I$. Note that $H_{k} (s) = 0$ for $s \in (0, 2^{k}]$ by the support property of $G_{k}$; therefore
\begin{equation*}
I = \nrm{\sum_{k} 1_{2^{k} < s \leq 2^{k+2}} s^{\alp} (-\Dlt)^{\alp} H_{k}}_{L^{r}_{\ds} L^{2}}
\end{equation*}
Since the intervals $(2^{k}, 2^{k+2}]$ have finite overlap, we may estimate 
\begin{align*}
& \hskip-2em
	\nrm{\sum_{k} 1_{2^{k} < s \leq 2^{k+2}}  \, s^{\alp} (-\Dlt)^{\alp} H_{k}}_{L^{r}_{\ds}L^{2}} \\
	\aleq & \bb( \sum_{k} \nrm{s^{\alp} (-\Dlt)^{\alp} H_{k}}_{L^{r}_{\ds}((2^{k}, 2^{k+2}]; L^{2})}^{r} \bb)^{1/r} \\
	\aleq & \bb( \sum_{k} \nrm{\int_{2^{k}}^{s} s^{\alp} (s-s')^{-\alp} s' \nrm{G_{k}(s')}_{L^{2}} \frac{\ud s'}{s'}}_{L^{r}_{\ds}(2^{k}, 2^{k+2}]}^{r} \bb)^{1/r}
\end{align*}
By Schur's test, it follows that
\begin{align*}
\bb( \sum_{k} \nrm{\int_{2^{k}}^{s} s^{\alp} (s-s')^{-\alp} s' \nrm{G_{k}(s')}_{L^{2}} \frac{\ud s'}{s'}}_{L^{r}_{\ds}(2^{k}, 2^{k+2}]}^{r} \bb)^{1/r}
\aleq \bb( \sum_{k} \nrm{s' G_{k}}_{L^{r}_{\dsp} L^{2}}^{r} \bb)^{1/r}
\end{align*}
Since the supports of $G_{k}$ in $s$ are disjoint, the last term is bounded by $\nrm{s G}_{L^{r}_{\ds} (\bbR^{+}; L^{2})}$.
Hence we have proved
\begin{equation*}
I = \nrm{\sum_{k} 1_{2^{k} < s \leq 2^{k+2}}  \, s^{\alp} (-\Dlt)^{\alp} H_{k}}_{L^{r}_{\ds}L^{2}} \aleq \nrm{s G}_{L^{r}_{\ds}(\bbR^{+}; L^{2})}
\end{equation*}

Next, we pass to the term $II$. For $s > 2^{k+2}$, we may write
\begin{equation*}
	H_{k}(s) = e^{(s - 2^{k+1}) \Dlt} \int_{2^{k}}^{2^{k+1}} s' e^{(2^{k+1}-s') \Dlt} G(s', x) \, \frac{d s'}{s'} 
\end{equation*}
Observe also that $(s - 2^{k+1}) \aeq s$. Thanks to the assumption that either $\alp \geq 0$ and $r = \infty$ or $\alp > 0$ and $r = 2$, the homogeneous estimates \eqref{eq:hh-L2:infty} and \eqref{eq:hh-L2:2} imply that
\begin{align*}
	\nrm{s^{\alp} (-\Dlt)^{\alp} H_{k}(s, x)}_{L^{r}_{\ds}((2^{k+2}, \infty); L^{2})}
	\aleq & \nrm{(s-2^{k+1})^{\alp} (-\Dlt)^{\alp} H_{k}(s, x)}_{L^{r}_{\ds}((2^{k+2}, \infty); L^{2})} \\
	\aleq & \nrm{\int_{2^{k}}^{2^{k+1}} s' e^{(2^{k+1} - s') \Dlt} G_{k}(s') \, \frac{d s'}{s'} }_{L^{2}} \\
	\aleq & \nrm{s' G_{k}}_{L^{1}_{\dsp} L^{2}}
\end{align*}
Summing over $k$, we obtain
\begin{equation*}
	II = \sum_{k} \nrm{s^{\alp} (-\Dlt)^{\alp} H_{k}(s, x)}_{L^{r}_{\ds}((2^{k+2}, \infty); L^{2})}
	\aleq \sum_{k} \nrm{s' G_{k}}_{L^{1}_{\dsp} L^{2}} \leq \nrm{s G}_{L^{1}_{\ds}(\bbR^{+}; L^{2})}
\end{equation*}
as desired. 
\end{proof}

\subsubsection{Extension to vector-valued functions} 
The $L^{2}$ theory for the real-valued solutions to \eqref{eq:heat-ih} extends naturally to vector-valued functions, which possibly tend to a nonzero constant $u_{\infty}$ at the spatial infinity. More precisely, on an interval $J \subset \bbR^{+}$ extending from $0$, consider the Cauchy problem
\begin{equation} \label{eq:heat-ih:vec}
\left\{
\begin{aligned}
	(\rd_{s} - \Dlt) u  =& G \\
	u \!\! \restriction_{s=0} =& u_{0}
\end{aligned}	
\right.
\end{equation}
where $u_{0} - u_{\infty} \in L^{2}(\bbH^{d}; \bbR^{N})$ for some $u_{\infty} \in \bbR^{N}$ and $G \in L^{1}(J; L^{2}(\bbH^{d}; \bbR^{N}))$. Then there exists a unique solution $u$ to \eqref{eq:heat-ih:vec} on $J$ in the class
\begin{equation*}
u(s, x) - u_{\infty} \in C(J \cup \set{0}; L^{2}(\bbH^{d}; \bbR^{N})).
\end{equation*}
To see this, simply observe that $(\rd_{s} - \Dlt) u_{\infty} = 0$ and apply the usual uniqueness argument for each component of $u - u_{\infty}$. The unique solution is given by Duhamel's formula
\begin{equation*}
	u - u_{\infty} = e^{s \Dlt} (u_{0} - u_{\infty}) + \int_{0}^{s} e^{(s - s') \Dlt} G(s') \, d s'
\end{equation*}
Note that the estimates in Lemma~\ref{lem:hh-L2:infty}, Lemma~\ref{lem:hh-L2:2} and Lemma~\ref{lem:ih} extend to the right-hand side of this formula in the natural componentwise manner.


\subsubsection{$L^{p}$ regularity theory}
We would also need $L^{p}$ boundedness and regularity theory of the heat semi-group $e^{s \Dlt}$.
We present a general argument which relies only on integration by parts.
 
\begin{lem} \label{l:hk}
Let $1 < p < \infty$. For any real-valued function $f \in C^{\infty}_{0}(\bbH^{d})$ and $s > 0$, the following bounds hold.
\begin{align}
	\nrm{e^{s \Dlt} f}_{L^{p}} \aleq & \nrm{f}_{L^{p}} \label{eq:pest} \\
	\nrm{s^{\frac{1}{2}} \nb e^{s \Dlt} f}_{L^{p}} \aleq & \nrm{f}_{L^{p}} \label{eq:napest} \\
	\nrm{s (-\Dlt) e^{s \Dlt} f}_{L^{p}} \aleq & \nrm{f}_{L^{p}} \label{eq:depest}
\end{align}
\end{lem}
\begin{proof} 
To ease the notation, we henceforth denote by $f(s)$ the solution to the homogeneous heat equation with $f(0) \in C^{\infty}_{0}(\bbH^{d})$. In what follows, we will present a formal argument, which can be easily made precise using the $L^{2}$ theory that we have already established.

We begin with some simple reductions. First, observe that once \eqref{eq:pest} and \eqref{eq:depest} are established, \eqref{eq:napest} follows easily from interpolation and Lemma~\ref{l:riesz}. Second, by self-adjointness of $e^{s \Dlt}$ and $s(-\Dlt) e^{s \Dlt}$ and duality, it suffices to prove \eqref{eq:pest} and \eqref{eq:depest} in the case $2 \leq p < \infty$. Finally, by interpolation and the $L^{2}$ theory that we have already established, it is enough to consider the case when $p$ is an even integer greater than $2$.

Let $p = 2n$ for some $n \geq 2$ and let $0 < \dlt \ll 1$ be a small parameter to be chosen below. 
As $f$ is a solution to the homogeneous heat equation, observe that $\abs{f}^{2}$ obeys the equation
\begin{equation*}
	\rd_{s} \abs{f}^{2} - \Dlt \abs{f}^{2} + 2 \abs{\nb f}^{2} = 0
\end{equation*}
Since $\Dlt f$ solves the homogeneous heat equation as well, $\abs{\Dlt f}^{2}$ obeys the same equation as above. Multiplying by $s^{2}$, we obtain
\begin{align*}
	\rd_{s} (s^{2} \abs{\Dlt f}^{2}) - \Dlt (s^{2} \abs{\Dlt f}^{2}) + 2 s^{2} \abs{\nb \Dlt f}^{2} - 2 s \abs{\Dlt f}^{2} = 0
\end{align*}
Rewriting the last term on the left-hand side, we have
\begin{equation*}
	\rd_{s} (s^{2} \abs{\Dlt f}^{2}) - \Dlt (s^{2} \abs{\Dlt f}^{2}) + 2 s^{2} \abs{\nb \Dlt f}^{2} + 2 s (\nb f \cdot \nb \Dlt f) - 2 s \nb \cdot (\nb f \Dlt f) = 0
\end{equation*} 
We introduce an auxiliary function
\begin{equation*}
	\psi^{2} = \abs{f}^{2} + \dlt^{2} s^{2} \abs{\Dlt f}^{2}
\end{equation*}
We claim that if $\dlt > 0$ is sufficiently small, then 
\begin{equation} \label{eq:pest-key}
	\sup_{s \in \bbR^{+}} \nrm{\psi(s)}_{L^{p}}^{p} \aleq \nrm{\psi(0)}_{L^{p}}^{p}
\end{equation}
Since $\psi(0) = \abs{f(0)}$, the desired estimates \eqref{eq:pest} and \eqref{eq:depest} would then follow.

It now only remains to verify the claim \eqref{eq:pest-key}. Adding up the equations for $\abs{f}^{2}$ and $s^{2} \abs{\Dlt f}^{2}$, we see that $\psi^{2}$ solves the equation
\begin{equation*}
	(\rd_{s} - \Dlt) \psi^{2} + F = 2 \dlt^{2} s \nb \cdot  (\nb f \Dlt f) 
\end{equation*}
where
\begin{equation*}
F = 2 \bb( \abs{\nb f}^{2} + \dlt^{2} s^{2} \abs{\nb \Dlt f}^{2} + \dlt^{2} (\nb f \cdot s \nb \Dlt f) \bb)
\end{equation*}
Multiplying the equation by $\psi^{2n-2}$
and integrating by parts over $\bbH^{d}$, we obtain
\begin{align*}
	\frac{1}{n} \rd_{s} \int_{\bbH^{d}} \psi^{2n} & + (n-1) \int_{\bbH^{d}} \abs{\nb \psi^{2}}^{2} \psi^{2n-4} + \int_{\bbH^{d}} F \psi^{2n-2} \\
	= & - 2 (n-1) \dlt \int_{\bbH^{d}} \nb f \cdot \nb \psi^{2} (\dlt s \Dlt f) \psi^{2n-4}
\end{align*}
Since $0 < \dlt \leq 1$, observe that $F$ obeys a $\dlt$-independent lower bound $F \geq \abs{\nb f}^{2}$.
Integrating over $s$-intervals of the form $(0, s)$, we arrive at
\begin{align*}
	\sup_{s \in \bbR^{+}} \nrm{\psi(s)}_{L^{p}}^{p}
	& + \int_{0}^{\infty} \int_{\bbH^{d}} \abs{\nb \psi^{2}}^{2} \psi^{2n-4} \, d s
	+ \int_{0}^{\infty} \int_{\bbH^{d}} \abs{\nb f}^{2} \psi^{2n-2} \, d s \\
	\aleq & \nrm{\psi(0)}_{L^{p}}^{p}
		+ \dlt \int_{0}^{\infty} \int_{\bbH^{d}} \abs{\nb f \cdot \nb \psi^{2} (\dlt s \Dlt f) \psi^{2n-4} } \, d s
\end{align*}
 By Cauchy-Schwarz and the inequality $\abs{\dlt s \Dlt f} \leq \abs{\psi}$, the second term on the right-hand side can be bounded from above by
\begin{align*}
	& \dlt \int_{0}^{\infty} \int_{\bbH^{d}} \abs{\nb f} \abs{\psi} \abs{\nb \psi^{2}} \psi^{2n-4} \, ds \\
	& \aleq   \dlt \bb(\int_{0}^{\infty} \int_{\bbH^{d}} \abs{\nb \psi^{2}}^{2} \psi^{2n-4} \, ds 
				+ \int_{0}^{\infty} \int_{\bbH^{d}} \abs{\nb f}^2 \psi^{2n-2} \, ds  \bb)
\end{align*}
Hence choosing $\dlt > 0$ sufficiently small, this term can be absorbed into the left-hand side. The desired claim \eqref{eq:pest-key} now follows. \qedhere
\end{proof}

\section{An overview of the proof of  Theorem~\ref{t:main}: Wave maps in the caloric gauge} \label{s:caloric}

In this section we set the stage for the proof of Theorem~\ref{t:main} and construct the caloric gauge. 

Our main objective over the next several sections will be to establish a priori estimates of the following form. 

\begin{prop}[A priori estimates]  \label{p:ap} 
Let $0 \in I \subset \R$ be a time interval and suppose that $u: I \times \Hp^4 \to  \NN$ is a smooth wave map on $I$.  Then there exists $\eps_{1}>0$ and a constant $C_0$ independent of $I$, $u$, so that if $\eps \leq \eps_{1}$ and
\EQ{ \label{eq:sd} 
  \| (du, \p_t u) \!\!\restriction_{t = 0}\|_{H^1 \times H^1( \Hp^4; T\NN)}  <  \eps
  }
  then 
  \EQ{ \label{eq:ap} 
  \sup_{t \in I} \| (du, \p_t u)(t)\|_{H^1 \times H^1( \Hp^4; T\NN)}  \le C_0 \eps.
  }
\end{prop} 

\begin{rem}
Since the wave maps system is critical with respect to the norm \eqref{eq:sd}, such an a priori estimate is not sufficient to continue the solution $u$ past $I$. In fact, we will prove that a certain controlling norm, stronger than the norm on the left-hand side of \eqref{eq:ap}, is bounded from above by $C \eps$; see Theorem~\ref{t:main1} below. For the sake of exposition, we defer the precise definition of the controlling norm until Section~\ref{s:outline}.
\end{rem}

We will prove Proposition~\ref{p:ap} by establishing the  following bootstrap hypothesis: There exists an $\eps_{1}>0$ small enough and a constant $C_0>0$ with the following property. Given  any time interval $I \ni 0,$ $\eps\leq \eps_{1}$, and smooth wave map $u$ on $I$ with initial data satisfying~\eqref{eq:sd} then  
\begin{align} \label{eq:bs} 
 \sup_{t \in I} \|  (du, \p_t u)(t)\|_{H^1 \times H^1( \Hp^4; T\NN)} &\le 2 C_0 \eps   \\
 &  \Longrightarrow \sup_{t \in I} \|  (du, \p_t u)(t)\|_{H^1 \times H^1( \Hp^4; T\NN)} \le C_0 \eps \label{eq:bsc}.
 \end{align}

To prove that~\eqref{eq:bs} implies~\eqref{eq:bsc}, we employ the derivative formulation of the wave maps equation and use the caloric gauge introduced by Tao. As discussed in Section~\ref{s:main-idea}, the construction of this gauge relies on a detailed understanding of the harmonic map heat flow. Therefore, our first order of business will be a small data global well-posedness theory for the harmonic map heat flow from $\Hp^4  \to \NN \subset \R^N$, established in Section~\ref{s:hmhf}. Relying on this theory, we will show how to construct the caloric gauge in Section~\ref{s:cg}. After a brief summary of the dynamic variables and equations of the caloric gauge in Section~\ref{s:caloric-summary}, a more detailed outline of the strategy for proving \eqref{eq:bsc} from \eqref{eq:bs} will be given in Section~\ref{s:outline}.
\subsection{The harmonic map heat flow}   \label{s:hmhf}
In what follows, we will often use the shorthand $L^{p}$ for $L^{p}(\bbH^{4}; \bbR^{N})$.

A map $u: \R^+ \times  \Hp^4 \to \NN$  satisfies the \emph{harmonic map heat flow} if  
\EQ{ \label{eq:hmhf} 
\p_s u  =  \h^{ab} D_{a} \rd_{b} u
}
In local coordinates on the target $\NN$ this becomes 
\EQ{\label{eq:hf} 
\p_s u^k =  \Delta_{\Hp^4} u^k + \h^{a b}  \Gamma^{k}_{ij}(u) \p_a u^i\p_b u^j
}
When the target is embedded  the harmonic map heat flow can be expressed in coordinates as
\EQ{ \label{eq:hfe} 
 \p_s u^K  - \De_{\Hp^4} u^K  =  \h^{a b} S^K_{IJ}(u) \p_a u^I \p_b u^J
}
where $S = S^K_{IJ}$ is the second fundamental form of the embedding $\NN \hookrightarrow  \R^N$. This is equivalent to the intrinsic formulation. 

The foundation for our construction of the caloric gauge is the  following small data global regularity result for~\eqref{eq:hfe}. 

\begin{lem}[Existence/uniqueness of the harmonic map heat flow] \label{l:hmhf}  Let $u_0: \Hp^4 \to  \NN \subset \R^N$ be smooth initial data for~\eqref{eq:hfe} such that 
\EQ{ \label{eq:cs} 
u_0(x)   \to u_{\infty} \in  \NN \mas \bfd(x; \zero) \to  \infty.
}
Then, there exists $ \de_1>0$ with the following property:  If 
\EQ{ \label{eq:hfdata} 
 \| du_0 \|_{H^1( \Hp^4; T \NN)}  = \de < \de_1
 }
 then there exists a smooth global solution $u(s, x)$ to the harmonic map heat flow equation ~\eqref{eq:hfe} with initial data $u_0$, and satisfying the uniform estimate
\begin{multline} \label{eq:hfap}
	\nrm{\Dlt u}_{L^{\infty}_{\ds}(\bbR^{+}; L^{2})} 
	+ \sum_{\alp \in \set{\frac{1}{4}, \frac{1}{2}, \frac{3}{4}}} \nrm{s^{\alp} (-\Dlt)^{\alp+1} u}_{L^{\infty}_{\ds} \cap L^{2}_{\ds} (\bbR^{+}; L^{2})}  \\
	+ \nrm{s (-\Dlt)^{2} u}_{L^{2}_{\ds}(\bbR^{+}; L^{2})} 
	\aleq \nrm{d u \!\! \restriction_{s=0}}_{H^{1}(\bbH^{4}, T \NN)}
\end{multline}
The solution $u$ is unique in the class of maps $u \in C( [0, \infty); H^2( \Hp^4; (\NN, u_\infty)))$ such that the left-hand side of \eqref{eq:hfap} is finite.
Moreover, we have the following uniform estimates on higher $s$-weighted derivatives of $u$, 
\begin{align} 
\nrm{s^{\alp} (-\Dlt)^{\alp+1}  u}_{L^{\infty}_{\ds}(\bbR^{+}; L^{2})}
\aleq_{\alp} & \nrm{d u \!\!\restriction_{s=0}}_{H^{1}(\bbH^{4}; T \NN)}, \quad \forall \alp \geq 0 \label{eq:hfap1} \\
\nrm{s^{\alp} (-\Dlt)^{\alp+1}  u}_{L^{2}_{\ds}(\bbR^{+}; L^{2})}
\aleq_{\alp} & \nrm{d u \!\!\restriction_{s=0}}_{H^{1}(\bbH^{4}; T \NN)}, \quad \forall \alp > 0  \label{eq:hfap2} \\
\nrm{s^{\alp} (-\Dlt)^{\alp} \rd_{s} u}_{L^{\infty}_{\ds} (\bbR^{+}; L^{2})}
\aleq_{\alp} & \nrm{d u \!\!\restriction_{s=0}}_{H^{1}(\bbH^{4}; T \NN)}, \quad \forall \alp \geq 0 \label{eq:hfap3}  \\
\nrm{s^{\alp} (-\Dlt)^{\alp} \rd_{s} u}_{L^{2}_{\ds}(\bbR^{+}; L^{2})}
\aleq_{\alp} & \nrm{d u \!\!\restriction_{s=0}}_{H^{1}(\bbH^{4}; T \NN)}, \quad \forall \alp \geq \frac{1}{2} \label{eq:hfap4} 
\end{align}
\end{lem}

\begin{proof}[Proof of Lemma~\ref{l:hmhf}] 
We work in the extrinsic setting with~\eqref{eq:hfe}. To ease the notation, we will use lower case alphabets (such as $i, j, k$) instead of capital letters to denote the indices $1, \ldots, N$.

\vskip.5em
\noindent{\bf Step 0.}
The proof follows by the usual Picard iteration scheme which is conducted in a bounded subset of the space 
\EQ{
\mathcal{X}_0(\R_+) :=  \big\{ u: u - u_{\infty} \in C^0(\R^+; H^2(\Hp^4;  \R^N ) ): \,  \| u\|_{\mathcal{X}_0( \R^+)} < \infty ,   \\ 
\,  u(0,  \cdot) = u_0 \in H^2( \Hp^4; (\NN, u_{\infty}))  \big\}  
}
where the norm~$\|  \cdot \|_{\mathcal{X}_0(J)}$ for any $s$-interval $J \subset \bbR^{+}$ is defined as 
\EQ{ 
	\nrm{u}_{\calX_{0}(J)} 
	:= &
	\nrm{\Dlt u}_{L^{\infty}_{\ds}(J; L^{2})}+  \nrm{s(-\Dlt)^{2} u}_{L^{2}_{\ds}(J; L^{2})} \label{eq:calX0J}   \\ 
	  & + \sum_{\alp \in \set{\frac{1}{4}, \frac{1}{2}, \frac{3}{4}} } \nrm{s^{\alp} (-\Dlt)^{\alp+1} u}_{L^{\infty}_{\ds} \cap L^{2}_{\ds}(J; L^{2})} 
}

The techniques in the proof are standard and we give a rough sketch introducing the main ideas by establishing the a priori estimates~\eqref{eq:hfap} in the next step. 
Before proceeding with the a priori estimates we clear up a minor point, namely that a solution $u \in \mathcal{X}_0(\R^+)$ to~\eqref{eq:hfe} that we  construct actually lands in $\NN \subset \R^N$ for all $s \in \R^+$, that is, 
\EQ{
 u(s, x) \in \NN  \quad  \forall \,  (s, x) \in [0, \infty) \times \Hp^4
 }
 We argue as in~\cite[Section $5.2$]{Lin-Wang}. Let $u: \R^+ \times \Hp^4 \to  \R^N$ be a solution to~\eqref{eq:hfe} with $u \in \mathcal{X}_0(\R^+)$. Let $\NN_\eps$ be an $\eps$-neighborhood of $\NN$ and let  
 \EQ{
  \pi_{\NN}:  \NN_\eps \to \NN
  }
  be the smooth, closest point projection map, i.e., for $w \in \NN_\eps$ we have 
  \EQ{
   \abs{\pi_\NN (w) - w}  =  \inf_{u \in \NN} \abs{ u - w}
  }
  Note that  for $w \in \NN$, $\na \pi_{\NN}(w) :  \R^N \to T_w\NN$ is an orthogonal projection, and moreover,  $\na^2 \pi_\NN (w): T_w\NN \otimes T_w \NN \to (T_w \NN)^{\perp}$ is precisely the second fundamental form, i.e., $\na^2 \pi_{\NN}(w)(\cdot, \cdot)  = - S(w)(\cdot, \cdot)$. 
  
  Now, set $ \rho( u) := \abs{ \pi_{\NN}(u) - u}^2$, and note that $\rho(u(0, \cdot)) = 0$. We have 
\ant{
 \p_s \rho(u) - \De \rho(u) = - 2  \abs{ \na ( \pi_\NN(u) - u)}^2  - 2 \sum_{j=1}^N  ( \pi_\NN^j(u) - u^j) ( \p_s - \De)( \pi_\NN^j(u) - u^j)
}  
We claim that the last term on the right above is $\equiv  0$. Assuming this last fact, we see that $\p_s \rho(u) - \De \rho(u) \le 0$ and an application of the maximum principle concludes the proof.  To prove the last term is $ \equiv 0$ we compute  
\ant{
( \p_s - \De)( \pi_\NN^j(u) - u^j) &=  \p_\ell  \pi_\NN^j(u)( \p_s u^\ell - \De u^\ell)  - \p_k \p_\ell \pi_\NN^j(u) \p^a u^k \p_a u^\ell    \\
&\quad - S_{\ell k}^j(u) \p^a u^k \p_a u^\ell  \\
& = \p_\ell  \pi_\NN^j(u) S^\ell (u) (du, du) + S_{\ell k}^j(u) \p^a u^k \p_a u^\ell - S_{\ell k}^j(u) \p^a u^k \p_a u^\ell \\
& = 0
}
where above we have used the fact that  $S(u)(du, du) \in (T_u \NN)^{\perp}$ to show that the first term in the third line above vanishes. 

\vskip.5em
\noindent{\bf Step 1: A priori estimates. }
We now prove the a priori estimates~\eqref{eq:hfap}. Let $u(s, x)$ be a smooth solution to~\eqref{eq:hfe} on an $s$-interval $I \subset \R^+$. Define $v(s, x) $ by 
\EQ{
v:=  \De (u - u_{\infty}) = \De u
}
where we view $v(s, x)$ as a vector in $\R^N$. 
Then $v$ satisfies the heat equation 
\EQ{ \label{eq:vh}
 \p_s v - \De v  = G,
 }
 where $G(s, x) \in \bbR^N$ is the vector with components 
 \EQ{ \label{eq:Gdef} 
  G^k :=  \De (S^k_{j \ell}(u)  \na^a u^j \na_a u^\ell)
}
Now, given an $s$-interval $J \subset \R^+$, we define the norms $X(J)$ and $X_{0}(J)$ by 
\EQ{
\nrm{v}_{X_{0}(J)} := & 
\nrm{v}_{L^{\infty}_{\ds}(J; L^{2})} + \nrm{s \Dlt v}_{L^{2}_{\ds}(J; L^{2})}  \\
&+  \sum_{\alp \in \set{\frac{1}{4}, \frac{1}{2}, \frac{3}{4}}} \nrm{s^{\alp} (-\Dlt)^{\alp} v}_{L^{\infty}_{\ds} \cap L^{2}_{\ds} (J; L^{2})} 
\label{eq:X0J} 
}
Note that 
 $\nrm{v}_{X_{0}(J)} = \nrm{u}_{\calX_{0}(J)}$.

By Duhamel's formula, the solution $v$ to \eqref{eq:vh} admits the (componentwise) representation
\begin{equation} \label{eq:vh-duhamel}
	v(s) = e^{s \Dlt} v_{0} + \int_{0}^{s} s' e^{(s - s') \Dlt} G(s') \, \dsp
\end{equation}
where $v_{0} = \Dlt u_{0} = \Dlt (u_{0} - u_{\infty})$.
Let $J \subset \bbR^{+}$ be any $s$-interval extending from $0$. By Lemma~\ref{lem:hh-L2:infty} and Lemma~\ref{lem:hh-L2:2}, we have
\begin{equation} \label{eq:vh-h}
	\nrm{e^{s \Dlt} v_{0}}_{X_{0}(\bbR^{+})} \aleq \nrm{v_{0}}_{L^{2}} \simeq \nrm{d u_{0}}_{H^{1}(\bbH^{4}; T \NN)}
\end{equation}
On the other hand, by Lemma~\ref{lem:ih} it follows that
\begin{equation} \label{eq:vh-ih}
	\nrm{\int_{0}^{s} s' e^{(s - s') \Dlt} G(s') \, \dsp}_{X_{0}(J)}
	\aleq \nrm{s G}_{L^{1}_{\ds} \cap L^{\infty}_{\ds} (J; L^{2})}
\end{equation}
For each fixed $s \in \bbR^{+}$, we claim that the following bound holds for $G$:
\begin{equation} \label{eq:vh-apriori-key}
	\nrm{s G}_{L^{1}_{\ds} \cap L^{\infty}_{\ds} (J; L^{2})} \aleq (1 + \nrm{v}_{X_{0}(J)} + \nrm{v}_{X_{0}(J)}^{2}) \nrm{v}_{X_0(J)}^{2}
\end{equation}
Assuming this bound holds, a usual continuity argument with $\de_1$ sufficiently small leads to the \emph{a priori} estimate
\EQ{
\| v\|_{X_{0}(\bbR^{+})}  \lesssim  \|v(0)\|_{L^2} 
}
which implies~\eqref{eq:hfap}. 

We now turn to the proof of \eqref{eq:vh-apriori-key}. From the definition of the norm $X_0(J)$, it is clear that \eqref{eq:vh-apriori-key} would follow once we establish
\begin{equation} \label{eq:vh-apriori-fix-s}
	\nrm{G}_{L^{2}}
	\aleq (1 + \nrm{v}_{L^{2}} + \nrm{v}_{L^{2}}^{2} ) \nrm{(-\Dlt)^{\frac{1}{2}} v}_{L^{2}}^{2} + \nrm{(-\Dlt)^{\frac{1}{4}} v}_{L^{2}} \nrm{(-\Dlt)^{\frac{3}{4}} v}_{L^{2}}
\end{equation}
for every fixed $s \in \bbR^{+}$.
We begin by expanding $G$ as defined in~\eqref{eq:Gdef}.  
\begin{align*}
	G^{k}
	= & \bfh^{c d} \nb_{c} \nb_{d} S^{k}_{ij}(u) \bfh^{ab} \nb_{a} u^{i} \nb_{b} u^{j} \\
	= & \rd^{2}_{m \ell} S^{k}_{i j} \bfh^{cd} \bfh^{ab} \nb_{c} u^{m} \nb_{d} u^{\ell} \nb_{a} u^{i} \nb_{b} u^{j} \\
	& + \rd_{\ell} S^{k}_{ij}(u) \Dlt u^{\ell} \bfh^{ab} \nb_{a} u^{i} \nb_{b} u^{j}
	+ 4\rd_{\ell} S^{k}_{ij}(u) \bfh^{cd} \bfh^{ab} \nb_{d} u^{\ell} \nb_{c} \nb_{a} u^{i} \nb_{b} u^{j} \\
	& + 2 S^{k}_{ij}(u) \bfh^{cd} \bfh^{ab} \nb_{c} \nb_{a} u^{i} \nb_{d} \nb_{b} u^{j}
	+ 2 S^{k}_{ij}(u) \bfh^{cd} \bfh^{ab} \nb_{c} \nb_{d} \nb_{a} u^{i} \nb_{b} u^{j}
\end{align*}
We express the above schematically as 
\ant{
G &=  \p^{(2)} S(u)(  \na u, \na u, \na u, \na u) +  \p S(u)( \De u, \na u, \na u) +  \p S(u)( \na u, \na^{2} u, \na u) \\
&\quad +  S(u) ( \na^{2} u, \na^{2} u) + S(u)( \na^{3} u, \na u)
}
We now take the $L^{2}$ norm and estimate each term as follows. Below we will often use the assumption that the second fundamental form $S$ and its derivatives are bounded, i.e., $ \abs{S(u)} + \abs{ \p S(u)} + \abs{ \p^{(2)} S(u)} \le C$.
Using the Gagliardo-Nirenberg and Sobolev inequalities, we estimate
\begin{align*}
	\nrm{\rd^{(2)} S(u) (\nb u, \nb u, \nb u, \nb u)}_{L^{2}}
	\aleq \nrm{\nb u}_{L^{8}}^{4}
	\aleq \nrm{\nb u}_{L^{4}}^{2} \nrm{\Dlt u}_{L^{4}}^{2}
	\aleq \nrm{v}_{L^{2}}^{2} \nrm{\nb v}_{L^{2}}^{2} 
\end{align*}
Similarly, 
\begin{align*}
	\nrm{\rd S(u) (\Dlt u, \nb u, \nb u)}_{L^{2}}
	\aleq \nrm{\Dlt u}_{L^{4}} \nrm{\nb u}_{L^{8}}^{2}
	\aleq \nrm{\Dlt u}_{L^{4}}^{2} \nrm{\nb u}_{L^{4}} 
	\aleq \nrm{v}_{L^{2}} \nrm{\nb v}_{L^{2}}^{2}
\end{align*}
Exactly the same argument applies to $\rd S(u) (\nb u, \nb^{2} u, \nb u)$. By the Sobolev inequality, it also follows that
\begin{align*}
	\nrm{S(u)(\nb^{2} u, \nb^{2} u)}_{L^2}
	\aleq & \nrm{\Dlt u}_{L^{4}}^{2} \aleq \nrm{\nb v}_{L^{2}}^{2}  \\
	\nrm{S(u) (\nb^{3} u, \nb u)}_{L^2}
	\aleq & \nrm{\nb^{3} u}_{L^{\frac{8}{3}}} \nrm{\nb u}_{L^{8}}
	\aleq \nrm{(-\Dlt)^{\frac{3}{4}} v}_{L^{2}} \nrm{(-\Dlt)^{\frac{1}{4}} v}_{L^{2}}
\end{align*}
These estimates imply \eqref{eq:vh-apriori-fix-s} as desired.


\vskip.5em
\noindent {\bf Step 2: Existence and uniqueness. }
To establish existence of solutions one can now set up the usual iteration argument in the space 
\EQ{
\mathcal{X}_{0, \de} := \{ u   \in \mathcal{X}_0(\R^+) \,  :  \, \|  u\|_{\mathcal{X}_0 (\R^+)}  \le C \de\}
}
Define a sequence of maps 
by setting $u_{-1}(s, x):= u_{\infty}$ and  $u_{j}$ for $j \ge 0$ to be the solution to the linear heat equation 
\EQ{
 (\p_s - \De) u_j &= S(u_{j-1})( d u_{j-1},  d u_{j-1}) \\
 u_j(0, x)&= u_0
 }
 One now shows using the same ideas used to prove~\eqref{eq:hfap} that the sequence is well defined in $\mathcal{X}_{0, \de}$ and Cauchy in $\mathcal{X}_{0, \de}$ for $\de$ small enough.  The uniqueness of solutions to~\eqref{eq:hfe} in the space $\calX_{0}(\bbR^{+})$, as well as smoothness of $u$ on $\bbR^{+} \times \bbH^{4}$ (given that $u_{0}$ is smooth) follow from a standard argument which we do not carry out here. 

\vskip.5em
\noindent {\bf Step 3: Smoothing bounds. }
Next, we turn to the bounds~\eqref{eq:hfap1} and \eqref{eq:hfap2}. The case $\alp < 1$ follows from Lemma~\ref{lem:hh-L2:infty}, Lemma~\ref{lem:hh-L2:2}, Lemma~\ref{lem:ih} and \eqref{eq:vh-apriori-key}. To treat higher derivatives, we apply Lemma~\ref{lem:ih-reg} with $k \geq 1$. We sketch the case $k = 1$ and leave the general case $k \geq 1$ to the reader.

Let $J \subset \bbR^{+}$ be any interval extending from $0$. Applying Lemma~\ref{lem:ih-reg} with $k = 1$ to \eqref{eq:vh}, it follows that
\begin{equation*}
	\nrm{s^{\frac{3}{2}} (-\Dlt)^{\frac{3}{2}} v}_{L^{\infty}_{\ds}(J; L^{2})}
	+ \nrm{s^{2} (-\Dlt)^{2} v}_{L^{2}_{\ds}(J; L^{2})} 
	\aleq \nrm{s \Dlt v}_{L^{2}_{\ds}(J; L^{2})} + \nrm{s \Dlt G}_{L^{2}_{\ds}(J; L^{2})}
\end{equation*}
The first term on the right-hand side is bounded by $\nrm{d u_{0}}_{H^{1}(\bbH^{4}; T \NN)}$ thanks to Step 1.
Proceeding as in the proof of \eqref{eq:vh-apriori-key} and using interpolation, it can be checked that
\begin{align*}
	\nrm{s \Dlt G}_{L^{2}_{\ds}(J; L^{2})}
	\aleq \bb( \nrm{v}_{X_{0}(J)} + \nrm{v}_{X_{0}(J)}^{5} \bb) \sum_{\alp \in \set{\frac{1}{4}, \frac{1}{2}, \frac{3}{4}} } \nrm{s^{1+\alp}(-\Dlt)^{1+\alp} v}_{L^{2}_{\ds}(J; L^{2})}
\end{align*}
By interpolation and Young's inequality, for every $\ti\de > 0$ there exists $C_{\ti\de} > 0$ such that
\begin{equation*}
	\sum_{\alp \in \set{\frac{1}{4}, \frac{1}{2}, \frac{3}{4}} } \nrm{s^{1+\alp}(-\Dlt)^{1+\alp} v}_{L^{2}_{\ds}(J; L^{2})}
	\leq \ti\de \nrm{s^{2} (-\Dlt)^{2} v}_{L^{2}_{\ds}(J; L^{2})} + C_{\ti\de} \nrm{v}_{X_{0}(J)}
\end{equation*}
Recall that $\nrm{v}_{X_{0}(J)} \aleq \nrm{d u_{0}}_{H^{1}(\bbH^{4}; T \NN)}$ by Step 1.  Choosing $\ti\de > 0$ sufficiently small depending on $\nrm{d u_{0}}_{H^{1}(\bbH^{4}; T \NN)}$, we may absorb the first term to the left-hand side, provided it is finite. By a standard continuity argument in $s$ (which ensures the required finiteness), we conclude that
\begin{equation*}
	\nrm{s^{\frac{3}{2}} (-\Dlt)^{\frac{3}{2}} v}_{L^{\infty}_{\ds}(\bbR^{+}; L^{2})}
	+ \nrm{s^{2} (-\Dlt)^{2} v}_{L^{2}_{\ds}(\bbR^{+}; L^{2})} 
	\aleq \nrm{d u_{0}}_{H^{1}(\bbH^{4}; T \NN)}
\end{equation*} 

\vskip.5em
\noindent {\bf Step 4: Regularity in $s$. }
Finally, we establish \eqref{eq:hfap3} and \eqref{eq:hfap4}. The idea is simply to use the harmonic map heat flow equation
\begin{equation*}
	\rd_{s} u - \Dlt u = \tilde{G}
\end{equation*}
where $\tilde{G}$ is defined as
\begin{equation*}
\tilde{G}^{k} = S^{k}_{ij} (u) \nb^{a} u^{i} \nb_{a} u^{j}
\end{equation*}
Let $r = 2$ or $\infty$. Applying $s^{\alp} (-\Dlt)^{\alp}$ to the equation and using $v = \Dlt v$, we derive
\begin{align*}
	\nrm{s^{\alp} (-\Dlt)^{\alp} \rd_{s} u}_{L^{r}_{\ds}(\bbR^{+}; L^{2})}
	\leq \nrm{s^{\alp} (-\Dlt)^{\alp} v}_{L^{r}_{\ds}(\bbR^{+}; L^{2})}
	+ \nrm{s^{\alp} (-\Dlt)^{\alp} \tilde{G}}_{L^{r}_{\ds}(\bbR^{+}; L^{2})}
\end{align*}
The first term on the right-hand side is acceptable; indeed, by the previous step,
\begin{equation*}
\nrm{s^{\alp} (-\Dlt)^{\alp} v}_{L^{r}_{\ds}(\bbR^{+}; L^{2})} \aleq_{\alp} \nrm{d u_{0}}_{H^{1}(\bbH^{4}, T \NN)}
\end{equation*}
for every $\alp \geq 0$ and $r = \infty$, or $\alp > 0$ and $r = 2$. On the other hand, the second term can be bounded by $C_{\alp} \nrm{d u_{0}}_{H^{1}(\bbH^{4}, T \NN)}^{2}$ for $\alp = 2, 3, \ldots$, by using the Leibniz rule to expand $s^{\alp} (-\Dlt)^{\alp} \tilde{G} = s^{\alp} (-\Dlt)^{\alp-1} G$ and applying the bounds \eqref{eq:hfap1} and \eqref{eq:hfap2}. By interpolation, it remains to handle the cases with the least number of derivatives. Using Gagliardo-Nirenberg and Sobolev, we have
\begin{align*}
\nrm{\tilde{G}}_{L^{\infty}_{\ds}(\bbR^{+}; L^{2})}
\aleq & \nrm{\nb u}_{L^{\infty}_{\ds}(\bbR^{+}; L^{4})}^{2}
\aleq \nrm{v}_{L^{\infty}_{\ds}(\bbR^{+}; L^{2})}^{2} \\
\nrm{s^{\frac{1}{2}} (-\Dlt)^{\frac{1}{2}} \tilde{G}}_{L^{2}_{\ds}(\bbR^{+}; L^{2})} 
\aleq & \nrm{s^{\frac{1}{2}} \nb \tilde{G}}_{L^{2}_{\ds}(\bbR^{+}; L^{2})} \\
\aleq & \nrm{\nrm{\nb u}_{L^{6}}^{3}}_{L^{2}_{\ds}(\bbR^{+})} 
	+ \nrm{\nrm{\nb u}_{L^{4}} \nrm{\nb^{2} u}_{L^{4}}}_{L^{2}_{\ds}} \\
\aleq &\bb( \nrm{v}_{L^{\infty}_{\ds}(\bbR^{+}; L^{2})} + \nrm{v}_{L^{\infty}_{\ds}(\bbR^{+}; L^{2})}^{2} \bb) \nrm{(-\Dlt)^{\frac{1}{2}} v}_{L^{2}_{\ds}(\bbR^{+}; L^{2})} ,
\end{align*}
both of which are acceptable. \qedhere
\end{proof}

As is usual for proofs using Picard iteration, the solution map $u_{0} \mapsto u$ is smooth with respect to the norms in the estimates \eqref{eq:hfap1}--\eqref{eq:hfap4}. We record precise bounds in the following statement.

\begin{lem}[Linearized harmonic map heat flow] \label{lem:lin-hm}
Let $I \subset \bbR$ be an interval, and let $u_{0} = (u_{0, t})_{t \in I}$ be a smooth 1-parameter family of maps $\bbH^{4} \to \NN \subset \bbR^{N}$. There exists $\de_2>0$ independent of $ t \in I$ and with $\de_2< \de_1$  where $\dlt_{1}$ is the small constant in Lemma~\ref{l:hmhf}, so that the following holds: Suppose that 
\begin{equation*}
 	u_{0, t}(x) \to u_{\infty} \in \NN \hbox{ as } \bfd(x; \zero) \to \infty, \quad
	\nrm{d u_{0, t}}_{H^{1}(\bbH^{4}; T \NN)}  = \de < \dlt_{2}
\end{equation*}
for every $t \in I$, and 
let $u = u_{t}(s, x)$ be the global solution to \eqref{eq:hmhf} with initial data $u_{0, t}$. Then $u$ is smooth in $(t, s, x) \in I \times \bbR^{+} \times \bbH^{4}$ and obeys the following bounds for each $t \in I$:
\begin{align}
	\nrm{s^{\alp} (-\Dlt)^{\alp+\frac{1}{2}} \rd_{t} u}_{L^{\infty}_{\ds}(\bbR^{+}; L^{2})}
	\aleq_{\dlt, \alp} & \nrm{\rd_{t} u \!\! \restriction_{s=0}}_{H^{1}(\bbH^{4}; T \NN)} , \quad \forall \alp \geq 0 \label{eq:lin-hm1} \\
	\nrm{s^{\alp} (-\Dlt)^{\alp+\frac{1}{2}} \rd_{t} u}_{L^{2}_{\ds}(\bbR^{+}; L^{2})}
	\aleq_{\dlt, \alp} & \nrm{\rd_{t} u \!\! \restriction_{s=0}}_{H^{1}(\bbH^{4}; T \NN)} , \quad \forall \alp > 0 \label{eq:lin-hm2} \\
	\nrm{s^{\alp} (-\Dlt)^{\alp-\frac{1}{2}} \rd_{t} \rd_{s} u}_{L^{\infty}_{\ds} (\bbR^{+}; L^{2})}
	\aleq_{\dlt, \alp} & \nrm{\rd_{t} u \!\! \restriction_{s=0}}_{H^{1}(\bbH^{4}; T \NN)} , \quad \forall \alp \geq 0  \label{eq:lin-hm3} \\	
	\nrm{s^{\alp} (-\Dlt)^{\alp-\frac{1}{2}} \rd_{t} \rd_{s} u}_{L^{2}_{\ds}(\bbR^{+}; L^{2})}
	\aleq_{\dlt, \alp} & \nrm{\rd_{t} u \!\! \restriction_{s=0}}_{H^{1}(\bbH^{4}; T \NN)} , \quad \forall \alp \geq \frac{1}{2}  \label{eq:lin-hm4}
\end{align}

\end{lem}

\begin{rem} \label{r:extra-dlt}
For technical convenience we have introduced a new small constant $\de_2>0$ in the statement of Lemma~\ref{lem:lin-hm}, which is acceptable thanks to the small data assumption. We remark that this  can be avoided by more carefully choosing a function space for the harmonic maps heat flow $u$ which is divisible in $s \in \R^+$. Here divisibility of a space $\calY( J)$ refers to the fact that for any function $u \in \calY( J)$ there exists a partition  of $J$ into a controlled number of subintervals $\{J_k\}$ 
on each of which the $\calY(J_k)$ norm of $u$ is small. 
\end{rem} 

Let $v' = \rd_{t} u$. Then $v'$ obeys the equation
\begin{equation} \label{eq:vph}
 	(\rd_{s} - \Dlt) v' = G'
\end{equation} 
where $G'$ is given by
\begin{align*}
	(G')^{k} 
	=& \rd_{t} (S^{k}_{ij}(u) \nb^{a} u^{i} \nb_{a} u^{j}) \\
	=& \rd_{\ell} S^{k}_{ij}(u) (v')^\ell \nb^{a} u^{i} \nb_{a} u^{j} 
		+ 2 S^{k}_{ij}(u) \nb^{a} (v')^{i} \nb_{a} u^{j}
\end{align*}
Hence $v'$ solves a \emph{linearized harmonic map heat flow} about $u$. 
Lemma~\ref{lem:lin-hm} is proved by analyzing this flow.

\begin{proof} [{Proof of Lemma~\ref{lem:lin-hm}}]
Fix $t \in I$. We begin by showing that for $\delta<\delta_2$ where $\delta_2$ is independent of $t$ and $I$ the following analogue of \eqref{eq:hfap} holds:
\begin{equation} \label{eq:lin-hm-apriori}
	\nrm{(-\Dlt)^{\frac{1}{2}} v' }_{X_{0}(\bbR^{+})} \aleq_{\dlt} \nrm{\rd_{t} u_{0}}_{H^{1}(\bbH^{4}; T \NN)}
\end{equation}
Indeed, applying Duhamel's principle to \eqref{eq:vph},
\begin{equation*}
	v'(s) = e^{s \Dlt} v'(0) +  \int_{0}^{s} s' e^{(s - s') \Dlt} G'(s') \, \dsp
\end{equation*}
Let $J\subseteq \R^+$ be any fixed $s$-interval. By Lemma~\ref{lem:hh-L2:infty} and Lemma~\ref{lem:hh-L2:2}, we have
\begin{equation*}
\nrm{(-\Dlt)^{\frac{1}{2}} e^{s \Dlt} v'(0)}_{X_{0}(J)} \aleq \nrm{(-\Dlt)^{\frac{1}{2}} v'(0)}_{L^{2}} \simeq \nrm{\rd_{t} u_{0}}_{H^{1}(\bbH^{4}; T \NN)}
\end{equation*}
which is acceptable. On the other hand, by Lemma~\ref{lem:ih}, we have
\begin{equation*}
	\nrm{(-\Dlt)^{\frac{1}{2}} \int_{0}^{s} s' e^{(s - s') \Dlt} G'(s') \, \dsp}_{X_{0}(J)}
	\aleq \nrm{\nb G'}_{L^{1}_{\ds} \cap L^{2}_{\ds}(J; L^{2})}
\end{equation*}
We claim that
\begin{equation} \label{eq:lin-hm-apriori-key}
	\nrm{\nb G'}_{L^{1}_{\ds} \cap L^{\infty}_{\ds}(J; L^{2})}
	\aleq (1 + \nrm{u}_{\calX_{0}(J)}^{2}) \nrm{u}_{\calX_0(J)} \nrm{(-\Dlt)^{\frac{1}{2}} v'}_{X_{0}(J)}
\end{equation}
Note the linear factor $\nrm{(-\Dlt)^{\frac{1}{2}} v'}_{X_{0}(J)}$.  Choosing $\delta_2<\delta_1$ we have $\nrm{u}_{\calX_{0}(J)} \aleq \dlt$ by Lemma~\ref{l:hmhf}. Choosing $\dlt_{2}$ smaller if necessary, we may prove
\begin{align*}
	\nrm{(-\Dlt)^{\frac{1}{2}} v'}_{X_{0}(J)}
	\leq &C  \nrm{\rd_{t} u_{0}}_{H^{1}(\bbH^{4}; T \NN)} 
	+ \frac{1}{2} \nrm{(-\Dlt)^{\frac{1}{2}} v'}_{X_{0}(J)}
\end{align*}
Note that the last term can be absorbed into the left-hand side which proves the desired a priori estimate \eqref{eq:lin-hm-apriori}. 

The estimate \eqref{eq:lin-hm-apriori-key} follows by expanding (schematically)
\begin{align*}
	\nb G' = & \rd^{(2)} S(u) (v', \nb u, \nb u, \nb u) + \rd S(u) (\nb v', \nb u, \nb u) 
	+ \rd S(u) (v', \nb^{2} u, \nb u)  \\
	& + S(u) (\nb^{2} v', \nb u) + S(u) (\nb v', \nb^{2} u)
\end{align*}
and treating each term using Gagliardo-Nirenberg and Sobolev, in the fashion of Step 1 in the proof of Lemma~\ref{l:hmhf}. We leave the routine details to the reader.

We now sketch the rest of the proof. Once the key a priori estimate \eqref{eq:lin-hm-apriori} is proved, the higher regularity bounds \eqref{eq:lin-hm1} and \eqref{eq:lin-hm2} can be established using Lemma~\ref{lem:ih-reg} as in the proof of \eqref{eq:hfap1} and \eqref{eq:hfap2}. Moreover, the bounds \eqref{eq:lin-hm3} and \eqref{eq:lin-hm4} for $\rd_{s} \rd_{t} u = \rd_{s} v'$ follow from the equation \eqref{eq:vph} as in the proof of \eqref{eq:hfap3} and \eqref{eq:hfap4}.  \qedhere
\end{proof}

Proceeding as in Lemma~\ref{lem:lin-hm} but commuting the harmonic map heat flow equation with $-\Dlt$, we obtain a priori estimates for $u$ with one less power of $s$ in terms of the $H^{3}$ norm of $d u \rest_{s = 0}$. Considering a smooth one parameter family of maps $u_{t}$ and taking $\rd_{t}$, we also get similar estimates for $\rd_{t} u$ in terms of the $H^{3}$ norm of $\rd_{t} u \rest_{s=0}$. We record these bounds in the following lemma. Such bounds will be useful later in the context of upgrading a priori estimates to well-posedness of the wave map equation.

\begin{lem}[A priori estimate for higher derivatives] \label{l:hrhf} 
Let $u_{0}$ be a smooth map $\bbH^{4} \to \NN$ satisfying the hypothesis of Lemma~\ref{l:hmhf}. Assume furthermore that
\begin{equation*}
	\nrm{d u_{0}}_{H^{3}(\bbH^{4}; T \NN)} < \infty
\end{equation*}
Then the solution $u$ to the harmonic map heat flow with initial data $u_{0}$ given by Lemma~\ref{l:hmhf} obeys the following estimates.
\begin{align*}
	\nrm{s^{\alp -1} (-\Dlt)^{\alp+1} u}_{L^{\infty}_{\ds}(\bbR^{+}; L^{2})}
	\aleq_{\dlt, \alp} & \nrm{d u \!\! \restriction_{s=0}}_{H^{3}(\bbH^{4}; T \NN)}, \quad \forall \alp \geq 1 \\
	\nrm{s^{\alp -1} (-\Dlt)^{\alp+1} u}_{L^{2}_{\ds}(\bbR^{+}; L^{2})}
	\aleq_{\dlt, \alp} & \nrm{d u \!\! \restriction_{s=0}}_{H^{3}(\bbH^{4}; T \NN)}, \quad \forall \alp > 1  \\
	\nrm{s^{\alp -1} (-\Dlt)^{\alp} \rd_{s} u}_{L^{\infty}_{\ds}(\bbR^{+}; L^{2})}
	\aleq_{\dlt, \alp} & \nrm{d u \!\! \restriction_{s=0}}_{H^{3}(\bbH^{4}; T \NN)}, \quad \forall \alp \geq 1 \\
	\nrm{s^{\alp -1} (-\Dlt)^{\alp} \rd_{s} u}_{L^{2}_{\ds}(\bbR^{+}; L^{2})}
	\aleq_{\dlt, \alp} & \nrm{d u \!\! \restriction_{s=0}}_{H^{3}(\bbH^{4}; T \NN)}, \quad \forall \alp \geq \frac{3}{2} 
\end{align*}

Let $I \subset \bbR$ be an interval and consider now a smooth 1-parameter family $(u_{0, t})_{t \in I}$ of maps $\bbH^{4} \to \NN$ satisfying the hypothesis of Lemma~\ref{lem:lin-hm}. Assume furthermore that 
\begin{equation*}
	\nrm{\rd_{t} u_{0, t}}_{H^{3}(\bbH^{4}; T \NN)} < \infty
\end{equation*}
for each $t \in I$. Then for each $t \in I$, the solution $u = u_{t}$ to the harmonic map heat flow with initial data $u_{0, t}$ given by Lemma~\ref{l:hmhf} obeys the following estimates.
\begin{align*}
	\nrm{s^{\alp - 1} (-\Dlt)^{\alp+\frac{1}{2}} \rd_{t} u}_{L^{\infty}_{\ds}(\bbR^{+}; L^{2})}
	\aleq_{\dlt, \alp} & \nrm{\rd_{t} u \!\! \restriction_{s=0}}_{H^{3}(\bbH^{4}; T \NN)} , \quad \forall \alp \geq 1 \\
	\nrm{s^{\alp - 1} (-\Dlt)^{\alp+\frac{1}{2}} \rd_{t} u}_{L^{2}_{\ds}(\bbR^{+}; L^{2})}
	\aleq_{\dlt, \alp} & \nrm{\rd_{t} u \!\! \restriction_{s=0}}_{H^{3}(\bbH^{4}; T \NN)} , \quad \forall \alp > 1 \\
	\nrm{s^{\alp - 1} (-\Dlt)^{\alp-\frac{1}{2}} \rd_{t} \rd_{s} u}_{L^{\infty}_{\ds}(\bbR^{+}; L^{2})}
	\aleq_{\dlt, \alp} & \nrm{\rd_{t} u \!\! \restriction_{s=0}}_{H^{3}(\bbH^{4}; T \NN)} , \quad \forall \alp \geq 1 \\
	\nrm{s^{\alp - 1} (-\Dlt)^{\alp-\frac{1}{2}} \rd_{t} \rd_{s} u}_{L^{2}_{\ds}(\bbR^{+}; L^{2})}
	\aleq_{\dlt, \alp} & \nrm{\rd_{t} u \!\! \restriction_{s=0}}_{H^{3}(\bbH^{4}; T \NN)} , \quad \forall \alp \geq \frac{3}{2} \\
\end{align*}
\end{lem}

The proof of Lemma~\ref{l:hrhf} is very similar to that of Lemma~\ref{lem:lin-hm} and is omitted. We note that $\de_1$ and $\de_2$ may need to be adjusted to be slightly smaller. 

\subsection{Construction of the Caloric Gauge} \label{s:cg}
Fix a smooth wave map $u$ defined on a time interval $I$ and satisfying the bootstrap assumption~\eqref{eq:bs} with initial data satisfying~\eqref{eq:sd}. Denote by $$\MM_I :=  I \times \Hp^4.$$ 
By~\eqref{eq:data1} and finite speed of propagation 
we can assume that for each $t  \in I$, $u(t, x)$ tends to the same fixed point $u(t, \infty) \in \NN$, i.e., 
\EQ{ \label{eq:bdt} 
\lim_{\bfd(x, \zero) \to \infty} u(t, x) =: u(t, \infty)  = u_{\infty} \in \NN
}
In the remainder of the paper, we use the convention that $d$, $\nb$, $D$ and $\bfD$ only refer to spatial derivatives (more precisely, restriction to the tangent space of constant $t, s$-surfaces) and use $d_{t,x}$, $\nb_{t,x}$, $D_{t,x}$ and $\bfD_{t,x}$ to denote the corresponding full spacetime derivatives.

The caloric gauge for the map $u(t, x)$ will be constructed in a sequence of steps.

\vskip.5em
\noindent {\bf Step 1.} In the first step we establish parabolic regularity estimates for the harmonic map heat flow resolution of $u(t, x)$.  For each fixed $t \in I$ we use Lemma~\ref{l:hmhf} to construct a solution $u(s, t,  \cdot)$ to the harmonic map heat flow with initial data $u(t, x)$,  
\EQ{ \label{eq:hfet} 
 \p_s u^K(t)  - \De_{\Hp^4} u^K(t)  &=  \h^{a b} S^{K}_{IJ}(u(t)) \p_a u^I(t) \p_b u^J(t),  \\
 u(0, t, x) &= u(t, x).
}
Indeed, we can choose $\eps>0$ small enough in the bootstrap hypothesis ~\eqref{eq:bs} so that 
\EQ{
\sup_{t \in I} \| du(t)\|_{H^1( \Hp^4; T\NN)}  \le 2C_0 \eps < \de_0
}
where $\de_0 = \min\{ \de_1, \de_2\}$ and  $\de_1, \de_2$ are the  small constants from Lemma~\ref{l:hmhf} and Lemma~\ref{lem:lin-hm}. It follows that for each $t \in I$ we have a unique solution $u(s, t, x)$ satisfying the bounds~\eqref{eq:hfap}--\eqref{eq:hfap4}. 
Moreover, by Lemma~\ref{lem:lin-hm}, it follows that the constants in~\eqref{eq:hfap}--\eqref{eq:hfap4} can be taken to be uniform in $t \in I$ and that the map 
\begin{equation*}
	u : \bbR^{+} \times \MM_{I} \to \NN \subset \bbR^{N}
\end{equation*}
is smooth. Moreover, $\rd_{t} u$ obeys the estimates \eqref{eq:lin-hm1}--\eqref{eq:lin-hm4}.

For simplicity of notation we suppress the dependence of the constants on $C_0.$ The reader may verify that any appearance of $C_0$ is accompanied by $\eps$. 

We summarize the $L^2$ estimates for $u(s, t, x)$ that we have established above in the following proposition. 


\begin{prop} \label{p:preg} 
Let $u: \MM_I \to \NN \subset \R^N$ be a smooth wave map with initial data satisfying~\eqref{eq:sd}. Assume the bootstrap assumption~\eqref{eq:bs} and let $u(s, t, x)$ be the harmonic map heat flow resolution of $u(t, x)$ constructed in Step~$1$ above.  For each fixed $t \in I$, $ du$, $ \p_t u$, and $\p_s u$  satisfy the following estimates:
\begin{equation} \label{eq:preg}
\begin{aligned}
\nrm{(\nb d u, \nb \rd_{t} u, \rd_{s} u, (-\Dlt)^{-\frac{1}{2}} \rd_{t} \rd_{s} u)(t)}_{L^{\infty}_{\ds}(\bbR^{+}; L^{2})} \\
\aleq \nrm{(d u, \rd_{t} u) (t)\rest_{s=0}}_{H^{1}\times H^1(\bbH^{4}; T \NN)} \aleq \eps
\end{aligned}
\end{equation}
Moreover, for any integer $k \geq 1$,
\begin{equation} \label{eq:kpreg}
\begin{aligned}
\nrm{s^{\frac{k}{2}} (\nb^{(k+1)} \ud u, \nb^{(k+1)} \rd_{t} u, \nb^{(k)} \rd_{s} u, \nb^{(k-1)} \rd_{t} \rd_{s} u)(t)}_{L^{\infty}_{\ds} \cap L^{2}_{\ds} (\bbR^{+}; L^{2})} \\
\aleq \nrm{(d u, \rd_{t} u) (t)\rest_{s=0}}_{H^{1}\times H^1(\bbH^{4}; T \NN)} \aleq \eps
\end{aligned}
\end{equation}
Lastly, for any real number $\alp > 0$, 
\begin{equation} \label{eq:12preg}
\begin{aligned}
\nrm{s^{\alp} ((-\Dlt)^{\alp+1} u, (-\Dlt)^{\alp+\frac{1}{2}} \rd_{t} u)}_{L^{\infty}_{\ds} \cap L^{2}_{\ds}(\bbR^{+}; L^{2})} \\
+ \nrm{s^{\alp} ((-\Dlt)^{\alp} \rd_{s} u, (-\Dlt)^{\alp-\frac{1}{2}} \rd_{t} \rd_{s} u)}_{L^{\infty}_{\ds}(\bbR^{+}; L^{2})}  \\
\aleq \nrm{(d u, \rd_{t} u) (t)\rest_{s=0}}_{H^{1}\times H^1(\bbH^{4}; T \NN)} \aleq \eps
\end{aligned}
\end{equation}

  \begin{rem}
    Note that the constants in Proposition~\ref{p:preg} are independent of $t$ and $I$. Observe also that several  of the Sobolev norms in Proposition~\ref{p:preg} are expressed entirely in terms of the covariant derivative $\na$ on $\Hp^4$ in contrast to the estimates in Lemma~\ref{l:hmhf} and Lemma~\ref{lem:lin-hm}, which involve powers of the Laplacian $(-\De)$. However, this distinction makes no difference due to the remarks in Section~\ref{s:fs}.  
    \end{rem}  
\end{prop}

The following  pointwise-in-$s$ decay properties of $u(s, t, x)$ and its derivatives as $s \to \infty$ are an easy consequence of Proposition~\ref{p:preg}.  

\begin{lem} \label{l:udec} 
Let $u(s, t, x)$ be defined as above.  Then for any integer $L \in \N$ we have 
\begin{align}  \label{eq:usinf} 
\|u(s, t, x)  - u_{\infty} \|_{L^{\infty}_{t, x}( \MM_I)}  \lesssim_L \eps s^{-L} 
 \mas s \to \infty
\end{align}
Moreover, we have the following decay estimates for higher derivatives.  
For each $j  = 1,   \dots N$   and for any integers $ \ell, k,  L  \ge 0 $ we have 
\begin{align} \label{eq:dusdec}
&\| du^j(s, \cdot, \cdot)\|_{L^{\infty}_{t, x}(\MM_I)} \lesssim_L  \eps s^{-L},  
\mas s \to \infty
 \\ 
 & \| \p_t u^j(s, \cdot, \cdot)\|_{L^{\infty}_{t, x}(\MM_I)} \lesssim_L  \eps s^{-L},   \label{eq:utdec}
\mas s \to \infty
 \\
& \label{eq:dusdec1} \|  \p_s^k \De^{\ell} u^j(s, \cdot,\cdot)\|_{L^{\infty}_{t, x}(\MM_I)} + \|  \p_s^k \De^{\ell} \p_t u^j(s, \cdot,\cdot)\|_{L^{\infty}_{t, x}(\MM_I)}
\lesssim_L   \eps s^{- L }  
\end{align} 
as $s \to \infty$. 
Furthermore for any $\ell, L \in \N$
\EQ{ \label{eq:Ddus} 
\| D^{(\ell)}  (d u)(s)  \|_{L^{\infty}_{t, x}} \lesssim_L \eps s^{-L} \mas s \to \infty
}
where $D^{(\ell)}$ is the $\ell$-th covariant derivative on $T^* \MM_I  \otimes u^*T\NN $. 
\end{lem}
\begin{proof}
The estimates~\eqref{eq:usinf}, ~\eqref{eq:dusdec}, and~\eqref{eq:utdec} are consequences of the Gagliardo-Nirenberg inequality, Sobolev embedding, and the Poincar\'e inequality on hyperbolic space followed by the parabolic regularity estimates~\eqref{eq:kpreg}. For example,  for any positive integer $L \ge 1$ we have 
 \ant{
\| u(s, \cdot, \cdot)- u_{\infty}\|_{L^{\infty}_{t, x}} &\lesssim_L s^{-L}\| s^L\De^{L} \De u(s)\|_{L^2}  \lesssim_L \eps s^{-L}
}
The estimate~\eqref{eq:dusdec1} 
is proved in the same way as above 
 together with the fact that $u$ is a smooth solution to 
\ant{
 \p_s u =  \De u + S(u)(du, du)
 }
 to replace derivatives in $s$ with spatial derivatives. Finally, we prove~\eqref{eq:Ddus} by repeatedly using the formula 
 \ant{
D_\be V = \na_\be V + S(u)( \p_\be u, V), \quad V \in \Gamma(u^*T \NN)
 }
 together with the estimates~\eqref{eq:dusdec} or~\eqref{eq:utdec}. 
\end{proof}

\vskip.5em
 \noindent{\bf Step 2.}  The  next step is to choose a dynamic frame $e(s, t, x)$ and associated connection form $A(s, t, x)$ on the pullback bundle $u^* T\NN$ over $ \R^+ \times \MM_I$. 
 Let $e$ be a frame on $u^{\ast} T \NN$. Using the isometric embedding $\NN \subset \bbR^{N}$, we may view $e$ as a collection of $\bbR^{N}$-valued functions $(e_{j})_{j=1}^{n}$ which are orthogonal, of length $1$, and belong to $T_{u} \NN$. As in Section~\ref{s:maps}, we write $(\rd_{s} u, \rd_{\alp} u) = (e_{j} \psi^{j}_{s}, e_{j} \psi^{j}_{\alp})$ and the associated connection 1-form $A = A^{i}_{j}$ on $\bbR^{+} \times \MM_{I}$ is a traceless anti-symmetric $n \times n$ matrix-valued 1-form characterized by
\EQ{
D_X e_j = A(X)^i_j e_i
}
where  $X$  is a vector field on $T (\R^+ \times \MM_I)$. In  coordinates $((s, x^{\al}), ( \frac{\p}{\p s}, \frac{\p}{\p x^{\al}}))$ on $T( \R^+ \times \MM_I)$ we will have $A^{i}_j = A^i_{j, s} ds + A^i_{j, \al} dx^{\al}$, where 
\ant{
A_s := A\Big( \frac{\p}{\p s}\Big), \quad A_\al := A\Big( \frac{\p}{\p x^{\al}}\Big), \quad  \al  = 0, 1, \dots, 4
}
We denote the curvature in the frame $e$ by $F_{\mathbf{a} \mathbf{b}} =   \p_\mathbf{a} A_\mathbf{b} - \p_\mathbf{b} A_\mathbf{a}  + [A_\mathbf{a}, A_\mathbf{b}]$, where $\mathbf{a}, \mathbf{b} = s, t, x^{1}, \ldots, x^{4}$. 
Recall from Section~\ref{s:curv} that $F$ is related to the pullback curvature $\bfR$ by
\begin{equation*}
	F_{\mathbf{a} \mathbf{b}} = \bfR(\psi_{\mathbf{a}}, \psi_{\mathbf{b}})
\end{equation*}
Given such a frame, we can recast~\ref{eq:hmhf} in terms of $e$ as 
\begin{equation*}
	\psi_{s} = \bfh^{a b} \bfD_{a} \psi_{b} = \tr_{\bfh} \bfD \psi
\end{equation*}


Now we define the caloric gauge following Tao~\cite{Tao04, Tao37}.  
\begin{defn}[Caloric gauge with limiting frame $e_{\infty}$] 
 Let $u(s, t, x)$ be the heat flow resolution of $u(t, x)$ constructed in Step~$1$  and as in Lemma~\ref{l:udec}. Let $e_{\infty}$ be any  frame for $T_{u_\infty}\NN$.  A smooth frame $e(s, t, x)$ is a  \emph{caloric gauge} for $u(s, t, x)$ with \emph{limiting frame} $e_{\infty}$ if 
\begin{enumerate}

\item The frame $e(s, t, x)$ satisfies what is referred to as the \emph{heat-temporal} condition, $A_s = 0$ for all $(t, x) \in \MM_I$, or equivalently 
\EQ{ \label{eq:ht} 
D_s e = 0.
}
\item For all $(t, x) \in \MM_I$ we have  
\EQ{ \label{eq:bf} 
e(s, t, x)  \to e_{\infty}  \mas s \to \infty.
}

\end{enumerate} 
\end{defn} 


\begin{prop}[Existence and Uniqueness of the Caloric Gauge relative to $e_{\infty}$] \label{p:cg}  Let $u(s, t, x)$ be the heat flow resolution of $u(t, x)$ as in Step~$1$. Let $e_{\infty}$ be any fixed frame for $T_{u_{\infty}} \NN$. Then there exists a unique caloric gauge $e(s, t, x)$ with limiting frame $e_{\infty}$. Moreover, if we denote by $A_\alp(s, t, x)$ the associated connection form  and $\psi_s, \Psi=\psi_\al dx^\alpha$ the representations of $\p_s u$ and $\p_\al u$ in the frame $e$ we have  
\begin{align}
&\| \psi_s(s)\|_{L^{\infty}_{t, x}} \lesssim_L s^{-L} \mas s \to \infty,  \label{eq:psdec}\\
& \| \Psi(s)\|_{L^{\infty}_{t, x}} \lesssim_L s^{-L} \mas s \to \infty,  \label{eq:padec}\\
& \|A (s) \|_{L^{\infty}_{t, x}} \lesssim_L s^{-L} \mas s \to \infty,  \label{eq:Adec} 
\end{align}
for any $L \in \N$. 
\end{prop}

\begin{proof} 
The boundary condition at $s= \infty$ ensures uniqueness. Indeed, suppose that $e$ and $e'$ are both caloric gauges with boundary frame $e_{\infty}$. Then by~\eqref{eq:ht} we see that $\abs{e_j - e_j'}^2$ is constant in $s$, vanishes at $s = \infty$ and is therefore identically $0$.

Next we prove existence. We will start with an arbitrary frame at $s = 0$, show that it admits a unique extension to $s \in \R^+$ satisfying the heat-temporal condition, and then perform an $s$-independent gauge transform on the extension to ensure that~\eqref{eq:bf} holds. 

Given any $(t, x) \in \MM_I$ and any choice of smooth frame $ \ti e(0, t,x)$ for $u(0)^* T \NN$, the heat-temporal gauge condition is imposed by solving the linear ODE~\eqref{eq:ht} with initial data $\ti e(0, t, x)$. Indeed, as we are viewing $\NN  \hookrightarrow \R^N$ as embedded, we solve 
\EQ{ \label{eq:eode} 
D_s \ti e_j  =  \p_s  \ti e_j + S(u)( \ti e_j, \p_s u) &= 0 \\
\ti e_j \rest_{s=0} &= \ti e_j(0,\cdot, \cdot)
}
for each $j = 1, \dots, n$. We find using Picard iteration a global, smooth, unique solution $\ti e(s,t, x)$ to~\eqref{eq:eode}. Moreover $ \ti e(s, t, x)$ is a frame for $u(s)^* T \NN$ for each $s \in \R^+$ since we have 
\EQ{
\p_s \ang{ \ti e_j, \ti e_k}_{ u^*T \NN} = 2\ang{D_s \ti e_j, \ti e_k}_{u^*T\NN} = 0 
}
and thus $  \ang{  \ti e_j,  \ti e_k}_{ u^*T \NN} = \de_{jk}$ for all $j, k   \in \{1, \dots, n\}$.   Next, using the equation~\eqref{eq:eode}, for any $L \in \N$, 
\ant{
 \| \p_s  \ti e_j(s)\|_{L^{\infty}_{t, x}}   =  \| S(u)( \ti e_j, \p_su)\|_{L^{\infty}_{t, x}} \lesssim  \| \p_s u(s)\|_{L^{\infty}_{t, x}} \lesssim_L s^{-L}.
 }
This implies that $ \ti e(s, t, x)$ converges to a limiting frame $ \ti e_{\infty}( t, x)$ as $ s \to \infty$. Now, let $ \ti A_\al$, denote the connection form associated to $ \ti e$ and let  $ \ti \psi_s$ and $\ti \Psi=\ti \psi_\al dx^\alpha$ denote the representations of $\p_s s$ and $\p_\al u$ in the frame $\ti e$. Note that 
\EQ{ \label{eq:tipdec} 
| \ti \psi_s| = |  \ti e \ti \psi_s|  =  \abs{ \p_s u} \lesssim_L s^{-L} , \quad 
|  \ti \Psi | = | \ti e \ti \Psi|  =  \abs{ ( du,  \p_t u)} \lesssim_L s^{-L}
}
Next we claim that 
\EQ{ \label{eq:FsaB} 
| \ti F_{s \alp}| \lesssim \abs{ \p_s u}  \abs{ \rd_{\alp} u} \lesssim_L s^{-L}
 }
where we recall that $\abs{F_{s \alp}(s, t, x)}$ is the Hilbert-Schmidt norm for $n \times n$ matrices.
Indeed~\eqref{eq:FsaB} is a direct consequence of the bounded geometry of $\NN$ and the identity 
\EQ{
\ti e \ti F_{s \al}  = R(u)( \p_s u, \p_\al u)  \ti e 
}

Now note that by~\eqref{eq:eode} we have 
\EQ{
 0 = D_\al D_s  \ti e = D_s D_\al \ti e +  \ti e  \ti F_{\al s} 
 }
Since $ \ti A_s = 0$ we have $D_s =  \p_s$ and thus 
\EQ{
\abs{ \p_s D_{\alp}  \ti e} =  | \ti F_{s \alp}| \lesssim \abs{ \p_s u}  \abs{ \rd_{\alp} u} \lesssim_L s^{-L}.
}
It follows that 
\EQ{ \label{eq:Dae} 
D_{\alp} \ti e \to D_{\alp} \ti e_{\infty} \mas s \to \infty, \quad \alp = 0, 1, \ldots, 4
}
Again since $ \ti A_s = 0$ we then have $ \p_s  \ti A_\al =  \ti F_{s\al} $ and thus 
\EQ{
|\p_s  \ti A_\al  |  = | \ti F_{s\al}|  \lesssim \abs{ \p_s u} \abs{ \p_\al u} \lesssim_{L} s^{-L} \mas s \to \infty, \quad \alp =0, 1, \ldots, 4
}
This implies that each $ \ti A_{\alp}(s, t, x)$ converges to a limit $\ti A_{\alp}(t, x)$; more succinctly
\EQ{
\ti A(s, t, x)  \to  \ti A_{\infty}(t, x)  \in  \mathfrak{so}(n) \mas s \to \infty.
}
Combining the above with~\eqref{eq:Dae} we can pass to limits in both sides of 
\EQ{
 \ti A_{j, \al}^k = \ang{ D_\al \ti e_j,  \ti e_k}
 }
to deduce that $ \ti A_{\infty}$ is the connection form for $ \ti e_{\infty}$. 

Now find a gauge transformation $B(t, x) \in SO(n)$ so that 
\EQ{
 \ti e_{\infty}(t, x) B(t, x) = e_\infty
 }
 The caloric gauge $e$ is then defined by applying $B(t, x)$ as above at each time $s$, i.e., 
 \EQ{
 e(s, t, x) :=  \ti e(s, t, x) B(t, x) , \quad e_j(s) := \ti e_\ell(s) B^\ell_j.
 }
 This ensures that for any fixed $(t, x) \in \MM_I$, we have the convergence, 
 \EQ{ \label{eq:ec} 
 e(s, t, x) \to e_\infty \mas s \to \infty
 }
 Let $A(s, t, x)$ denote the connection form associated to $e(s, t, x)$. 
Arguing as above and using~\eqref{eq:ec} we must have 
 \EQ{
A(s, t, x) \to 0  \mas s \to \infty 
 }
since the associated connection 1-form to the constant frame $e_{\infty}$ on $\MM_{I}$ is $A_{\infty} = 0$.
To deduce~\eqref{eq:Adec} regarding the rate of this convergence, we again argue as above to see that $\p_s A_\alp  = F_{s \alp}$ and that 
 \EQ{ \label{eq:Fsadec} 
| \ti F_{s \alp}| \lesssim \abs{ \p_s u}  \abs{  \rd_{\alp} u} \lesssim_L s^{-L}
 }
 Hence 
 \EQ{
 \abs{A_{\alp}(s)}   \le\int_{s}^\I \abs{F_{s \alp}(s') } \, ds'  \lesssim_{L+1}  s^{-L} \mas s \to \infty, \quad \alp = 0, 1, \ldots, 4
 }
 
 Finally we note that~\eqref{eq:psdec} and~\eqref{eq:padec} follow from the exact same argument as in~\eqref{eq:tipdec}. 
  \end{proof}

\subsection{Summary of the dynamic variables in the caloric gauge} \label{s:caloric-summary}

Putting the contents of the previous subsection together we record a list of the dynamic variables that have been introduced starting from a wave map $(u(t, x), \p_t u(t, x))$ on a time interval $I$, finding its harmonic heat flow resolution $u(s, t, x)$, and passing to the caloric gauge $e$ with boundary frame $e_{\infty}$. Below $\al  \in \{ 0, 1, \dots, 4\}$, $a \in\{ 1, \dots, 4\}$,  $j  \in \{1, \dots n\}$.
\begin{align}
&\p_s u := e\psi_s = \psi_s^j e_j, \quad 
 \p_\al u  = e \psi_\al , \\
& \psi_s^j(s, t, x)  =  \Db^a \psi_a^j(s, t, x) := \ang{e_k \h^{ab} ( \Db \psi)^k_{a b}, e_j}_{u^* \g}, \\
&  \Db^\al \psi_\al^j (0, t, x) :=\ang{e_k \etab^{\al \be}(\Db \psi)^k_{\al \be}(0, t, x), e_j}_{u^* \g} = 0, \\
& D_\al e =: eA_\al  , \\
&D_s e  = :e A_s = 0, \\
& \Db_\al \psi_\be  = \Db_\beta \psi_\al  \label{eq:abba},  \\
& \p_s \psi_\al  = \Db_s \psi_\al = \Db_\al \psi_s, \label{eq:Dsps} \\
& \p_s A_\al = F_{s \al}, \quad 
	F_{\al \be} = \bfR^{(0)}( \psi_{\alp}, \psi_{\bt}), \quad
	e F_{\al \be} = R(u)(\rd_{\alp} u, \rd_{\bt} u) e  \label{eq:psAF}.
\end{align}

In the remainder of the paper, it will be convenient to have separate notation for the spatial and temporal components of tensors. For the space-time $1$-forms $\psi_\alp$, $A_\alp$, and $F_{s \alpha }$ we introduce the following notation:
\ali{
\Psi =& \psi_\alp \, d x^\alp = \psi_t \, d t + \underline{\Psi}, \quad
\underline{\Psi} = \psi_a  \, d x^a \\
A =& A_\alp \, d x^\alp = A_t \, d t + \underline{A} , \quad
\underline{A} = A_a \, d x^a\\
F_s =& F_{s\alp} \, d x^\alp = F_{st} \, d t + \underline{F}_s , \quad
\underline{F}_s = F_{sa} \, d x^a
}
We remind the reader that the pointwise norm of space-time one forms is defined with the auxiliary metric $\tilde{\etab}$; see \eqref{eq:tilde-eta}. Moreover, for differential operators, we use the convention that $d$, $\nb$, and $\Db$ refer only to spatial components, and $d_{t,x}$, $\nb_{t,x}$, and $\Db_{t,x}$ refer to the whole spatial and temporal components.

We also make note of  the boundary conditions
\EQ{
&e(s, t, x) \to e_\infty  \mas s\to + \infty  \\
& \psi_s(s, t, x), \, \,  \psi_\al(s, t, x), \, \, A_\al(s, t, x)  \longrightarrow 0 \mas s \to + \infty\\
}
Crucially, the above allow us to express $A_\al$ and $\psi_\al$ in terms of $\psi_s$ and $F_{s \al}$ by the fundamental theorem of calculus, 
\EQ{ \label{eq:AF} 
A_\al(s_0, t, x)  = -\int_{s_0}^\infty \p_s A_\al(s, t, x) \, ds   = - \int_{s_0}^\infty F_{s \al}(s) \, ds
}
and 
\EQ{ \label{eq:paps} 
\psi_\al(s_0, t, x) &= -\int_{s_0}^\infty  \p_s \psi_\al \, ds  =  -  \int_{s_0}^\infty \bfD_{\al} \psi_s \, ds \\
& =  - \int_{s_0}^\infty \p_\al \psi_s  \, ds - \int_{s_0}^\infty A_\al \psi_s \, ds \\
& = - \int_{s_0}^\infty \p_\al \psi_s  \, ds + \int_{s_0}^\infty \int_{s}^\infty F_{s \al}(s') \psi_s(s) \, ds' \, ds.
}
To estimate the  above we will often use the pointwise bound 
\EQ{ \label{eq:Fsapw} 
\abs{F_{s \alp}(s)} \lesssim  \abs{\p_s u(s)}  \abs{ (du, \p_t u)(s)}
}
which follows from the identity 
\ant{
e F_{s \al} = R(u)( \p_s u, \p_\al u) e 
}  
and  the bounded geometry of $\NN$. 

\subsection{Equivalence of norms in the caloric gauge} \label{s:eqnorm} 
We now transfer the $L^2$-based parabolic regularity estimates for $\p_\al u$ and $\p_s u$ from Section~\ref{s:cg} to $L^2$ based estimates for $\psi_s$ and $\psi_\al$.

To begin, let $\phi$ be any smooth section of $u^* T\NN$, and let $e$ be any frame for $u^*T\NN$. Set $ \phi =: e  \varphi$. Then,  
\ant{
\| \varphi \|_{L^2( \Hp^4; \R^n)}  = \| \phi \|_{L^2( \Hp^4; \R^N)}.
}
 It follows then, for example, that 
\ant{
\| \psi_s( s, t)\|_{L^2 (\Hp^4; \R^n)} = \| \p_s u( s, t)\|_{L^2(\Hp^4; \R^N)}.
}
To prove similar equivalence of norms statements in the caloric gauge $e$, for higher derivatives of $\p_su $, and $\p_\al u$ we will use the formula
\EQ{ \label{eq:Ddd}
D_\be \phi = \na_\be\phi + S(u)( \p_\be u, \phi)  = e( \na_\be \varphi + A_\be \varphi).
}	
We start with the following proposition. 
 
 \begin{prop}[Equivalence of norms] \label{p:en} 
Let $u(s,t, x)$ be the heat flow resolution of $u(t, x)$ as in Proposition~\ref{p:preg}.  Let $\psi_\al$ and $\psi_s$ be the representations of $\p_\al u$ and $\p_s u$ in the caloric gauge $e$ with limiting frame $e_\infty$ as in Proposition~\ref{p:cg}. Then for each $t \in I$, 
\begin{equation} \label{eq:eqna}
\begin{aligned} 
	\nrm{\nb \Psi(t)}_{L^{\infty}_{\ds}(\bbR^{+}; L^{2}(\mathbb{H}^{4}; \bbR^{N}))}
	\simeq & \nrm{(\nb du, \nb \rd_{t} u)(t)}_{L^{\infty}_{\ds}(\bbR^{+}; L^{2}(\mathbb{H}^{4}; \bbR^{N}))}  \\
	\simeq & \nrm{D d_{t,x} u(t)}_{L^{\infty}_{\ds}(\bbR^{+}; L^{2}(\mathbb{H}^{4}; \bbR^{N}))} 
\end{aligned}
\end{equation}
and
\begin{equation} \label{eq:eqns}
\begin{aligned}
	\nrm{s^{\frac{1}{2}} \nb_{t,x} \psi_{s}(t)}_{L^{\infty}_{\ds} \cap L^{2}_{\ds}(\bbR^{+}; L^{2}(\bbH^{4}; \bbR^{n}))}
	\simeq & \nrm{s^{\frac{1}{2}} \nb_{t,x} \rd_{s} u(t)}_{L^{\infty}_{\ds} \cap L^{2}_{\ds}(\bbR^{+}; L^{2}(\bbH^{4}; \bbR^{N}))} \\
	\simeq & \nrm{s^{\frac{1}{2}} D_{t,x} \rd_{s} u(t)}_{L^{\infty}_{\ds} \cap L^{2}_{\ds}(\bbR^{+}; L^{2}(\bbH^{4}; \bbR^{N}))}
\end{aligned}
\end{equation}
where the implicit constants are independent of $t$ and $I$. 
\end{prop}

From~\eqref{eq:Ddd} it is clear that to compare, say $\na \p_\al  u$ with $\na \psi_\al$, we will need control over the the connection form $A$. We require the following lemma. 
\begin{lem} \label{l:A}
Let $u(s,t, x)$ be the heat flow resolution of $u(t, x)$ as in Proposition~\ref{p:preg}.  Let $e$ be the caloric gauge with limiting frame $e_\infty$ as in Proposition~\ref{p:cg},  and let $A$ denote the corresponding connection form.  Then for each $t \in I$ we have 
\begin{align} \label{eq:AL4}
\| A(t) \|_{\Ls^\I( \R^+; L^4(\bbH^{4}; \mathfrak{so}(n) ))} \lesssim  \| (du, \p_t u)(t) \rest_{s=0} \|_{H^1(\bbH^{4}; T \NN)}^2 \lesssim  \eps^2
\end{align}
We also have the following estimates on the derivatives of $A$ and $\underline{A}$. 
\begin{equation} \label{eq:AL2}
\begin{aligned}
 \| (\nb A, \nb_{t,x} \underline{A})(t)\|_{\Ls^\I( \R^+; L^2(\bbH^{4}; \mathfrak{so}(n)))}
\lesssim  \| (du, \p_t u)(t) \rest_{s=0} \|_{H^1\times H^1(\bbH^{4}; T \NN)}^2 \aleq \eps^{2}
\end{aligned}
\end{equation} 
Moreover, for all $k \ge 1$, we have, 
\begin{equation} \label{eq:dkAL2}
\begin{aligned}
 \| s^{\frac{k}{2}} (\na^{(k+1)}A, \nb^{(k)} \nb_{t,x} \underline{A})(t) \|_{\Ls^\I \cap \Ls^2(\bbR^{+}; L^2(\bbH^{4}; \mathfrak{so}(n)))} \\
 \lesssim_{k}  \| (du, \p_t u)(t) \rest_{s=0} \|_{H^1\times H^1(\bbH^{4}; T \NN)}^2 \aleq \eps^{2}
\end{aligned}
\end{equation}
All constants are independent of $t$ and $I$.
\end{lem}

Given Proposition~\ref{p:en} and Lemma~\ref{l:A} we can also deduce the following bounds on higher derivatives of $\Psi$ and $\psi_s$. 
\begin{cor}\label{c:en}
Let $u(s,t, x)$ be the heat flow resolution of $u(t, x)$ as in Proposition~\ref{p:preg}. Let $\psi_\al$ and $\psi_s$ be the representations of $\p_\al u$ and $\p_s u$ in the caloric gauge $e$ with limiting frame $e_\infty$ as in Proposition~\ref{p:cg}. Then for each $t \in I$ and all $k \ge 2$ 
\EQ{ \label{eq:eqnsk}
 \| s^{\frac{k}{2}} D^{(k-1)} D_{t, x} \p_s u (t) \|_{\Ls^\I \cap \Ls^2(\bbR^{+}; L^{2})} 
   \lesssim_k    \nrm{(d u, \rd_{t} u)(t) \rest_{s=0}}_{H^{1}\times H^1(\bbH^{4}; T \NN)}   \\
 \| s^{\frac{k}{2}} \na^{(k-1)} \na_{t, x} \psi_s (t)\|_{\Ls^\I \cap \Ls^2 (\bbR^{+}; L^{2})}   
    \lesssim_k    \nrm{(d u, \rd_{t} u)(t) \rest_{s=0}}_{H^{1}\times H^1(\bbH^{4}; T \NN)}  
 }
 and for each $t \in I$ and  $k \ge 1$, 
 \EQ{ \label{eq:eqnak}
 \| s^{\frac{k}{2}} D^{(k+1)}  d_{t,x} u(t) \|_{\Ls^\I \cap \Ls^2 (\bbR^{+}; L^{2})}   
     \lesssim_k    \nrm{(d u, \rd_{t} u)(t) \rest_{s=0}}_{H^{1}\times H^1(\bbH^{4}; T \NN)}   \\
 \| s^{\frac{k}{2}} \na^{(k+1)} \Psi(t) \|_{\Ls^\I \cap \Ls^2(\bbR^{+}; L^{2})} 
     \lesssim_k    \nrm{(d u, \rd_{t} u)(t) \rest_{s=0}}_{H^{1}\times H^1(\bbH^{4}; T \NN)}   
 }
\end{cor}

%

We first prove Lemma~\ref{l:A}. 

\begin{proof}[Proof of Lemma~\ref{l:A}]
Fix $t \in I$; in what follows, we suppress writing $t$ with the understanding that every computation is performed at the fixed time $t$. By~\eqref{eq:AF}   we have 
$
A_\al(s_0)  = -\int_{s_0}^\I F_{ s\al }(s)ds
$
for any $s_0 >0$.  It follows from~\eqref{eq:Fsapw} that 
\ant{
\| A(s_0)\|_{L^4} &\lesssim \int_{s_0}^\I \|F_{s}(s)\|_{L^4} ds   \lesssim  \int_{s_0}^\I \|  \p_s u\|_{L^4}\|  d_{t,x} u\|_{L^\infty} \, ds
 \\
& \lesssim  \left(\int_{s_0}^\I s  \| \p_s u\|_{L^4}^2 \, \ds \right)^{\frac{1}{2}} \left( \int_{s_0}^\I  s  \|  d_{t,x} u\|_{L^\infty}^2 \, \ds \right)^{\frac{1}{2}}.
}  
By Sobolev embedding the first integral on the right above is bounded using~\eqref{eq:kpreg} with $k =1$, by  
\ant{
 \left(\int_{s_0}^\I s  \| \p_s u\|_{L^4}^2 \, \ds \right)^{\frac{1}{2}}  &\lesssim  \left(\int_{s_0}^\I   \|s^{\frac{1}{2}} \na \p_s u\|_{L^2}^2 \, \ds \right)^{\frac{1}{2}}   \\
 &\lesssim \| (du, \p_t u)(t) \rest_{s=0} \|_{H^1\times H^1(\bbH^{4}; T \NN)} \lesssim  \eps
 }
 To bound the second integral we use Gagliardo-Nirenberg, Sobolev and $L^{2}$ interpolation, 
 \ant{
 \Bigg( \int_{s_0}^\I   s \|  d u\|_{L^\infty}^2 \,& \ds \Bigg)^{\frac{1}{2}} \lesssim  \left( \int_{s_0}^\I s   \|   d u \|_{L^8} \|  \na  d u \|_{L^8} \, \ds \right)^{\frac{1}{2}} \\
 &\lesssim \left( \int_{s_0}^\I    \|  s^{\frac{1}{4}} (-\De)^{\frac{5}{4}} u \|_{L^2} \| s^{\frac{3}{4}} (-\Delta)^{\frac{7}{4}} u\|_{L^2} \, \ds \right)^{\frac{1}{2}}  \\
&\lesssim \| (du, \p_t u)(t) \rest_{s=0} \|_{H^1\times H^1(\bbH^{4}; T \NN)} \lesssim  \eps
 }
 where the last line follows from~\eqref{eq:kpreg} and~\eqref{eq:12preg}.  
 The contribution of $\p_t u$ to the second integral can be estimated in a similar way.

To prove the second bound we begin with the formula 
\begin{align*}
	\nb_{\bt} A_{\alp}(s_{0})
	= & - \int_{s_{0}}^{\infty} \nb_{\bt} F_{s \alp}(s) \, ds \\
	= & - \int_{s_{0}}^{\infty} \bfD_{\bt} F_{s \alp}(s) \, ds + \int_{s_{0}}^{\infty} [A_{\bt}, F_{s \alp}](s) \, ds 
\end{align*}
where at most one of $\alp$ and $\bt$ is $0$. By \eqref{eq:DF} and \eqref{eq:DR-ptwise} in Section~\ref{s:curv}, as well as $\rd_{s} u = e \psi_{s}$ and $\rd_{\alp} u = e \psi_{\alp}$, the term $\bfD_{\bt} F_{s \alp}$ is pointwisely bounded by 
\begin{align*}
\abs{\bfD_{\bt} F_{s \alp}}
\aleq & \abs{D_{\bt} \rd_{s} u} \abs{\rd_{\alp} u}
	+ \abs{\rd_{s} u}\abs{D_{\bt} \rd_{\alp} u} 
	+ \abs{\rd_{\bt} u} \abs{\rd_{s} u} \abs{\rd_{\alp} u}
\end{align*}
Using $D_{\bt} \rd_{\alp} u = D_{\alp} \rd_{\bt} u$ if necessary, it follows that
\begin{align*}
\abs{\bfD_{t,x} F_{s }} + \abs{[A, F_{s }]}
\aleq & \abs{\nb_{t,x} \rd_{s} u} \abs{d_{t,x} u} + \abs{\rd_{s} u} \abs{\nb d_{t,x} u} \\
& + \abs{\rd_{s} u} \abs{d_{t,x} u}^{2}
	 + \abs{A} \abs{\rd_{s} u} \abs{d_{t,x} u}
\end{align*}
Therefore
\begin{align*}
& \hskip-2em
\nrm{\nb_{t,x} A}_{L^{\infty}_{\ds}(\bbR^{+}; L^{2})} \\
\aleq & \int_{0}^{\infty} s \nrm{\nb_{t,x} \rd_{s} u}_{L^{\frac{8}{3}}} \nrm{d_{t,x} u}_{L^{8}} \, \ds
	+ \int_{0}^{\infty} s \nrm{\rd_{s} u}_{L^{4}} \nrm{\nb d_{t,x} u}_{L^{4}}^{2} \, \ds \\
	& + \int_{0}^{\infty} s \nrm{\rd_{s} u}_{L^{4}} \nrm{d_{t,x} u}_{L^{8}}^{2} \, \ds 
	 + \int_{0}^{\infty} s \nrm{A}_{L^{4}} \nrm{\rd_{s} u}_{L^{8}} \nrm{d_{t,x} u}_{L^{8}} \, \ds \\
\aleq &  \nrm{(du, \rd_{t} u)(t) \rest_{s=0}}_{H^{1}\times H^1(\bbH^{4}; T \NN)} \aleq \eps,
\end{align*}
 where the last line is a consequence of Sobolev embedding, Gagliardo-Nirenberg, the estimates in Proposition~\ref{p:preg} and \eqref{eq:AL4}.  This completes the proof of~\eqref{eq:AL2}. One can argue similarly to establish~\eqref{eq:dkAL2}. \qedhere

 \end{proof}

Next, we establish Proposition~\ref{p:en}. 

\begin{proof}[Proof of Proposition~\ref{p:en}] 
As before, we fix and suppress $t \in I$.
We first prove~\eqref{eq:eqna}. To prove the first $\simeq$ in~\eqref{eq:eqna} we deduce from ~\eqref{eq:Ddd} with $\phi = \p_\al u$ and $\varphi  = \psi_\al$ that 
\ant{ 
& \hskip-1em
\Big| \| (\na d u,\na \p_t u) \|_{\Ls^\infty (\bbR^{+}; L^2)}  - \| \na \Psi\|_{\Ls^\infty (\bbR^{+}; L^2)} \Big|  \\
& \le \|S(u)( du, e \Psi)\|_{\Ls^\infty (\bbR^{+}; L^2)} + \|S(u)( \p_t u, e \Psi)\|_{\Ls^\infty (\bbR^{+}; L^2)} + \| A e \Psi\|_{\Ls^\infty (\bbR^{+}; L^2)} \\
& \lesssim  \Big( \| (du,\p_t u)\|_{\Ls^\infty (\bbR^{+}; L^4)} +  \|A\|_{\Ls^\infty (\bbR^{+}; L^4)} \Big)  \| \Psi\|_{\Ls^\infty (\bbR^{+}; L^4)} \\
& \lesssim  \eps \|  \na \Psi\|_{\Ls^\infty (\bbR^{+}; L^2)}.
}
A similar argument gives the second $\simeq$ in~\eqref{eq:eqna}. The proof of~\eqref{eq:eqns} is identical after including the appropriate factor of $s^{\frac{1}{2}}$. 
\end{proof} 

Finally, we give a brief sketch of the proof of Corollary~\ref{c:en}. 
\begin{proof}[Proof of Corollary~\ref{c:en}]
The proof is essentially the same as the proof of Proposition~\ref{p:en} and is based on equation~\eqref{eq:Ddd}. For example, to prove~\eqref{eq:eqnsk} with $k = 2$ we apply~\eqref{eq:Ddd} with $\phi  = D_{\alp} \p_su$. We have, 
\ant{
D_b(D_\al \p_s u) =&  \na_b  \na_\al  \p_su + \na_{b} ( S(u)( \p_s u, \p_\al u)) + S(u)( D_\al  \p_su, \p_b u) \\
& = e\na_b \na_\al \psi_s + e \na_b(A_\al \psi_s)  + e A_b \na_\al  \psi_s + e A_b A_\al \psi_s.
}
One can then argue as in the proof of Proposition~\ref{p:en} using the bounds from Proposition~\ref{p:preg}, Lemma~\ref{l:A}, and~\eqref{eq:eqns} to conclude.
\end{proof}

\subsection{The wave map system in the caloric gauge: Dynamic equations for $\psi_s$ and $\psi_\al$} \label{s:cd} 

Given the estimates in the previous subsection, where we transferred the $L^2$-based parabolic regularity estimates for the harmonic map heat flow resolution $u(s, t, x)$  to the representations of $\p_\al u $ and $\p_s u$ in the caloric gauge, namely $\psi_\al$ and $\psi_s$, we can now work with the \emph{dynamic} equations for $\psi_\al$ and $\psi_s$ to close the a priori estimates~\eqref{eq:ap}.  We place particular importance on the \emph{scalar} nonlinear wave equation~\eqref{eq:wmp} satisfied by $\psi_s$ since one can \emph{close} a priori estimates for solutions to~\eqref{eq:wmp} using standard analysis based on Strichartz estimates. 

Following~\cite{Tao04}, we refer to $\psi_s$  as the \emph{heat tension field},  
\begin{equation} \label{eq:htf}
\psi_{s}(s,t,x) := \Db^{a} \psi_{a}
\end{equation}
We also define the \emph{wave tension field} $w$ by 
\EQ{\label{eq:wtf} 
w(s, t, x):= \Db^{\alp} \psi_{\alp}
}
We note that since $u(0, t, x)$ is a wave map we have $w(0, t, x) = 0$. 

We  derive a covariant wave equation for the heat tension field $\psi_s$ and a covariant heat equation for the wave tension field $w$.  
A key observation is that both equations are \emph{scalar}, in the sense that the linear parts coincide with the scalar wave and heat equations on $\bbH^{4}$, respectively.
\begin{lem}[Dynamic equations in the caloric gauge: wave equation for $\psi_s$] \label{l:wmp}
Let $\psi_s$ and $w$ be defined as in~\eqref{eq:wtf} and~\eqref{eq:htf}. 
Then the heat tension field $\psi_s$ satisfies 
\EQ{ \label{eq:wmp}
\bfD^{\alp} \bfD_{\alp} \psi_{s} = \rd_{s} w - \tensor{F}{_{s}^{\alp}} \psi_{\alp}
  } 
  where the wave tension field $w$ satisfies 
    \EQ{ \label{eq:hw} 
(\rd_{s} - \bfD^{b} \bfD_{b}) w
&= \tensor{F}{_{s}^{ \alp}} \psi_{\alp} + 3 F^{\alp b} \bfD_{\alp} \psi_{b}
+ \bfD^{\alp} \tensor{F}{_{\alp}^{b}} \psi_{b} + \bfD^{b} \tensor{F}{^{\alpha}_{b}} \psi_{\alp} \\
w(0, t, x) &= 0
}

  
 \end{lem}

\begin{rem} 
In order to estimate $w(s)$, we will solve the inhomogeneous heat equation \eqref{eq:hw} from $0$ to $s$, exploiting the fact that the initial data for $w$ is zero. Observe that all other terms in \eqref{eq:wmp} only involve variables at $s$ or larger. In the Littlewood-Paley analogy discussed in Section~\ref{s:main-idea}, $\rd_{s} w$ can be thought of as the contribution of Littlewood-Paley projections at higher frequencies.
\end{rem}

\begin{proof}
First we prove~\eqref{eq:wmp}. Using the identity~\eqref{eq:Dsps}, the fact that $\Db_s = \p_s$ and that the domain curvature $\RR^{\mu_1}_{ \,  \mu_2 \mu_3\mu_4}$ vanishes if any of the  indices $ \mu_j = s$   we have 
\begin{align*}
	\bfD^{\alp} \bfD_{\alp} \psi_{s}
	= \bfD^{\alp} \bfD_{s} \psi_{\alp}
	= \tensor{F}{^{\alp}_{s}} \psi_{\alp} + \rd_{s} w
\end{align*}

Next, we derive the covariant heat equation \eqref{eq:hw} for the heat tension field $w$ making  use of the identities~\eqref{eq:abba} and~\eqref{eq:Dsps}.  
We have
\begin{align*}
	\rd_{s} w 
	= & \rd_{s} \bfD^{\alp} \psi_{\alp} = \bfD^{\alp} \bfD_{s} \psi_{\alp} + [\bfD_{s}, \bfD^{\alp}] \psi_{\alp} \\
	= & \bfD^{\alp} \bfD_{\alp} \bfD^{b} \psi_{b} + \tensor{F}{_{s}^{\alp}} \psi_{\alp} \\
	= & \bfD^{\alp} \bfD^{b} \bfD_{b} \psi_{\alp} + \bfD^{\alp} [\bfD_{\alp}, \bfD^{b}] \psi_{b}  + \tensor{F}{_{s}^{\alp}} \psi_{\alp} \\
	= & \bfD^{b} \bfD_{b} w +\tensor{F}{_{s}^{\alp}} \psi_{\alp} 
	+ [\bfD^{\alp}, \bfD^{b}] \bfD_{b} \psi_{\alp} + \bfD^{b} [\bfD^{\alp}, \bfD_{b}] \psi_{\alp} + \bfD^{\alp} [\bfD_{\alp}, \bfD^{b}] \psi_{b}  
\end{align*}
By the commutation formula \eqref{eq:commDD}, the last three terms take the form
\begin{align*}
& \hskip-1em
[\bfD^{\alp}, \bfD^{b}] \bfD_{b} \psi_{\alp} + \bfD^{b} [\bfD^{\alp}, \bfD_{b}] \psi_{\alp} + \bfD^{\alp} [\bfD_{\alp}, \bfD^{b}] \psi_{b} \\
= &F^{\alp b} \bfD_{b} \psi_{\alp} + \tensor{\calR}{^{a b}_{b}^{c}} \bfD_{c} \psi_{a} + \tensor{\calR}{^{a b}_{a}^{c}} \bfD_{b} \psi_{c}
	+ \bfD^{b}(\tensor{F}{^{\alp}_{b}} \psi_{\alp} + \tensor{\calR}{^{a}_{b a}^{c}} \psi_{c}) \\
	& + \bfD^{\alp} (\tensor{F}{_{\alp}^{b}} \psi_{b})
	+ \bfD^{a} (\tensor{\calR}{_{a}^{b}_{b}^{c}} \psi_{c})
\end{align*}
where we used the fact that the domain curvature vanishes if any of the indices equal to $0$ to replace $\alp$ by $a$. By the anti-symmetry properties of the curvature tensor $\calR$, note that the terms involving $\calR$ all cancel. Plugging this into the equation for $\rd_{s} w$, we obtain \eqref{eq:hw}. \qedhere
\end{proof}

\begin{lem}[Heat equation for $\psi_s$]  \label{l:psh}The heat tension field $\psi_s$ satisfies the following linearized covariant heat equation in $s$: 
\EQ{ \label{eq:psh} 
\rd_{s} \psi_{s} - \bfD^{b} \bfD_{b} \psi_{s} = \tensor{F}{_{s}^{b}} \psi_{b} = \bfR^{(0)}(\psi_{s}, \psi^{b}) \psi_{b}
}	
\end{lem}
\begin{proof}
Since any component of the domain curvature  tensor $\RR$ on $\R^+ \times \MM_I$ with an $s$ index is $ = 0$, we have 
\ant{
\p_s \psi_s &= \p_s \Db^b \psi_b  = \Db_s \Db^b \psi_b =  \Db^b \Db_s \psi_b + F^{ \, \, b}_s \psi_b \\
& = \Db^b \Db_b \psi_s + F^{ \, \, b}_s \psi_b
}
which establishes~\eqref{eq:psh}.
\end{proof}
\begin{lem}[Covariant heat equations for $\psi_\al$]  \label{l:pah} 
 Let $\psi_t, \psi_a$ denote the representations of $\p_t u$ and $\p_a u$ in the caloric gauge $e$, where $a \in \{1, \dots, 4\}$ is a spatial index. Then $\psi_t$ satisfies the following linearized covariant heat equation 
\EQ{ \label{eq:pth}
(\rd_{s} - \bfD^{b} \bfD_{b}) \psi_{t} = \tensor{F}{_{t}^{b}} \psi_{b} = \bfR^{(0)}(\psi_{t}, \psi^{b}) \psi_{b}
}
and $\psi_a$ satisfies the tensorial covariant heat equation 
\EQ{ \label{eq:pxh}
(\rd_{s} - \bfD^{b} \bfD_{b} - 3) \psi_{a} = \tensor{F}{_{a}^{b}} \psi_{b} = \bfR^{(0)}(\psi_{a}, \psi^{b}) \psi_{b}
}
\end{lem}
We remark that the linear part of \eqref{eq:pth} coincides with the scalar linear heat equation where $\psi_{t}$ is viewed as a scalar function on $\bbH^{4}$, whereas the linear part of \eqref{eq:pxh} is tensorial.
\begin{proof} 
With $\al  \in \{  0, 1,  \dots, 4\}$ we have 
\ant{
 \p_s \psi_\al &= \Db_s \psi_\al  = \Db_\al \psi_s =  \Db_\al  \Db^b \psi_b \\
 & = \Db^b  \Db_\al \psi_b + [\Db_\al, \Db_b] \h^{bc} \psi_c  \\
 & = \Db^b  \Db_b \psi_\al + [\Db_\al, \Db_b] \h^{bc} \psi_c  \\
 }
As in Lemma~\ref{l:psh}, the components of the domain curvature  tensor $\RR$ on $\R^+ \times \MM_I$   all vanish when $\al =0$. Hence, in that case we have
\ant{
[\Db_t, \Db_b] \h^{bc} \psi_c = F_t^{\, \, b} \psi_b
}
When $\al = a  \in \{1, \dots, 4\}$ we have 
\ant{
[\Db_a, \Db_b] \h^{bc} \psi_c &= F_a^{\, \, b} \psi_b - \RR^b_{ \, \, e ba} \h^{ec} \psi_c 
  = F_a^{\, \, b} \psi_b - \RR_{  e a} \h^{ec} \psi_c  \\
& = F_a^{\, \, b} \psi_b + 3 \h_{ae} \h^{ec} \psi_c  = F_a^{\, \, b} \psi_b + 3 \psi_a
}
which completes the proof. \qedhere
\end{proof} 

We conclude with identities regarding the spacetime and space divergences of $F$. These identities are useful for handling the last two terms on the right-hand side of \eqref{eq:hw}, as well as dealing with the term $\nb^{\alp} A_{\alp}$ that arise from expanding the covariant d'Alembertian $\bfD^{\alp} \bfD_{\alp}$.
\begin{lem}[Spacetime and space divergences of $F$]  \label{l:divFs}
Let $F_{\mathbf{a} \bt}$ denote a component of the curvature 2-form of a wave map in the caloric gauge $e$, where $\mathbf{a}$ denotes any one of the coordinates $(s, x^{\alp})$ and $\bt = 0, 1, \ldots, 4$. Then $F_{\mathbf{a} \bt}$ satisfies the following covariant divergence identities:
\begin{equation} \label{eq:divFs}
	\bfD^{\bt} F_{\mathbf{a} \bt} = \bfR^{(0)}(\bfD^{\bt} \psi_{\mathbf{a}}, \psi_{\bt}) + \bfR^{(0)}(\psi_{\mathbf{a}}, w) + \bfR^{(1)}(\psi^{\bt}; \psi_{\mathbf{a}}, \psi_{\bt})
\end{equation}
\begin{equation} \label{eq:sdivFs}
	\bfD^{b} F_{\mathbf{a} b} = \bfR^{(0)}(\bfD^{b} \psi_{\mathbf{a}}, \psi_{b}) + \bfR^{(0)}(\psi_{\mathbf{a}}, \psi_{s}) + \bfR^{(1)}(\psi^{b}; \psi_{\mathbf{a}}, \psi_{b})
\end{equation}
\end{lem}
\begin{proof} 
These identities follow from \eqref{eq:DF} and the definition of $w$ and $\psi_{s}$. \qedhere
\end{proof}

\section{Re-statement of the main theorem and outline of the strategy}\label{s:outline} 

Now that we have constructed the caloric gauge we can restate Theorem~\ref{t:main}, this time including  a priori estimates on dispersive norms of $\psi_s$. 

Let  $I \subset \R$ be a time interval and let $v = v(s, t, x)$ be a smooth function $v: \R^+ \times I \times \Hp^4  \to \R^n$.  For each $s>0$ define  
\EQ{
\| v(s) \|_{\mathcal{S}_s(I)} &:=   \| s^{\frac{1}{2}}v(s) \|_{ L^2_t(I;   L^8_x( \Hp^4; \R^n))} +  \|s^{\frac{1}{2}}(-\De)^{-\frac{1}{2}} \p_t v \|_{ L^2_t(I;   L^8_x( \Hp^4; \R^n))} \\
& \quad +   \| s^{\frac{1}{2}} \na_{t, x} v(s) \|_{ L^\I_t(I;   L^2_x( \Hp^4; \R^n))}. 
}
We then define the norm $\mathcal{S}(I)$ by 
\EQ{
\| v \|_{\mathcal{S}(I)} :=  \| v(s)\|_{\Ls^\I \cap \Ls^2 \mathcal{S}_s(I)}
}
where we have used the slight abuse of notation $L^{\infty}_{\ds} \cap L^{2}_{\ds} X := L^{\infty}_{\ds} \cap L^{2}_{\ds} (\bbR^{+}; X)$ for any space(-time) norm $X$, which will be in effect for the rest of this paper.

\begin{thm} \label{t:main1}  
Let $\NN$ be a complete $n$-dimensional manifold without boundary and with bounded geometry.  
There exists an $\eps_{0}>0$ small enough so that for all smooth  $(u_0, u_1)$ as in~\eqref{eq:data}, \eqref{eq:data1} with  
\EQ{ \label{eq:smalldata} 
\| (d u_0, u_1)\|_{H^1 \times H^1( \Hp^4; T\NN)} = \eps < \eps_{0}
}
there exists a unique, global smooth solution $u: \R \times \Hp^4 \to \NN \subset \R^N$ to~\eqref{wme} with initial data $(u, \p_t u) \!\! \restriction_{t=0}  = (u_0, u_1)$ and  satisfying 
\EQ{
\sup_{t \in \R} \| (du(t), \p_t u(t)) \|_{H^1 \times H^1( \Hp^4; T\NN)}  \lesssim \eps. 
}
Moreover we have the following dispersive properties of the global wave map $u$:  Let $u(s, t, x)$ denote the harmonic map heat flow resolution of $u(t,x)$,  let $e$ be the caloric gauge with limiting frame $e_\infty$, and define $\psi_s$  by $e\psi_s =  \p_s u$. Then, 
\begin{enumerate}

\item { \emph{A priori control of dispersive norms:}} $\psi_s$ satisfies the a priori estimates,  
\EQ{ \label{eq:psap} 
\| \psi_s \|_{\mathcal{S}(\R)} &\lesssim  \| s^{\frac{1}{2}} \na_{t, x}\psi_s \rest_{t = 0} \|_{\Ls^\I \cap \Ls^2 (L^2( \Hp^4; \R^n))} \\
& \lesssim  \| (d u_0, u_1)\|_{H^1 \times H^1( \Hp^4; T\NN)}  < \eps_{0}.
}
\item {\emph{Scattering:}}  For each fixed $s>0$  $\vec \psi_s:= ( \psi_s, \partial_t \psi_s)$ scatters to free waves as $t \to \pm \infty$, that is, for each fixed $s>0$ there exist $\R^n$-valued free waves $\phi_{  \pm}(s)$ such that  
\EQ{ \label{eq:pscat} 
 \| \na_{t, x} \psi_s(s, t) -   \na_{t, x}\phi_{ \pm}(s, t) \|_{ L^2(  \Hp^4; \R^n)} \to 0 \mas t \to \pm \infty.
 }
 \item{ \emph{Uniform control of higher derivatives:}} Higher regularity of the intial data  $\vec \psi_s(s, 0)$ is preserved by the flow. Moreover we have the following uniform estimates: 
 \EQ{ \label{eq:Depap} 
 \|   \De \psi_s \|_{ \cS( \R)} \lesssim \| (d u_0, u_1) \|_{ H^3 \times H^3( \Hp^4; T\NN)}
}
\item{\emph{Pointwise decay:}} We have the following qualitative pointwise decay: 
\EQ{ \label{eq:upwdec} 
\| u(t,  \cdot) - u_{\I} \|_{L^\infty_x} \to 0 \mas t \to \pm \infty.
}
\end{enumerate} 
\end{thm} 


We will now continue the proof of Theorem~\ref{t:main1} that we began in Section~\ref{s:cg}. Indeed, our first main goal is still to establish the a priori estimates in Proposition~\ref{p:ap}. However, in order to prove Proposition~\ref{p:ap} we will first deduce~\eqref{eq:psap} under the bootstrap assumption~\eqref{eq:bs}. 

To prove~\eqref{eq:psap} we will study the Cauchy problem for the wave equation for $\psi_s$ from Lemma~\ref{l:wmp}. 
By expanding out the left-hand side of \eqref{eq:wmp}, $\psi_s$ satisfies 
\EQ{ \label{eq:psis} 
 \Box_{\Hp^4} \psi_s  &= \p_s w -  \etab^{\al \be} F_{s \al} \psi_\be   - 2 \etab^{\al \be} A_\al \nb_\be \psi_s  \\
 & \quad  -  \etab^{\al \be} A_\al A_\be \psi_s - \etab^{\al \be} (\na_\al A)_\be \psi_s \\
   \vec \psi(s,0) &= ( \psi_s(s, 0), \p_t \psi_s(s, 0))
} 
where $w(s, t, x)$ solves ~\eqref{eq:hw}. We remark that  by Proposition~\ref{p:en} (in particular~\eqref{eq:eqns} restricted to $ \{ t= 0\}$) and Proposition~\ref{p:preg}, the small data assumption~\eqref{eq:smalldata} implies that 
\EQ{
 \| s^{\frac{1}{2}} \na_{t, x}\psi_s \rest_{t = 0} \|_{\Ls^\I \cap \Ls^2 (L^2( \Hp^4; \R^n))}  \lesssim  \| (d u_0, u_1)\|_{H^1 \times H^1( \Hp^4; T\NN)}  \lesssim \eps_{1}.
 }
Hence the proper setting for studying~\eqref{eq:psis} is that of the Cauchy problem for initial data with \emph{small critical Sobolev norm}. One of the key tools here is of course Strichartz estimates for the 
wave equation on $\Hp^4$ established by Anker and Pierfelice~\cite{AP14}; see also~\cite{MT11, MTay12}. In what follows, we will refer to a triple $(p, q, \gamma)$ as admissible  if 
\ant{
 \frac{1}{p} +  \frac{d}{q} = \frac{d}{2} - \gamma, \mand  \, \, \frac{1}{p}  + \frac{d-1}{2q} \le \frac{d-1}{4},  \mand \, \, p \ge 2, q \geq 2,
}
and if $d=3$ then $(p, q, \gamma)  \neq (2, \infty, 1)$. Note that we also allow the energy estimates $(p, q, \gamma) = ( \infty, 2, 0)$. 
\begin{thm}\emph{\cite{AP14, MT11, MTay12}} \label{t:str}
Let $J \subset \R$ be a time interval and let $v: I \times \Hp^d \to \R$,  be a solution to the inhomogenous wave equation, \EQ{
&\Box_{\Hp^d} v = G, \quad  \vec v(0) = (v_0, v_1)
}
and let $(p, q, \ga)$ be an admissible triple. Then, 
\EQ{
 \|   \na_{t, x} v \|_{L^{p}_t(J; W^{-\gamma, q}(\Hp^d))} \lesssim \|(v_0, v_1) \|_{H^1 \times L^2( \Hp^4)} + \|G \|_{L^1_t(J;  L^2_x( \Hp^d))}
}
\end{thm} 

\begin{rem} We note that the Strichartz estimates  above are a consequence of the dispersive estimates proved in~\cite{AP14, MTay12}, which we will also use in the proof of the pointwise decay estimate~\eqref{eq:upwdec} in Theorem~\ref{t:main1}. 
\begin{prop} \label{p:dis}{ \emph{ \cite{AP14, MTay12}}}
Fix  $d \ge 3$,  $q \in (2, \infty)$, and set $\s = (d+1)( \frac{1}{2} - \frac{1}{q} )$. 
\begin{itemize}
\item For $0 <  |t| \le 2$, the following short-time dispersive estimate holds:
\EQ{ \label{eq:shortd}
	\|e^{\pm i t \sqrt{-\De_{\Hp^{d}}}}f \| _{L^{q}(\Hp^{d})} \lesssim_q \abs{t}^{-(d-1) (\frac{1}{2} - \frac{1}{q})} \|f\|_{W^{\s, q'}(\Hp^{d})}
}
\item For $\abs{t} \geq 2$, the following long-time dispersive estimate holds: 
\EQ{ \label{eq:longd}
	\|e^{\pm i t \sqrt{-\De_{\Hp^{d}}}}f\|_{L^q(\Hp^{d})} \lesssim_q \abs{t}^{-\frac{3}{2}} \|f\|_{W^{\s, q'}(\Hp^{d})}
}
\end{itemize}
\end{prop}
\end{rem}

We apply Theorem~\ref{t:str} to~\eqref{eq:psis}.  Multiplying~\eqref{eq:psis} by $s^{\frac{1}{2}}$ and  using~\eqref{t:str} with $d =4$ and $(p_1, q_1, \gamma_1) = (2, 8, 1)$ and $(p_2, q_2, \gamma_2) = ( \infty, 2, 0)$,  we deduce that for all $s >0$,  
\EQ{
\| \psi_s(s) \|_{\mathcal{S}_s(I)} &\lesssim  \| s^{\frac{1}{2}} \na_{t, x}\psi_s \rest_{t = 0} \|_{ L^2_x} +  \| s^{\frac{1}{2}} \p_s w\|_{L^1_t L^2_x}+ \|  s^{\frac{1}{2}} \tensor{F}{_{s }^{\alp}} \psi_\al\|_{L^1_t L^2_x}  \\
& \,  + \|  s^{\frac{1}{2}}\na^\al A_\al \psi_s \|_{L^1_t L^2_x} + \|s^{\frac{1}{2}}A^\al  \nb_\al \psi_s\|_{L^1_t L^2_x} + \|s^{\frac{1}{2}}A^\al A_\al  \psi_s\|_{L^1_t L^2_x} 
}
Taking the $\Ls^\infty \cap \Ls^2$ norm of both sides yields, 
\EQ{ \label{eq:psstr} 
\| \psi_s \|_{\mathcal{S}(I)} &\lesssim  \| s^{\frac{1}{2}} \na_{t, x}\psi_s \rest_{t = 0} \|_{\Ls^\infty \cap \Ls^2 (L^2_x)} +  \| s^{\frac{1}{2}} \p_s w\|_{\Ls^\infty \cap \Ls^2 (L^1_t L^2_x)} \\
& \quad + \|  s^{\frac{1}{2}} \tensor{F}{_{s}^{\alp}} \psi_\al\|_{\Ls^\infty \cap \Ls^2 (L^1_t L^2_x)}    + \|  s^{\frac{1}{2}}\na^\al A_\al \psi_s \|_{\Ls^\infty \cap \Ls^2(L^1_t L^2_x)} \\
&\quad  + \|s^{\frac{1}{2}}A^\al  \nb_\al \psi_s\|_{\Ls^\infty \cap \Ls^2 (L^1_t L^2_x)} + \|s^{\frac{1}{2}}A^\al A_\al  \psi_s\|_{\Ls^\infty \cap \Ls^2 (L^1_t L^2_x)} 
}
The next order of business would be to control the terms on the  right-hand side above, under an additional bootstrap assumption that 
\EQ{ \label{eq:bsp}
\| \psi_s \|_{\mathcal{S}(I)} \leq 2 C_{1} \eps
}
for a large constant $C_{1}$ to be determined later. As in the case of the first bootstrap constant $C_{0}$, in what follows we suppress the dependence of constants on $C_{1}$. This convention is justified by the fact that $C_{1}$ is always accompanied by a factor of $\eps$.

Before directly addressing the terms on the right-hand side of~\eqref{eq:psstr} we first use~\eqref{eq:bsp} to establish a wider class of estimates. In particular, we use the fact that for each fixed $t \in I$, $\psi_s$ satisfies the~\emph{parabolic} equation~\eqref{eq:psh} to deduce  appropriately $s$-weighted estimates on the $L^2_t L^8_x$ and $L^\I_t L^2_x$ norms of higher derivatives of $\psi_s$.  This is accomplished in the next section. The rough outline of the remainder of the paper is as follows: 
\begin{itemize}
\item In Section~\ref{s:pr} we use the parabolic equation satisfied by $\psi_s$ to control higher derivatives of $\psi_s$ and $(-\De)^{-\frac{1}{2}} \p_t \psi_s$  in $\Ls^\I \cap \Ls^2 L^2_t L^8_x$ and of $\na_{t, x} \psi_s$ in $\Ls^\I \cap \Ls^2 L^\I_t L^2_x$  (with appropriate $s$-weights) by $\|\psi_s\|_{\cS(I)}$.  Here we directly invoke the $L^p_x$ estimates for heat semi-group on $\Hp^4$ from Lemma~\ref{l:hk}. We then  establish a wider class of estimates for $\psi_s$, $\psi_\al$, and  $A_\al$ that will be needed to close the a priori estimates~\eqref{eq:psap} in the next section.  
\item In Section~\ref{s:wave} we close the a priori estimates~\eqref{eq:psap} under the bootstrap assumption~\eqref{eq:bs}. A significant part  of the section is devoted to controlling the terms  $\|s^{\frac{1}{2}} \p_s w\|_{\Ls^\I \cap \Ls^2 L^1_t L^2_x}$  and $\|  s^{\frac{1}{2}}\na^\al A_\al \psi_s \|_{\Ls^\infty \cap \Ls^2 L^1_t L^2_x}$ from~\eqref{eq:psstr}. Here we will make use of the parabolic equation~\eqref{eq:hw} satisfied by $w$, again directly invoking properties of the heat semi-group from  Lemma~\ref{l:hk}. 
\item Finally, in Section~\ref{s:proof} we  complete the proofs of Proposition~\ref{p:ap} and Theorem~\ref{t:main1}. 
\end{itemize}

\section{Parabolic regularity for $\psi_s$}\label{s:pr} 
In this section we establish $L^p$-type parabolic regularity estimates for $\psi_s$ under the bootstrap assumption~\eqref{eq:bsp} by directly invoking properties of the heat semigroup $e^{s \De}$ from Lemma~\ref{l:hk}. In Section~\ref{s:wide} we then establish a wider collection of estimates for $\psi_s$, $\Psi$, and $A$ that will be needed in the sequel.


\begin{prop} \label{p:preg8}
Let $\psi_s$ be the heat tension field as in~\eqref{eq:htf}. Suppose that the bootstrap assumptions \eqref{eq:bs} and \eqref{eq:bsp} hold.
Then, 
\EQ{ \label{eq:preg8}
& \| s \na \na_{t, x} \psi_s \|_{\Ls^\I \cap \Ls^2 L^\infty_t  L^2_x}+ \|s \na_{t, x} \psi_s\|_{\Ls^\I \cap \Ls^2 L^2_t  L^8_x} \lesssim \| \psi_s \|_{\mathcal{S}(I)}, \\
&\| s^{\frac{3}{2}} \na^2 \na_{t, x} \psi_s \|_{\Ls^\I \cap \Ls^2 L^\infty_t  L^2_x} + \| s^{\frac{3}{2}} \na  \na_{t, x}  \psi_s \|_{\Ls^\I \cap \Ls^2 L^2_t  L^8_x} \lesssim \| \psi_s \|_{\mathcal{S}(I)},
}
where the implicit constants do not depend on $C_0$ in \eqref{eq:bs} and $C_{1}$ in \eqref{eq:bsp}.
\end{prop}

The proof of Proposition~\ref{p:preg8}  and the arguments that follow require the following lemma, which is a simple consequence of Proposition~\ref{p:preg} and Lemma~\ref{l:A}. 
\begin{lem} \label{l:PA}
Assume that the bootstrap assumptions~\eqref{eq:bs} and \eqref{eq:bsp} hold. Then,  
\EQ{ \label{eq:painf} 
&\| s^{\frac{1}{2}} \Psi  \|_{\Ls^\I   L^{\infty}_t L^{\infty}_x} +    \| s\na  \Psi  \|_{\Ls^\I L^{\infty}_t L^{\infty}_x} +   \| s\Db  \Psi  \|_{\Ls^\I L^{\infty}_t L^{\infty}_x} \lesssim  \eps,  \\ 
&  \| s^{\frac{1}{2}} \Psi  \|_{L^\I_t \Ls^2   L^{\infty}_x}+ \| s\na  \Psi  \|_{L^{\infty}_t\Ls^2  L^{\infty}_x}+  \| s\Db  \Psi  \|_{ L^{\infty}_t \Ls^2 L^{\infty}_x} \lesssim  \eps,  
}
and
\EQ{
& \| s^{\frac{3}{8}}  \Db  \Psi \|_{ \Ls^\I L^\I_t L_{x}^{\frac{16}{5}}} \simeq \| s^{\frac{3}{8}}  \na  \Psi \|_{ \Ls^\I L^\I_t L_{x}^{\frac{16}{5}}} \lesssim \eps,  \label{eq:DP165}
}
 and 
 \EQ{ \label{eq:Ainf} 
 \| s^{\frac{1}{2}} A\|_{ \Ls^\I L^\I_t L^\infty_x}  + \| s \na  A\|_{ \Ls^\I L^\I_t L^\infty_x} +  \| s^{\frac{3}{2}} \na^{2} A\|_{ \Ls^\I L^\I_t L^\infty_x} &\lesssim \eps^2  \\
    \| s \na_{t, x} \underline{ A}\|_{ \Ls^\I L^\I_t L^\infty_x} +  \| s^{\frac{3}{2}} \na \na_{t, x} \underline{A}\|_{ \Ls^\I L^\I_t L^\infty_x}& \lesssim \eps^2  
}
\end{lem} 

\begin{rem}
By the torsion-free property $\bfD_{\alp} \Psi_{\bt} = \bfD_{\bt} \Psi_{\alp}$, the following analogue of \eqref{eq:DP165} holds for $\bfD_{t,x} \underline{\Psi}$:
\begin{equation*} \tag{\ref{eq:DP165}$'$}
\| s^{\frac{3}{8}}  \Db_{t,x}  \underline{\Psi} \|_{ \Ls^\I L^\I_t L_{x}^{\frac{16}{5}}}  \lesssim \eps
\end{equation*}
A similar remark applies to \eqref{eq:painf}.
\end{rem}

\begin{proof}
The bounds on the first terms in each line on the left in~\eqref{eq:painf} are consequences of the Gagliardo-Nirenberg inequality,  
\ant{
\|  s^{\frac{1}{2}} \Psi \|_{ L^\infty_x}^2   &=   \| s^{\frac{1}{2}} d_{t,x} u \|_{ L^\infty_x}^2   \lesssim  \| s^{\frac{1}{4}} (-\De)^{\frac{1}{2}} u \|_{L_{x}^8} \| s^{\frac{3}{4}}  \na d_{t,x}u\|_{L_{x}^8} \\
& \lesssim  \| s^{\frac{1}{4}} (- \De)^{\frac{5}{4}} u \|_{L_{x}^2}  \| s^{\frac{1}{2}} \na^2 d_{t,x} u \|_{L_{x}^2}^{\frac{1}{2}} \|s \na^3 d_{t,x} u \|_{L_{x}^2}^{\frac{1}{2}}
}
along with Proposition~\ref{p:preg}. The second terms on the left  in~\eqref{eq:painf} are controlled similarly. The third terms are then easy consequences of the previous bounds along with~\eqref{eq:Ainf}. To prove~\eqref{eq:Ainf} we again use Gagliardo-Nirenberg to obtain  
\ant{ 
\| s^{\frac{1}{2}} A\|_{ L^\infty_x}  & \lesssim  s^{\frac{1}{2}} \|  A\|_{L_{x}^8}^{\frac{1}{2}} \| \na  A\|_{L_{x}^8}^{\frac{1}{2}}    \lesssim \|A \|_{L_{x}^4}^{\frac{1}{4}} \| s^{\frac{1}{2}}\na^2 A\|_{L_{x}^2}^{\frac{1}{2}} \| s \na^3 A \|_{L_{x}^2}^{\frac{1}{4}} ,
}
and then conclude using Lemma~\ref{l:A}. The terms involving $s \nb A$, $s^{\frac{3}{2}} \na^{2} A$, $s \na_{t, x}  \underline{A}$  and $s^{\frac{3}{2}} \na \na_{t, x}  \underline{A}$ in~\eqref{eq:Ainf} are controlled in an identical fashion. 

Finally, we turn to the proof of~\eqref{eq:DP165}. The first equivalence follows from \eqref{eq:AL4} (and taking $\eps$ small enough), so it suffices to prove
\EQ{ \label{eq:nadu165} 
 \| s^{\frac{3}{8}}   \na \Psi  \|_{ \Ls^\I L^\I_t L_{x}^{\frac{16}{5}}} \lesssim \eps .
}
This estimate is an easy consequence of Proposition~\ref{p:preg} and Gagliardo-Nirenberg, 
\ant{
 \| s^{\frac{3}{8}}   \na \Psi \|_{ \Ls^\I L^\I_t L_{x}^{\frac{16}{5}}} 
 \lesssim  &\|    \na \Psi \|_{ \Ls^\I L^\I_t L_{x}^2}^{\frac{1}{4}} \| \na^2 \Psi \|_{ \Ls^\I L^\I_t L_{x}^2}^{\frac{3}{4}} \\
 \lesssim  & \|    \na d_{t, x} u\|_{ \Ls^\I L^\I_t L_{x}^2}^{\frac{1}{4}} \| \na^2 d_{t, x} u \|_{ \Ls^\I L^\I_t L_{x}^2}^{\frac{3}{4}} \lesssim \eps,
 }
 which completes the proof.  \qedhere
\end{proof} 

\begin{proof}[Proof of Proposition~\ref{p:preg8}]
Recall from Lemma~\ref{l:psh} that  that $\psi_s$ satisfies the heat equation 
\EQ{ \label{eq:G1def} 
\p_s \psi_s  - \De \psi_s = F_{s}^{\, \, b} \psi_b + 2 A^b \nabla_b \psi_s + A^b A_b \psi_s + \na^b A_b \psi_s = : G_1(s).
}
Therefore by Duhamel's formula we have 
\EQ{ \label{eq:dph}
\psi_s(s) = e^{\frac{s}{2} \De} \psi_s (s/2) + \int_{\frac{s}{2}}^s e^{(s-s') \De} G_1( s') \, ds'
}
for each $s>0$.  
Therefore, 
\ali{\label{preg temp 1}
s (-\De)^{\frac{1}{2}} \psi_s(s) = s^{\frac{1}{2}} (-\De)^{\frac{1}{2}}e^{\frac{s}{2} \De} s^{\frac{1}{2}}\psi_s (s/2) + s \int_{\frac{s}{2}}^s (-\De)^{\frac{1}{2}}e^{(s-s') \De} G_1( s') \, ds',
}
and
\begin{multline} \label{preg temp 2} 
s (-\De) \psi_s(s) = s^{\frac{1}{2}} (-\De)^{\frac{1}{2}}e^{\frac{s}{2} \De} s^{\frac{1}{2}}(-\De)^{\frac{1}{2}}\psi_s (s/2)  \\ + s \int_{\frac{s}{2}}^s (-\De)^{\frac{1}{2}}e^{(s-s') \De} (-\De)^{\frac{1}{2}}G_1( s') \, ds'.
\end{multline}
Similarly, commuting $\na_t$ through~\eqref{eq:dph} and
~\eqref{preg temp 1} we have 
\begin{multline} \label{preg temp 1t} 
s \na_t \psi_s(s) = s^{\frac{1}{2}}(-\De)^{\frac{1}{2}}e^{\frac{s}{2} \De}s^{\frac{1}{2}} (-\De)^{-\frac{1}{2}} \na_t \psi_s (s/2)  \\+ s \int_{\frac{s}{2}}^s  e^{(s-s') \De} \na_t G_1( s') \, ds'
\end{multline}
and
\begin{multline}\label{preg temp 2t}
s (-\De)^{\frac{1}{2}} \nabla_t \psi_s(s) = s^{\frac{1}{2}} (-\De)^{\frac{1}{2}}e^{\frac{s}{2} \De} s^{\frac{1}{2}}\nabla_t\psi_s (s/2) \\+ s \int_{\frac{s}{2}}^s (-\De)^{\frac{1}{2}}e^{(s-s') \De}\nabla_t G_1( s') \, ds'
\end{multline}
Next, from the heat semigroup bound~\eqref{eq:napest} with $p=8$ applied to the right-hand side of~\eqref{preg temp 1}  and the bound~\eqref{eq:pest} applied to the right-hand side of~\eqref{preg temp 1t}  
we see that 
\ant{
 \|s \na_{t, x} \psi_s(s)\|_{L^2_t L^8_x} &\lesssim \| (s/2)^{\frac{1}{2}} \psi_s(s/2) \|_{L^2_t L^8_x}  +  \| (s/2)^{\frac{1}{2}} (-\De)^{-\frac{1}{2}} \na_t \psi_s(s/2) \|_{L^2_t L^8_x} \\
 & \quad +  \int_{s/2}^s \frac{s}{(s-s')^{\frac{1}{2}}(s')^{\frac{1}{2}}} \| (s')^{\frac{3}{2}} G_1(s')\|_{L^2_t L^8_x} \, \dsp \\
 & \quad +  \int_{s/2}^s \Big(\frac{s}{s'}\Big) \| (s')^2 \na_t  G_1(s')\|_{L^2_t L^8_x} \, \dsp 
 }
By Schur's test, it follows that 
\EQ{ \label{eq:snaps}
 \|s \na_{t,x} \psi_s\|_{\Ls^\I \cap \Ls^2 L^2_t L^8_x}  &\lesssim  \| \psi_s\|_{\mathcal{S}(I)} + \| s^{\frac{3}{2}} G_1(s)\|_{\Ls^\I \cap \Ls^2 L^2_t L^8_x}  \\
 & \quad + \| s^2 \na_t G_1(s) \|_{\Ls^\I \cap \Ls^2 L^2_t L^8_x}
}
Similarly applying~\eqref{eq:napest} with $p=2$ to~\eqref{preg temp 2} and~\eqref{preg temp 2t} and arguing as above we obtain
\EQ{ \label{eq:snaps2}
 \|s \na \nabla_{t,x} \psi_s\|_{\Ls^\I \cap \Ls^2 L^\infty_t L^2_x}  \lesssim  \| \psi_s\|_{\mathcal{S}(I)} + \| s^{\frac{3}{2}} \nabla_{t,x} G_1(s)\|_{\Ls^\I \cap \Ls^2 L^\infty_t L^2_x}
}
Next, we show that 
\begin{align}   \notag
\| &s^{\frac{3}{2}} G_1\|_{\Ls^\I \cap \Ls^2 L^2_t L^8_x} + \| s^2 \na_t G_1(s) \|_{\Ls^\I \cap \Ls^2 L^2_t L^8_x} +  \| s^{\frac{3}{2}} \nabla_{t,x} G_1\|_{\Ls^\I \cap \Ls^2 L^\infty_t L^2_x}  \\& \lesssim \, \eps^2  \Big(\|s \na_{t, x} \psi_s\|_{\Ls^\I \cap \Ls^2 L^2_t L^8_x}    \label{eq:s32G1} 
 +   \|s \na\nabla_{t,x} \psi_s\|_{\Ls^\I \cap \Ls^2 L^\infty_t L^2_x} \Big)  \\
 & \quad + \eps^2 \| \psi_s\|_{\calS(I)} \notag
\end{align}
Note that~\eqref{eq:s32G1} is sufficient to establish the first line in~\eqref{eq:preg8} since the terms inside the parenthesis  on the right-hand side above can be absorbed into the left-hand side of either~\eqref{eq:snaps} or~\eqref{eq:snaps2} as long as $\eps$ is small enough. 

We begin with the first $L_t^2L_x^8$ estimate of $s^{\frac{3}{2}}G_{1}(s)$ in~\eqref{eq:s32G1}, bounding each of the terms in $s^{\frac{3}{2}}G_1(s)$ as follows. Using~\eqref{eq:Fdu} and~\eqref{eq:painf} we have 
\ant{
\| s^{\frac{3}{2}}F_{s}^{\, \, b} \psi_b \|_{\Ls^\I \cap \Ls^2 L^2_t L^8_x} &\lesssim \| s^{\frac{1}{2}} \psi_s \|_{\Ls^\I \cap \Ls^2 L^2_t L^8_x} \| s^{\frac{1}{2}} \Psi \|_{\Ls^\I  L^\I_t L^\I_x}^2 \\
& \lesssim  \eps^2 \|  \psi_s \|_{\cS(I)}
}
Next, from~\eqref{eq:Ainf} we see that 
\ant{
\| s^{\frac{3}{2}}A^b \nabla_b \psi_s \|_{\Ls^\I \cap \Ls^2 L^2_t L^8_x} &\lesssim  \| s^{\frac{1}{2}} A \|_{\Ls^\infty L^\I_t L^\I_x} \| s \na \psi_s \|_{\Ls^\I \cap \Ls^2 L^2_t L^8_x} \\
& \lesssim \eps^2  \| s \na  \psi_s \|_{\Ls^\I \cap \Ls^2 L^2_t L^8_x}
}
which can be absorbed onto the left-hand side of~\eqref{eq:snaps}. Then note that 
\ali{\label{preg temp 4}
\| s^{\frac{3}{2}}A^bA_b \psi_s \|_{\Ls^\I \cap \Ls^2 L^2_t L^8_x} &\lesssim  \| s^{\frac{1}{2}}A \|_{\Ls^\infty L^\I_t L^\I_x}^2  \| s^{\frac{1}{2}} \psi_s \|_{\Ls^\I \cap \Ls^2 L^2_t L^8_x} \\
& \lesssim  \eps^4 \|  \psi_s \|_{\cS(I)}
}
Finally, using again~\eqref{eq:Ainf} gives 
\ali{\label{preg temp 5}
\| s^{\frac{3}{2}}\na^b A_b  \psi_s \|_{\Ls^\I \cap \Ls^2 L^2_t L^8_x} &\lesssim  \| s \na   \underline{A} \|_{\Ls^\infty L^\I_t L^\I_x}  \| s^{\frac{1}{2}} \psi_s \|_{\Ls^\I \cap \Ls^2 L^2_t L^8_x} \\
& \lesssim  \eps^2 \|  \psi_s \|_{\cS(I)}
}
Putting the previous four displayed equations together we have established that 
\ant{
\| s^{\frac{3}{2}} G_1\|_{\Ls^\I \cap \Ls^2 L^2_t L^8_x} \lesssim \eps^2  \|s \na  \psi_s\|_{\Ls^\I \cap \Ls^2 L^2_t L^8_x} +  \eps^2 \| \psi_s \|_{ \cS(I)}
}
as desired.

Now we turn to the $L_t^\infty L_x^2$ estimates of $s^{\frac{3}{2}}\na_{t, x} G_1$  and the $L^2_t L^8_x$ estimate of $s^2 \na_t G_1$ in~\eqref{eq:s32G1}. We claim that 
 \EQ{ \label{eq:s32naG}
   \| s^{\frac{3}{2}} \nabla_{t,x} G_1\|_{\Ls^\I \cap \Ls^2 L^\infty_t L^2_x}   \lesssim  \eps^2 \|s \na \nabla_{t,x} \psi_s\|_{\Ls^\I \cap \Ls^2 L^\infty_t L^2_x}
 + \eps^2 \| \psi_s\|_{\calS(I)}
   }
and 
\EQ{ \label{eq:s2natG}
\| s^2 \na_t G_1\|_{\Ls^\I \cap \Ls^2 L^2_t L^8_x} \lesssim \eps^2  \|s \na_t \psi_s\|_{\Ls^\I \cap \Ls^2 L^2_t L^8_x} +  \eps^2 \| \psi_s \|_{ \cS(I)}
}
We perform the details of the proof only for~\eqref{eq:s32naG} since~\eqref{eq:s2natG} can be proved similarly. We begin by computing  $\na_{t, x} G_1(s)$. For any index $\mu \in \{0, \dots, 4\}$ we have 
\EQ{ \label{eq:naG1}
 \na_\mu G_1 =  \na_\mu(F_{s}^{\, \, b} \psi_b) + 2  \na_\mu(A^b \nabla_b \psi_s) + \na_\mu(A^b A_b \psi_s) +  \na_\mu(\na^b A_b \psi_s )
 }
and we can expand and estimate each of the terms on the right above as follows. The first term is given by  
\EQ{ \label{eq:naFsb}
 \na_\mu(F_{s}^{\, \, b} \psi_b) &=  \Db_\mu(F_{s}^{\, \, b} \psi_b) - A_\mu F_{s}^{\, \, b} \psi_b =  \Db_\mu F_{s}^{ \, \, b} \psi_b + F_{s}^{ \, \, b} \Db_b\psi_\mu  - A_\mu F_{s}^{\, \, b} \psi_b \\
 & = \bfR^{(1)}( \psi_\mu;  \psi_s, \psi^b) \psi_b +  \bfR^{(0)}( \Db_\mu \psi_s, \psi^b) \psi_b +  \bfR^{(0)}( \psi_s, \Db_\mu \psi^b) \psi_b  \\
 &\quad + \bfR^{(0)}(  \psi_s , \psi^b) \Db_b \psi_\mu  - A_\mu  \bfR^{(0)}( \psi_s, \psi^b) \psi_b
 }
 We then have 
 \begin{multline*}
  \| s^{\frac{3}{2}}\bfR^{(1)}( \Psi;  \psi_s, \psi^b) \psi_b  \|_{ \Ls^\I \cap \Ls^2 L^\I_t L^2_x}  \\
  \lesssim  \| s^{\frac{1}{2}} \psi_s \|_{\Ls^\I \cap \Ls^2 L^\I_t L^4_x}  \| s^{\frac{1}{2}} \Psi \|_{ \Ls^\I L^\I_t L^\I_x}^2 \| \Psi \|_{\Ls^\I L^\I_t L^4_x}  \lesssim \eps^3 \| \psi_s \|_{ \cS(I)}
   \end{multline*}
 where txnd inequality follows from Sobolev embedding,~\eqref{eq:painf}, and~\eqref{eq:eqna}. For the next term we use~\eqref{eq:eqns} and~\eqref{eq:painf} to deduce that 
 \ant{
 \| s^{\frac{3}{2}}\bfR^{(0)}( \Db_{t, x} \psi_s, \psi^b) \psi_b  \|_{ \Ls^\I \cap \Ls^2 L^\I_t L^2_x}  & \lesssim  \|s^{\frac{1}{2}} \Db_{t, x} \psi_s \|_{\Ls^\I \cap \Ls^2 L^\I_t L^2_x} \| s^{\frac{1}{2}} \Psi \|_{ \Ls^\I L^\I_t L^\I_x}^{2} \\
 & \lesssim  \eps^2 \| \psi_s \|_{\cS(I)}
 }
 For the third term in~\eqref{eq:naFsb} we use Sobolev embedding, ~\eqref{eq:eqna} and~\eqref{eq:painf} to conclude that 
\begin{multline*}
   \| s^{\frac{3}{2}} \bfR^{(0)}( \psi_s, \Db_{t, x} \psi^b) \psi_b   \|_{ \Ls^\I \cap \Ls^2 L^\I_t L^2_x}  \\
    \lesssim   \|s^{\frac{1}{2}} \psi_s \|_{\Ls^\I \cap \Ls^2 L^\I_t L^4_x} \|  \Psi \|_{ \Ls^\I L^\I_t L^4_x } \| s  \Db_{t,x} \underline{\Psi} \|_{\Ls^\I L^\I_t L^\I_x} \lesssim  \eps^2  \| \psi_s \|_{ \cS(I)}
  \end{multline*} 
The fourth term in~\eqref{eq:naFsb} is controlled in an identical fashion. Finally for the fifth term we have
\begin{multline*}
   \| s^{\frac{3}{2}} A \,  \bfR^{(0)}( \psi_s, \psi^b) \psi_b   \|_{ \Ls^\I \cap \Ls^2 L^\I_t L^2_x}  \\
    \lesssim   \|s^{\frac{1}{2}} \psi_s \|_{\Ls^\I \cap \Ls^2 L^\I_t L^4_x} \|  A \|_{ \Ls^\I L^\I_t L^4_x } \| s^{\frac{1}{2}}   \Psi \|_{\Ls^\I L^\I_t L^\I_x}^{2} \lesssim  \eps^3  \| \psi_s \|_{ \cS(I)}
   \end{multline*} 
    using~\eqref{eq:AL4} and~\eqref{eq:painf} in the last line. 
    
    The next term in~\eqref{eq:naG1} can be expanded as follows, 
 \ant{
  \na_\mu(A^b  \na_b \psi_s) &=  \na_\mu A^b \na_b \psi_s + A^b \na_\mu \na_b \psi_s \\ 
  & =  \na_\mu A^b \na_b \psi_s + A^b \na_b \na_\mu \psi_s
 }
 where the commutation $\na_\mu \na_b \psi_s   = \na_b \na_\mu \psi_s$ does not pick up any domain curvature terms since the components $\psi_s^j$  of $\psi_s$ are treated in this context as scalar functions on $\R^+ \times \MM_I$. We can then use~\eqref{eq:Ainf} to show that 
\begin{align*}
 \| s^{\frac{3}{2}}  \na_{t, x}(A^b  \na_b \psi_s) \|_{ \Ls^\I \cap \Ls^2  L^\I_t L^2_x}   &\lesssim  \|s  \na_{t,x} \underline{A} \|_{\Ls^\I L^\I_{t, x}} \|s^{\frac{1}{2}} \na \psi_s \|_{ \Ls^\I \cap \Ls^2  L^\I_t L^2_x}  \\
 &\quad +  \|s^{\frac{1}{2}}A \|_{\Ls^\I L^\I_{t, x}} \| s \na \na_{t, x} \psi_s \|_{ \Ls^\I \cap \Ls^2  L^\I_t L^2_x} \\
 & \lesssim  \eps^2 \| \psi_s \|_{ \cS(I)} +  \eps^2 \| s \na \na_{t, x} \psi_s \|_{ \Ls^\I \cap \Ls^2  L^\I_t L^2_x}
 \end{align*} 
Similarly, 
  \begin{multline*}
  \| s^{\frac{3}{2}}\na_{t, x} (A^b A_b \psi_s) \|_{ \Ls^\I \cap \Ls^2  L^\I_t L^2_x}   \lesssim  \| s^{\frac{1}{2}} \na_{t, x} \psi_s \|_{ \Ls^\I \cap \Ls^2 L^\I_t L^2_x}  \| s^{\frac{1}{2}} A \|_{ \Ls^\I L^\I_{t, x}}^2 \\
   + \| s^{\frac{1}{2}} \psi_s \|_{ \Ls^\I \cap \Ls^2 L^\I_t L^4_x} \|A \|_{ \Ls^\I L^\I_t L_{x}^4} \| s  \na_{t, x}  \underline{A} \|_{ \Ls^\I L^\I_{t, x}}   \lesssim  \eps^2 \| \psi_s \|_{ \cS(I)}
  \end{multline*} 
  For the last term we have 
  \ant{
  \na_\mu(\na^b A_b \psi_s ) =  \na_\mu \na^b A_b \psi_s +  \na^b A_b \na_\mu \psi_s
  }
  If $\mu = t$ note that we can commute $\na_t \na^b A_b \psi_s= \na^b \na_t  A_b \psi_s$. Hence, using~\eqref{eq:dkAL2} and~\eqref{eq:Ainf}
  \ant{
   \| s^{\frac{3}{2}} \na_{t, x}(\na^b A_b \psi_s ) \||_{ \Ls^\I \cap \Ls^2  L^\I_t L^2_x} &  \lesssim  \| s\na \na_{t, x} \underline{A} \|_{ \Ls^\I L^\I_t L^4_x}   \|s^{\frac{1}{2}} \psi_s \|_{ \Ls^\I \cap \Ls^2 L^\I_t L^4_x}  \\
   &  \quad +  \|s \na \underline{A} \|_{ \Ls^\I L^\I_{t, x}}  \|s^{\frac{1}{2}} \na_{t, x}  \psi_s \|_{ \Ls^\I \cap \Ls^2 L^\I_t L^2_x} \\
   & \lesssim \eps^2 \| \psi_s \|_{ \cS(I)}
   }
   This completes the proof of~\eqref{eq:s32naG}. The inequality~\eqref{eq:s2natG} can be proved in a similar fashion, completing the proof of~\eqref{eq:s32G1}.   

Finally, note that the second line in~\eqref{eq:preg8} can be proved in a nearly identical manner after applying $(-\De)^{\frac{1}{2}}$ to~\eqref{preg temp 1},~\eqref{preg temp 1t} and~\eqref{preg temp 2},~\eqref{preg temp 2t}. 
Indeed, arguing exactly as before, and this time noting that we have already established the first line in~\eqref{eq:preg8}, we see that it suffices to prove that 
\begin{multline*}
 \| s^2 \na_{t, x}  G_1\|_{\Ls^\I \cap \Ls^2 L^2_t L^8_x} + \| s^2 \na\nabla_{t,x} G_1\|_{\Ls^\I \cap \Ls^2 L^\infty_t L^2_x} \\ \lesssim  \eps  \|s^{\frac{3}{2}} \na \na_{t, x} \psi_s\|_{\Ls^\I \cap \Ls^2 L^2_t L^8_x} 
+ \eps  \|s^{\frac{3}{2}} \na^2\nabla_{t,x} \psi_s\|_{\Ls^\I \cap \Ls^2 L^\infty_t L^2_x}
+  \| \psi_s \|_{\cS(I)},
\end{multline*}
which in turn follows from a nearly identical argument to the one used to prove~\eqref{eq:s32G1}, this time requiring the full set of bounds in~\eqref{eq:painf},~\eqref{eq:Ainf}. 
\end{proof}

\subsection{A collection of estimates}  \label{s:wide}
In the next section we will require a wider class of estimates for $\psi_s, \Psi$, and  $A$ than those that have already been explicitly stated. For convenience, we collect some of these estimates here. 

The first simple observation is that we can use Proposition~\ref{p:preg} and Proposition~\ref{p:preg8} together with the Gagliardo--Nirenberg inequality 
to establish the following estimates for $\psi_s$. 
\begin{lem} \label{l:16}
Let $\psi_s$ be the heat tension field defined as in~\eqref{eq:htf}. Then,  
\begin{multline} \label{eq:psinfinf}
\| s  \psi_s\|_{\Ls^\I \cap\Ls^2 L^\infty_t L^{\I}_x} \\ 
\lesssim  \| s^{\frac{1}{2}} \na \psi_s \|_{\Ls^\I \cap\Ls^2 L^\infty_t L^{2}_x}^{\frac{1}{4}} \| s \na^2 \psi_s \|_{\Ls^\I \cap\Ls^2 L^\infty_t L^{2}_x}^{\frac{1}{2}} \| s^{\frac{3}{2}} \na^3 \psi_s \|_{\Ls^\I \cap\Ls^2 L^\infty_t L^{2}_x}^{\frac{1}{4}} \\
 \lesssim   \| \psi_s \|_{\cS(I)}, 
\end{multline} 
\begin{multline} \label{eq:ps8}
\| s^{\frac{5}{4}}  \na_{t, x} \psi_s\|_{\Ls^\I \cap\Ls^2 L^\infty_t L^{8}_x}  
 \\ \lesssim  \| s \na \na_{t, x} \psi_s \|_{\Ls^{\infty}\cap \Ls^2 L^\infty_t L^{2}_x}^{ \frac{1}{2}} \|  s^{\frac{3}{2}} \na^2 \na_{t, x} \psi_s \|_{\Ls^{\infty}\cap \Ls^2 L^\infty_t L^{2}_x}^{\frac{1}{2}} \lesssim  \| \psi_s \|_{\cS(I)} 
\end{multline} 
and
\begin{multline} \label{eq:ps165}
\| s^{\frac{3}{8}} s^{\frac{1}{2}}  \na_{t, x} \psi_s\|_{\Ls^\I \cap\Ls^2 L^\infty_t L^{\frac{16}{5}}_x} 
\\ \lesssim  \| s^{\frac{1}{2}} \na_{t, x} \psi_s \|_{\Ls^{\infty}\cap \Ls^2 L^\infty_t L^{2}_x}^{ \frac{1}{4}} \|  s \na \na_{t, x} \psi_s \|_{\Ls^{\infty}\cap \Ls^2 L^\infty_t L^{2}_x}^{\frac{3}{4}} \lesssim    \| \psi_s \|_{\cS(I)} 
\end{multline} 
Moreover, we have 
\begin{align} \label{eq:ps16} \notag
 \| s^{\frac{5}{8}}&   \psi_s \|_{\Ls^\I \cap\Ls^2 L^2_t L^{16}_x}  +   \| s^{\frac{9}{8}} \na_{t, x}  \psi_s \|_{\Ls^\I \cap\Ls^2 L^2_t L^{16}_x} 
  \\& \lesssim      \|  s^{\frac{1}{2}}   \psi_s\|_{\Ls^\I \cap\Ls^2 L^2_t L^{8}_x}^{\frac{3}{4}}    \|  s \na  \psi_s\|_{\Ls^\I \cap\Ls^2 L^2_t L^{8}_x}^{\frac{1}{4}}    \\     
 &\quad +  \|  s \na_{t, x}  \psi_s\|_{\Ls^\I \cap\Ls^2 L^2_t L^{8}_x}^{\frac{3}{4}}    \|  s^{\frac{3}{2}} \na \na_{t, x}  \psi_s\|_{\Ls^\I \cap\Ls^2 L^2_t L^{8}_x}^{\frac{1}{4}}    
\lesssim  \| \psi_s \|_{ \mathcal{S}(I)},  \notag
\end{align}

\end{lem} 

\begin{proof} 
The proof of~\eqref{eq:psinfinf} follows from two applications of the Gagliardo-Nirenberg inequality followed by~\eqref{eq:preg8}.  The proofs of~\eqref{eq:ps8}~\eqref{eq:ps165}, and~\eqref{eq:ps16} are similar.  
\end{proof} 


With Lemma~\ref{l:16} in hand we can establish the following estimates for $\Psi$ and  $A$. 

\begin{cor}\label{c:APsi} 
Let $\Psi$
and $A$ 
be as above and assume that the bootstrap assumptions~\eqref{eq:bs} and \eqref{eq:bsp} hold. 
Then,   
\begin{align} \label{eq:pa16}
&\| s^{\frac{1}{8}} \Psi \|_{ \Ls^{\infty} \cap \Ls^2 L^{2}_t L^{16}_x} \lesssim \| \psi_s \|_{\mathcal{S}(I)}  \\
& \| s^{\frac{1}{4}} \Psi \|_{ \Ls^{\infty} \cap \Ls^2 L^{2}_t L^{\infty}_x} \lesssim \| \psi_s \|_{\mathcal{S}(I)} \label{eq:pa2inf} \\
& \| s^{\frac{1}{4}} \Psi \|_{ \Ls^\I \cap \Ls^2 L^{\infty}_t L^{8}_x} \lesssim \| \psi_s \|_{\mathcal{S}(I)} \label{eq:pa8}  \\
& \| s^{\frac{1}{2}} \Db \Psi \|_{ \Ls^\I \cap \Ls^2 L^{\infty}_t L^{4}_x} \lesssim \| \psi_s \|_{\mathcal{S}(I)} \label{eq:Dpa4}  
\end{align}
Moreover, 
\begin{align} \label{eq:Ainf1} 
&\| A \|_{ \Ls^\infty L^1_t L^\I_x} \lesssim   \| \psi_s \|_{\mathcal{S}(I)}^2 \\
& \| s^{\frac{1}{2}} A \|_{\Ls^\I \cap  \Ls^2 L^\I_t L^\I_x} \lesssim \eps \| \psi_s \|_{\mathcal{S}(I)}  \label{eq:Ainf2}
\end{align}
\end{cor} 

\begin{proof}
 To prove~\eqref{eq:pa16}  note that for each $s>0$ we have 
  \ant{
  \psi_\al(s)  = - \int_s^\infty s' \na_\al \psi_s  \,  \dsp - \int_s^\I s' A_\al \psi_s \, \dsp
}
Note that
\begin{align*}
\nrm{s^{\frac{1}{8}} \Psi(s)}_{L^{2}_{t} L^{16}_{x}}
\aleq & \int_{s}^{\infty} \bb( \frac{s}{s'} \bb)^{\frac{1}{8}} \nrm{(s')^{\frac{9}{8}} \nb_{t,x} \psi_{s}}_{L^{2}_{t} L^{16}_{x}} \dsp \\
& + \int_{s}^{\infty} \bb( \frac{s}{s'} \bb)^{\frac{1}{2}} \nrm{(s')^{\frac{1}{8}} A}_{L^{\infty}_{t,x}} \nrm{(s')^{\frac{5}{8}} \psi_{s}}_{L^{2}_{t} L^{16}_{x}} \dsp
\end{align*}
Therefore, \eqref{eq:pa16} follows by an application of Schur's test, \eqref{eq:Ainf} and \eqref{eq:ps16}.

Similarly, using Gagliardo-Nirenberg we obtain, 
\ant{
\| s^{\frac{1}{4}}\Psi (s)\|_{L^{2}_t L^{\infty}_x} \lesssim  & \int_s^\infty  \big(\frac{s}{s'} \big)^{\frac{1}{4}}  \| s' \na_{ t, x} \psi_s \|_{L^2_t L^{8}_x}^{\frac{1}{2}}\| (s')^{\frac{3}{2}}\nb \nb_{t,x} \psi_s \|_{L^2_t L^{8}_x}^{\frac{1}{2}} \, \dsp    \\
&+   \int_s^\infty  \big(\frac{s}{s'} \big)^{\frac{1}{4}}\|(s')^{\frac{1}{2}} A\|_{L^{\infty}_{t, x}}  \| (s')^{\frac{1}{2}}  \psi_s \|_{L^2_t L^{8}_x}^{\frac{1}{2}}\| s' \na_{t, x} \psi_s \|_{L^2_t L^{8}_x}^{\frac{1}{2}}  \, \dsp 
 }
 Then \eqref{eq:pa2inf} follows by Schur's test as before.

Next, we prove~\eqref{eq:pa8}. As above we have 
\ant{
\| s^{\frac{1}{4}} \Psi \|_{ L^{\infty}_t L^{8}_x} \lesssim   & \int_s^\infty  \Big( \frac{s}{s'} \Big)^{\frac{1}{4}} \| (s')^{\frac{5}{4}}  \na_{t, x} \psi_s \|_{L^\infty_t L^8_x} \,  \dsp \\
& +  \int_s^\infty  \Big( \frac{s}{s'} \Big)^{\frac{1}{4}} \| (s')^{\frac{1}{2}} A \|_{L^\infty_t L^\infty_x}  \| (s')^{\frac{3}{4}}  \psi_s \|_{L^\infty_t L^8_x} \,  \dsp \\
}
By Schur's test followed by Sobolev embedding  we  see that 
\ant{
\| s^{\frac{1}{4}} \Psi \|_{  \Ls^\I \cap \Ls^2 L^{\infty}_t L^{8}_x} &\lesssim \| s^{\frac{5}{4}}  \na_{t, x} \psi_s \|_{ \Ls^\I \cap \Ls^2 L^\infty_t L^8_x} \\
& \quad +  \| s^{\frac{1}{2}} A \|_{\Ls^\I L^\infty_t L^\infty_x}  \| s^{\frac{3}{4}}  (- \De)^{\frac{3}{4}} \psi_s \|_{\Ls^\I \cap \Ls^2 L^\infty_t L^2_x} \\
& \lesssim  \|   \psi_s \|_{ \cS(I)}
}
where the last line above follows from~\eqref{eq:pa8},~\eqref{eq:Ainf}, interpolation, and~\eqref{eq:preg8}. 

To prove~\eqref{eq:Dpa4} we claim that it suffices  to show 
\EQ{ \label{eq:more}
\| s^{\frac{1}{2}} \na \Psi \|_{\Ls^\I \cap \Ls^2 L^\infty_t L^4_x} \lesssim \| \psi_s \|_{\cS(I)}
}
To see this, note that using~\eqref{eq:pa8}, Gagliardo-Nirenberg, ~\eqref{eq:AL4}, and~\eqref{eq:dkAL2} we have 
 \begin{multline*} 
 \abs{  \| s^{\frac{1}{2}} \na \Psi \|_{\Ls^\I \cap \Ls^2 L^\infty_t L^4_x} -  \| s^{\frac{1}{2}} \Db \Psi \|_{\Ls^\I \cap \Ls^2 L^\infty_t L^4_x}}  \\
 \lesssim \| A \|^{\frac{1}{2}}_{\Ls^\I L^\I_t L^4_x}  \| s^{\frac{1}{2}}\na  A \|^{\frac{1}{2}}_{\Ls^\I L^\I_t L^4_x}  \| s^{\frac{1}{4}}   \Psi \|_{ \Ls^\I \cap \Ls^2 L^\infty_t L^8_x} \lesssim \eps \| \psi_s \|_{\cS(I)}
 \end{multline*} 
To prove~\eqref{eq:more}, we start with
\ant{
\na_b \psi_\al =-  \int_s^\I \na_b \Db_\al \psi_s \, ds'
}
It follows that 
\ant{
\| s^{\frac{1}{2}} \na \Psi \|_{ L^\infty_t L^4_x} &\lesssim   \int_{s}^\I  \Big(\frac{s}{s'} \Big)^{\frac{1}{2}} \| (s')^{\frac{3}{2}}  \na \na_{t, x}  \psi_s \|_{L^\infty_t L^4_x} \, \dsp  \\
& +  \int_{s}^\I  \Big(\frac{s}{s'} \Big)^{\frac{1}{2}} \| (s')^{\frac{1}{2}}A\|_{L^\I_t L^\I_x} \| s'   \na_{t, x}  \psi_s \|_{L^\infty_t L^4_x} \, \dsp  \\
& +  \int_{s}^\I  \Big(\frac{s}{s'} \Big)^{\frac{1}{2}}  \| s'\na A\|_{L^\I_t L^\I_x} \| (s')^{\frac{1}{2}}    \psi_s \|_{L^\infty_t L^4_x} \, \dsp  
}
By Schur's test and~\eqref{eq:Ainf} we obtain~\eqref{eq:more}. 

Next, using~\eqref{eq:AF} and~\eqref{eq:Fdu} we see that 
\ant{
\| A \|_{ \Ls^\infty L^1_t L^\I_x}  &\lesssim  \int_0^\I  \| s^{\frac{3}{4}}\psi_s \|_{L^2_t L^\I_x} \| s^{\frac{1}{4}}\Psi \|_{L^2_t L^\I_x} \, \ds \\
& \lesssim \left(  \int_0^\I  \| s^{\frac{3}{4}}\psi_s \|_{L^2_t L^\I_x}^2  \, \ds \right)^{\frac{1}{2}} \left(  \int_0^\I   \| s^{\frac{1}{4}}\Psi \|_{L^2_t L^\I_x}^2   \, \ds \right)^{\frac{1}{2}}\\ 
& \lesssim \left(  \int_0^\I  \| s^{\frac{1}{2}}\psi_s \|_{L^2_t L^8_x} \| s \na \psi_s \|_{L^2_t L^8_x} \, \ds \right)^{\frac{1}{2}} \left(  \int_0^\I   \| s^{\frac{1}{4}}\Psi\|_{L^2_t L^\I_x}^2   \, \ds \right)^{\frac{1}{2}}  \\
& \lesssim  \| \psi_s \|_{\Ls^\I \cap \Ls^2 \mathcal{S}_{s}(I)}^2 = \nrm{\psi_{s}}_{\calS(I)}^{2}
}
where we used~\eqref{eq:pa2inf} in the last line. 
Lastly, we establish~\eqref{eq:Ainf2}. Arguing similarly, we have 
\ant{
\| s^{\frac{1}{2}} A \|_{L^\I_tL^\I_x} \lesssim \int_s^\I \Big( \frac{s}{s'} \Big)^{\frac{1}{2}} \|s' \psi_s\|_{L^\I_t L^\I_x} \| (s')^{\frac{1}{2}} \Psi \|_{L^\I_t L^\I_x} \, \dsp
} 
Using Schur's test, we deduce that  
\ant{
\| s^{\frac{1}{2}} A \|_{ \Ls^\I \cap \Ls^2 L^\I_tL^\I_x} &\lesssim \|s \psi_s\|_{\Ls^\I \cap \Ls^2 L^\I_t L^\I_x} \| s^{\frac{1}{2}} \Psi \|_{\Ls^\I L^\I_t L^\I_x} \\
& \lesssim \eps \| \psi_s \|_{ \cS(I)}
}
where the last line follows from~\eqref{eq:psinfinf} and~\eqref{eq:painf}. 
This completes the proof. 
\end{proof}

 \section{Analysis of the wave equation satisfied by $\psi_s$}  \label{s:wave} 
 
In this section we complete the proof of~\eqref{eq:psap} under the bootstrap assumption~\eqref{eq:bs}. Indeed we establish the following proposition. 
\begin{prop}\label{p:pSI} Let $\psi_s$ denote the heat tension field as defined in~\eqref{eq:htf}. There exists $\eps_{2}>0$ small enough so that if the bootstrap assumptions~\eqref{eq:bs} and \eqref{eq:bsp} hold with $\eps <\eps_{2}$, then 
\EQ{ \label{eq:psisap} 
\| \psi_s \|_{\mathcal{S}(I)} &\lesssim  \| s^{\frac{1}{2}} \na_{t, x}\psi_s \rest_{t = 0} \|_{\Ls^\infty \cap \Ls^2 L^2_x} \\
& \lesssim  \| (d u_0, u_1) \|_{H^1 \times H^1(\Hp^4; T\NN)} = \eps
} 
for an implicit constant that is independent of $I$, as well as the constants $C_0$ from~\eqref{eq:bs} and $C_{1}$ from~\eqref{eq:bsp}.  \end{prop} 

\begin{rem} \label{r:ineq2} 
Note that Proposition~\ref{p:pSI} closes the bootstrap hypothesis~\eqref{eq:bsp}. The other bootstrap assumption \eqref{eq:bs} will be closed in Section~\ref{s:proof}.
We also remark that the second inequality in~\eqref{eq:psisap} has already been established using the theory of the harmonic map heat flow and our assumptions on the initial data $(u_0, u_1)$. 
\end{rem} 

Once we have established the a priori estimates in Proposition~\ref{p:ap}, the estimate~\eqref{eq:psisap} holds for the heat tension field $\psi_s$ associated to any smooth wave map $u$ on any time interval $I$ and with smooth data as in
~\eqref{eq:sd}.  One can then deduce the following uniform control of the $\mathcal{S}(I)$ norm of $\De \psi_s$. This control of higher derivatives of $\psi_s$ will be used in the proof of Theorem~\ref{t:main1}.  

\begin{cor}\label{c:hr} 
 Let $(u_0, u_1)$ be initial data as in~\eqref{eq:data},~\eqref{eq:data1} with 
\EQ{
\| (d u_0, u_1) \|_{ H^3 \times H^3( \Hp^4; T\NN)} =: B_0 < \infty
}
and let $(u, \partial_t u)$ be a smooth wave map on a time interval $I \subset \R $ with initial data $(u, \partial_t u) \rest_{t = 0}  = (u_0, u_1)$. Let $\psi_s$ denote the heat tension field associated to $u$ as defined  above. Then,  there exists $\eps_{3}>0$ small enough so that if 
\EQ{ \label{eq:psapb} 
\| \psi_s \|_{\mathcal{S}(I)}  \le  \eps_{3}
}
 we have the following uniform estimates on the $\calS(I)$ norm of $(\De \psi_s, \p_t \De \psi_s)$  in terms of the $H^3 \times H^3$ norm of the data $(d u_0, u_1)$, 
 \EQ{ \label{eq:DepsiS} 
 \|   \De \psi_s \|_{ \cS(I)} \lesssim B_0
}
\end{cor}

\begin{rem} We note that given the a priori control~\eqref{eq:psapb}, the uniform control of higher derivatives of  $\De \psi_s$ as in Corollary~\ref{c:hr} follows from essentially the same proof as the proof of Proposition~\ref{p:pSI} after taking the Laplacian of~\eqref{eq:psis}. Indeed, one obtains an effectively linear equation for $\De \psi_s$ with coefficients that are small because of~\eqref{eq:psapb}.  To avoid repetition, we omit the details of the proof of Corollary~\ref{c:hr} and instead refer the reader to the details of the proof of Proposition~\ref{p:pSI}. 

Lastly, we remark that the introduction of $\eps_{3}$ can be avoided by carefully making use of divisible norms; cf. Remark~\ref{r:extra-dlt}.
\end{rem}

The proof of Proposition~\ref{p:pSI} requires the following two additional lemmas. 
\begin{lem}\label{l:w}
Let $w  = w(s, t, x)$ denote the wave tension field defined as in~\eqref{eq:wtf}. Then, assuming~\eqref{eq:bs} and~\eqref{eq:bsp}, we have 
\begin{align} \label{eq:w12} 
&\|s^{\frac{1}{2}} \p_s w\|_{ \Ls^\I \cap \Ls^2 L^1_t L^2_x} \lesssim  \eps \| \psi_s \|_{ \mathcal{S}(I)}^2 \\
& \|s^{\frac{1}{4}}  \na w \|_{ \Ls^2 L^2_t L^2_x} \lesssim   \eps \| \psi_s\|_{\mathcal{S}(I)}^2   \label{eq:w222}
\end{align}
\end{lem} 

\begin{cor} \label{c:A83}
Under the assumptions of the previous lemma, we have
\EQ{ \label{eq:naA83}
  \| \na^\al A_\al \|_{\Ls^\I  L^{2}_t L^{\frac{8}{3}}_x} \lesssim    \| \psi_s \|_{ \cS(I)}^2
}
\end{cor} 

We postpone the proofs of Lemma~\ref{l:w}  and Corollary~\ref{c:A83} and complete the proof of Proposition~\ref{p:pSI}. 

\begin{proof}[Proof of Proposition~\ref{p:pSI}]

The starting point is~\eqref{eq:psstr} which we recall was the inequality 
\ant{
\| \psi_s \|_{\mathcal{S}(I)} &\lesssim  \| s^{\frac{1}{2}} \na_{t, x}\psi_s \rest_{t = 0} \|_{\Ls^\infty \cap \Ls^2 L^2_x} +  \| s^{\frac{1}{2}} \p_s w\|_{\Ls^\infty \cap \Ls^2 L^1_t L^2_x} \\
& \quad + \|  s^{\frac{1}{2}} \tensor{F}{_{s}^{\alp}} \psi_\al\|_{\Ls^\infty \cap \Ls^2 L^1_t L^2_x}    + \|  s^{\frac{1}{2}}\na^\al A_\al \psi_s \|_{\Ls^\infty \cap \Ls^2 L^1_t L^2_x} \\
&\quad  + \|s^{\frac{1}{2}}A^\al  \nb_\al \psi_s\|_{\Ls^\infty \cap \Ls^2 L^1_t L^2_x} + \|s^{\frac{1}{2}}A^\al A_\al  \psi_s\|_{\Ls^\infty \cap \Ls^2 L^1_t L^2_x} 
}
established in Section~\ref{s:outline} using Strichartz estimates. 

We estimate the last five terms on the right-hand side above as follows. First by Lemma~\ref{l:w} we have 
\ant{
\|s^{\frac{1}{2}} \p_s w\|_{ \Ls^\I \cap \Ls^2 L^1_t L^2_x} \lesssim  \eps \| \psi_s \|_{ \mathcal{S}(I)}^2
}
Next using Sobolev embedding, ~\eqref{eq:pa16} and~\eqref{eq:12preg} we have  
\ant{
\|s^{\frac{1}{2}} F_{s}^{\, \, \al} \psi_\al \|_{ \Ls^\I \cap \Ls^2 L^1_t L^2_x} &\lesssim \|s^{\frac{1}{8}} \Psi \|_{ \Ls^\I \cap \Ls^2 L^2_t L^{16}_x}^2 \| s^{\frac{1}{4}} \p_su \|_{ \Ls^\I L^\infty_t L^{\frac{8}{3}}_x}  \\
&\lesssim  \|s^{\frac{1}{8}} \Psi \|_{ \Ls^\I \cap \Ls^2 L^2_t L^{16}_x}^2 \| s^{\frac{1}{4}}(-\De)^{\frac{1}{4}} \p_su \|_{ \Ls^\I  L^\infty_t L^{2}_x} \\
& \lesssim  \eps \| \psi_s \|_{\cS(I)}^2 
}
For the next term we use~\eqref{eq:naA83} to deduce 
\ant{
\|  s^{\frac{1}{2}}\na^\al A_\al \psi_s \|_{\Ls^\infty \cap \Ls^2 L^1_t L^2_x} &\lesssim \|\na^\al A_\al \|_{ \Ls^\infty L^2_t L^{\frac{8}{3}}_x} \| s^{\frac{1}{2}}\psi_s \|_{\Ls^\I \cap \Ls^2 L^2_t L^8_x} \\
&\lesssim  \eps \|\psi_s \|^2_{\cS(I)}
}
Similarly, this time using~\eqref{eq:Ainf1} we see that  
\ant{
\|s^{\frac{1}{2}}A^\al  \nb_\al \psi_s\|_{\Ls^\infty \cap \Ls^2 L^1_t L^2_x} &\lesssim \| A\|_{\Ls^\I  L^1_t L^\I_x} \| s^{\frac{1}{2}} \na_{t,x} \psi_s \|_{\Ls^\I \cap \Ls^2 L^\infty_t L^2_x} \\
& \lesssim \eps \|\psi_s \|^2_{\cS(I)}
}
Finally, we use~\eqref{eq:AL4}, ~\eqref{eq:Ainf1}, and Sobolev embedding to estimate
\ant{
\|s^{\frac{1}{2}}A^\al A_\al  \psi_s\|_{\Ls^\infty \cap \Ls^2 L^1_t L^2_x}  & \lesssim  \|A \|_{\Ls^\I L^\I_t L^4_x} \|A \|_{\Ls^\I  L^1_t L^\I_x} \|s^{\frac{1}{2}} \psi_s \|_{\Ls^\I \cap \Ls^2 L^\infty_t L^4} \\
&\lesssim \eps^3 \|\psi_s \|^2_{\cS(I)}
}
Putting this together we have proved that 
\EQ{ \label{eq:ap1} 
\| \psi_s \|_{\mathcal{S}(I)} \lesssim  \| s^{\frac{1}{2}} \na_{t, x}\psi_s \rest_{t = 0} \|_{\Ls^\infty \cap \Ls^2 L^2_x} + \eps \| \psi_s \|_{\mathcal{S}(I)}^2 
}
By the usual continuity argument it follows that 
\ant{
\| \psi_s \|_{\mathcal{S}(I)} \lesssim  \| s^{\frac{1}{2}} \na_{t, x}\psi_s \rest_{t = 0} \|_{\Ls^\infty \cap \Ls^2 L^2_x} 
}
as desired. In view of Remark~\ref{r:ineq2} this completes the proof of Proposition~\ref{p:pSI}. 
\end{proof}

It remains to prove Lemma~\ref{l:w} and Corollary~\ref{c:A83}. 

\begin{proof}[Proof of Lemma~\ref{l:w}] 
We begin with the proof of~\eqref{eq:w12}. We first show that 
\EQ{ \label{eq:wdew}
\|s^{-\frac{1}{2}} w\|_{\Ls^\I \cap \Ls^2 L^1_tL^2_x}+ \| s^{\frac{1}{2}} \De w\|_{\Ls^\I \cap \Ls^2 L^1_tL^2_x} \lesssim \eps \| \psi_s \|_{\cS(I)}^2
}
Recall that $w$ obeys the covariant heat equation~\eqref{eq:hw}. Expanding out the right-hand side using~\eqref{eq:Fdu} and~\eqref{eq:DF} gives
\ant{
 \p_s w - \Db^a \Db_a w  &=  \bfR^{(0)}( w,  \psi^a) \psi_a + 3 \bfR^{(0)}( \psi_\mu, \psi^a) \Db^\mu \psi_a +  \bfR^{(1)}( \psi^\mu;  \psi_\mu ,\psi^a) \psi_a  \\
 & \quad + \bfR^{(1)}( \psi^a;  \psi^\mu, \psi_a) \psi_\mu + \bfR^{(0)}(\psi_\mu, \Db^\mu \psi^a) \psi_a + \bfR^{(0)}( \Db^a \psi^\mu, \psi_a) \psi_\mu
}
On the other hand, expanding out the left-hand side using $\Db_a = \na_a + A_a$ yields, 
\begin{align} 
\p_s w - \De w &= -2 A^a \p_a w - \na^a A_a w - A^a A_a w  \notag \\
&  \,  +  \bfR^{(0)}( w,  \psi^a) \psi_a + 3 \bfR^{(0)}( \psi_\mu, \psi^a) \Db_a \psi^\mu +  \bfR^{(1)}( \psi^\mu;  \psi_\mu ,\psi^a) \psi_a  \notag \\
 & \, + \bfR^{(1)}( \psi^a;  \psi^\mu, \psi_a) \psi_\mu + \bfR^{(0)}(\psi_\mu, \Db^a \psi^\mu) \psi_a + \bfR^{(0)}( \Db^a \psi^\mu, \psi_a) \psi_\mu
 \notag \\
 &=: G_2(s) \label{eq:weq}
\end{align}

Now, since $w(s ,t, x)$ satisfies $w(0, t, x) = 0$ we can write 
\EQ{ \label{eq:s-w} 
s^{-\frac{1}{2}}w(s) =  \int_0^{\frac{s}{2}} s^{-\frac{1}{2}}e^{(s-s') \De} G_2(s') \, ds' +  \int_{\frac{s}{2}}^s s^{-\frac{1}{2}} e^{(s-s') \De} G_2(s') \, ds'
}
It follows that 
\EQ{ \label{eq:dew-s}
\hspace{-.1in} s^{\frac{1}{2}} \De w = \int_0^{\frac{s}{2}} s^{\frac{1}{2}} \De e^{(s-s') \De} G_2(s') \, ds' -  \int_{\frac{s}{2}}^s s^{\frac{1}{2}} (-\De)^{\frac{1}{2}}e^{(s-s') \De} (-\De)^{\frac{1}{2}}G_2(s') \, ds'
}
Next, we use~\eqref{eq:pest} on~\eqref{eq:s-w} and then~\eqref{eq:depest} and~\eqref{eq:napest} on~\eqref{eq:dew-s} to deduce that 
\ant{
\| s^{-\frac{1}{2}}  w\|_{L^1_tL^2_x}& \lesssim \int_0^{\frac{s}{2}} \Big( \frac{s'}{s} \Big)^{\frac{1}{2}} \| (s')^{\frac{1}{2}} G_2(s')\|_{L^1_t L^2_x} \, \dsp   \\
 &\quad  + \int_{\frac{s}{2}}^s \Big( \frac{s'}{s} \Big)^{\frac{1}{2}} \|(s')^{\frac{1}{2}} G_2(s')\|_{L^1_t L^2_x} \, \dsp 
} 
and
\ant{
\| s^{\frac{1}{2}} \De w\|_{L^1_tL^2_x}& \lesssim \int_0^{\frac{s}{2}} \frac{s^{\frac{1}{2}}(s')^{\frac{1}{2}}}{s-s'} \| (s')^{\frac{1}{2}} G_2(s')\|_{L^1_t L^2_x} \, \dsp   \\
 &\quad  + \int_{\frac{s}{2}}^s \frac{s^{\frac{1}{2}}}{(s-s')^{\frac{1}{2}}} \|s' \na G_2(s')\|_{L^1_t L^2_x} \, \dsp \\
  &  \lesssim  \int_0^{\frac{s}{2}} \Big(\frac{s'}{s} \Big)^{\frac{1}{2}} \| (s')^{\frac{1}{2}} G_2(s')\|_{L^1_t L^2_x} \, \dsp   \\
  &\quad + \int_{\frac{s}{2}}^s \frac{s^{\frac{1}{2}}}{(s-s')^{\frac{1}{2}}} \|s' \na G_2(s')\|_{L^1_t L^2_x} \, \dsp \\
}
Applying Schur's test to the two previous displayed inequalities we obtain 
\ant{
\|s^{-\frac{1}{2}} w\|_{\Ls^\I \cap \Ls^2 L^1_tL^2_x}+ \| s^{\frac{1}{2}} \De w\|_{\Ls^\I \cap \Ls^2 L^1_tL^2_x} &\lesssim \| s^{\frac{1}{2}} G_2\|_{\Ls^\I \cap \Ls^2 L^1_tL^2_x} \\
& \quad +  \|s \na G_2\|_{\Ls^\I \cap \Ls^2 L^1_tL^2_x}
}
Hence it suffices to show that 
\EQ{ \label{eq:G2} 
 \| s^{\frac{1}{2}} G_2\|_{\Ls^\I \cap \Ls^2 L^1_tL^2_x} & \lesssim  \eps^{2} \Big( \|s^{-\frac{1}{2}} w\|_{\Ls^\I \cap \Ls^2 L^1_tL^2_x}+ \| s^{\frac{1}{2}} \De w\|_{\Ls^\I \cap \Ls^2 L^1_tL^2_x} \Big) \\
& \quad + \eps \| \psi_s\|_{\cS(I)}^2
}
and 
\EQ{ \label{eq:naG2} 
  \|s \na G_2\|_{\Ls^\I \cap \Ls^2 L^1_tL^2_x} &\lesssim  \eps^{2} \Big( \|s^{-\frac{1}{2}} w\|_{\Ls^\I \cap \Ls^2 L^1_tL^2_x}+ \| s^{\frac{1}{2}} \De w\|_{\Ls^\I \cap \Ls^2 L^1_tL^2_x}\Big) \\
& \quad + \eps \| \psi_s\|_{\cS(I)}^2
}
We first estimate each of the terms that arise on the left-hand side of~\eqref{eq:G2}, recalling that $G_2(s)$ is defined in~\eqref{eq:weq}.  First, using~\eqref{eq:Ainf}  we have 
\ant{
\| s^{\frac{1}{2}}A^a \p_a w\|_{\Ls^\I \cap \Ls^2 L^1_tL^2_x} &\lesssim \| s^{\frac{1}{2}}A \|_{\Ls^\infty L^\I_t L^\I_x} \| \na w\|_{\Ls^\I \cap \Ls^2 L^1_tL^2_x}  \\
&\lesssim \eps^2 \|s^{-\frac{1}{2}} w\|_{\Ls^\I \cap \Ls^2 L^1_tL^2_x}^{\frac{1}{2}}  \| s^{\frac{1}{2}} \De w\|_{\Ls^\I \cap \Ls^2 L^1_tL^2_x}^{\frac{1}{2}}
}
and 
\ant{
\| s^{\frac{1}{2}}\na^aA_a  w\|_{\Ls^\I \cap \Ls^2 L^1_tL^2_x} &\lesssim \| s \na A \|_{\Ls^\infty L^\I_t L^\I_x} \| s^{-\frac{1}{2}} w\|_{\Ls^\I \cap \Ls^2 L^1_tL^2_x}  \\
&\lesssim \eps^2 \|s^{-\frac{1}{2}} w\|_{\Ls^\I \cap \Ls^2 L^1_tL^2_x}
}
and
\ant{
\| s^{\frac{1}{2}}A^a A_a w\|_{\Ls^\I \cap \Ls^2 L^1_tL^2_x} &\lesssim \| s^{\frac{1}{2}}A \|_{\Ls^\infty L^\I_t L^\I_x}^2 \| s^{-\frac{1}{2}} w\|_{\Ls^\I \cap \Ls^2 L^1_tL^2_x}  \\
&\lesssim \eps^4 \|s^{-\frac{1}{2}} w\|_{\Ls^\I \cap \Ls^2 L^1_tL^2_x}
}
Next, note that 
\ant{
\| s^{\frac{1}{2}}\bfR^{(0)}( w,  \psi^a) \psi_a \|_{\Ls^\I \cap \Ls^2 L^1_tL^2_x} & \lesssim  \| s^{-\frac{1}{2}} w\|_{\Ls^\I \cap \Ls^2 L^1_tL^2_x}  \| s^{\frac{1}{2}} \Psi \|^2_{ \Ls^\I  L^\I_t L^\I_x} \\
&\lesssim  \eps^2  \| s^{-\frac{1}{2}} w\|_{\Ls^\I \cap \Ls^2 L^1_tL^2_x}
}
where we have used~\eqref{eq:painf} in the last line above. 

Next, we observe that using~\eqref{eq:AL4},~\eqref{eq:pa16},  Sobolev embedding, and Gagliardo--Nirenberg  we have 
\ant{
\| &s^{\frac{1}{2}}\bfR^{(0)}( \psi_\mu, \psi^a) \Db_a \psi^\mu \|_{\Ls^\I \cap \Ls^2 L^1_tL^2_x}   \\
&\lesssim  \| s^{\frac{1}{8} }\Psi \|_{\Ls^\I \cap \Ls^2 L^2_tL^{16}_x}^2 \Big(  \| s^{\frac{1}{4}} \na \Psi \|_{ \Ls^\I L^\I_t  L^{\frac{8}{3}}_x} + \| s^{\frac{1}{4}} \Psi \|_{ \Ls^\I L^\I_t  L^{8}_x} \| A \|_{\Ls^\I L^\I_t L^4_x} \Big)  \\
& \lesssim  \| \psi_s \|_{ \cS(I)}^2   \|  \na \Psi \|_{\Ls^\I L^\I_t L_{x}^2}^{\frac{1}{2}} \|s^{\frac{1}{2}}  \na^2 \Psi\|_{\Ls^\I L^\I_t L_{x}^2}^{\frac{1}{2}}   \lesssim \eps  \| \psi_s \|_{ \cS(I)}^2
}
where in the last line above we used~\eqref{eq:eqnak}. Arguing similarly we have 
\begin{multline*}
 \| s^{\frac{1}{2}}\bfR^{(1)}( \psi^\mu;  \psi_\mu ,\psi^a) \psi_a \|_{\Ls^\I \cap \Ls^2 L^1_tL^2_x}  \\
  \lesssim  \| s^{\frac{1}{8}} \Psi \|_{\Ls^\I \cap \Ls^2 L^2_t L^{16}_x}^2  \| s^{\frac{1}{4}} \Psi \|_{ \Ls^\I L^\I_t  L^{8}_x}  \| \Psi \|_{ \Ls^\I L^\I_t  L^{4}_x} 
  \lesssim \eps^2 \| \psi_s \|_{\cS(I)}^2
 \end{multline*}
The remaining  terms in~\eqref{eq:weq} are controlled in a similar fashion as the previous two,  thus completing the proof of~\eqref{eq:G2}. A nearly identical argument is used to prove~\eqref{eq:naG2}. This proves~\eqref{eq:wdew}. Finally, to complete the proof of~\eqref{eq:w12} we note that from equation~\eqref{eq:weq} we have 
\ant{
 \| s^{\frac{1}{2}} \p_s w\|_{\Ls^\I \cap \Ls^2 L^1_tL^2_x} &\lesssim \| s^{\frac{1}{2}} \De w\|_{\Ls^\I \cap \Ls^2 L^1_tL^2_x} + \| s^{\frac{1}{2}} G_2 \|_{\Ls^\I \cap \Ls^2 L^1_tL^2_x} \\
 & \lesssim  \eps \| \psi_s\|_{\cS(I)}^2
 }
 where we have used~\eqref{eq:G2} followed by~\eqref{eq:wdew} in the last line above. 
 
 Next, we establish~\eqref{eq:w222}. First note that it suffices to show that 
 \EQ{ \label{eq:w222a} 
 \|s^{\frac{1}{4}}  (-\De)^{\frac{1}{2}} w \|_{ \Ls^2 L^2_t L^2_x} \lesssim   \eps \| \psi_s\|_{\mathcal{S}(I)}^2
 }
 Using that $w(0, t, x) = 0$ from the Duhamel formula  and with $G_2$ as above we have 
 \ant{
 s^{\frac{1}{4}} (-\De)^{\frac{1}{2}} w (s)  =  \int_0^s s^{\frac{1}{4}}(-\De)^{\frac{1}{2}} e^{(s-s') \De} G_2(s') \, ds'
 }
 Using~\eqref{eq:napest} we see that 
 \ant{
 \|s^{\frac{1}{4}} (-\De)^{\frac{1}{2}} w(s) \|_{L^2_tL^2_x} \lesssim \int_0^s \frac{ s^{\frac{1}{4}} (s')^{\frac{1}{4}}}{(s-s')^{\frac{1}{2}}} \| (s')^{\frac{3}{4}}G_2(s')\|_{L^2_t L^2_x} \,  \dsp
 }
By Schur's test we then deduce that 
\ant{
\|s^{\frac{1}{4}} (-\De)^{\frac{1}{2}} w(s) \|_{\Ls^2 L^2_tL^2_x} \lesssim  \| s^{\frac{3}{4}} G_2(s) \|_{\Ls^2 L^2_t L^2_x}
}
Therefore, it suffices to show that 
\EQ{ \label{eq:G22} 
\| s^{\frac{3}{4}} G_2(s) \|_{\Ls^2 L^2_t L^2_x} \lesssim    \eps^{2} \| s^{\frac{1}{4}} (-\De)^{\frac{1}{2}} w \|_{\Ls^2 L^2_tL^2_x}  +  \eps  \| \psi_s \|_{ \cS(I)}^2
}
We estimate each of the terms in $G_2$, which is defined in~\eqref{eq:weq}.  We first control all of the terms involving $w$. Using~\eqref{eq:Ainf2} we have 
\ant{
\| s^{\frac{3}{4}} A^a \p_a w \|_{\Ls^2 L^2_t L^2_x} &\lesssim  \| s^{\frac{1}{2}} A \|_{\Ls^\I L^\I_t L^\I_x}  \| s^{\frac{1}{4}} \na w \|_{\Ls^2 L^2_t L^2_x}  \\
& \lesssim  \eps^{2}  \| s^{\frac{1}{4}} (-\De)^{\frac{1}{2}} w \|_{\Ls^2 L^2_t L^2_x} 
}
Next, we use~\eqref{eq:dkAL2} to deduce that 
\ant{
\| s^{\frac{3}{4}} \na^a A_a w \|_{\Ls^2 L^2_t L^2_x} &\lesssim  \| s^{\frac{1}{2}}  \na A \|_{ \Ls^\I L^\I_t L_{x}^4} \|  s^{\frac{1}{4}}  w \|_{\Ls^2 L^2_t L^4_x}  \\
& \lesssim \| s^{\frac{1}{2}}  \na^2 A \|_{ \Ls^\I L^\I_t L_{x}^2} \|  s^{\frac{1}{4}} (-\De)^{\frac{1}{2}}  w \|_{\Ls^2 L^2_t L^2_x}  \\
& \lesssim \eps^{2}  \|  s^{\frac{1}{4}}  (-\De)^{\frac{1}{2}} w \|_{\Ls^2 L^2_t L^2_x}
}
 Continuing as above, using~\eqref{eq:Ainf2} and~\eqref{eq:AL4} we observe 
 \ant{
 \| A^a A_a w  \|_{\Ls^2 L^2_t L^2_x}  & \lesssim  \|s^{\frac{1}{2}} A \|_{\Ls^\I L^\I_t L^\I_x } \| A \|_{ \Ls^\I L^\I_t L^4_x} \| s^{\frac{1}{4}} w \|_{\Ls^2 L^2_t L^4_x} \\
 & \lesssim  \eps^4  \|  s^{\frac{1}{4}}  (-\De)^{\frac{1}{2}} w \|_{\Ls^2 L^2_t L^2_x}
 }
 And lastly, using~\eqref{eq:pa8} and~\eqref{eq:bsp} we see that 
 \ant{
 \|  \bfR^{(0)}( w,  \psi^a) \psi_a \|_{\Ls^2 L^2_t L^2_x}  & \lesssim   \| s^{\frac{1}{4}} w \|_{\Ls^2 L^2_t L^4_x}   \| s^{\frac{1}{4}} \Psi \|_{\Ls^\I  L^\I_t L^8_x}^2 \\
 & \lesssim   \|  s^{\frac{1}{4}}  (-\De)^{\frac{1}{2}} w \|_{\Ls^2 L^2_t L^2_x} \| \psi_s \|_{ \cS(I)}^2 \\
 & \lesssim \eps^2 \|  s^{\frac{1}{4}}  (-\De)^{\frac{1}{2}} w \|_{\Ls^2 L^2_t L^2_x}
 }
 Next, we estimate the terms that do not contain $w$. To start we use~\eqref{eq:pa16},~\eqref{eq:pa8}, and~\eqref{eq:DP165}  to establish that 
 \begin{multline*}
  \|s^{\frac{3}{4}} \bfR^{(0)}( \psi_\mu, \psi^a) \Db_a \psi^\mu   \|_{\Ls^2 L^2_t L^2_x}  \\ \lesssim  \| s^{\frac{1}{8}}  \Psi \|_{ \Ls^2L^2_t L^{16}_x} \| s^{\frac{1}{4}} \Psi \|_{ \Ls^\I L^\I_t L^8_x} \| s^{\frac{3}{8}} \Db  \Psi \|_{ \Ls^\I L^\I_t  L^{\frac{16}{5}}_x}  
  \lesssim  \eps  \| \psi_s \|_{ \cS(I)}^2
  \end{multline*} 
  Similarly, we use~\eqref{eq:pa16},~\eqref{eq:pa8}, Sobolev embedding,~\eqref{eq:DP165}, and~\eqref{eq:eqna} to show that 
  \ant{
   \| s^{\frac{3}{4}} & \bfR^{(1)}( \psi^\mu;  \psi_\mu ,\psi^a) \psi_a \|_{\Ls^2 L^2_t L^2_x}   \\
   & \lesssim   \| s^{\frac{1}{8}}  \Psi \|_{ \Ls^2L^2_t L^{16}_x} \| s^{\frac{1}{4}} \Psi \|_{ \Ls^\I L^\I_t L^8_x} \| s^{\frac{3}{8}}  \Psi \|_{ \Ls^\I L^\I_t L^{16}_x} \| \Psi \|_{ \Ls^\I L^\I_t L^4_x} \\
   & \lesssim  \| \psi_s \|_{ \cS(I)}^2   \| s^{\frac{3}{8}}   \na  \Psi \|_{ \Ls^\I L^\I_t L^{\frac{16}{5}}_x} \|  \na  \Psi \|_{ \Ls^\I L^\I_t L^2_x} \\
   & \lesssim \eps^2  \| \psi_s \|_{ \cS(I)}^2  
  }
 The remaining terms in~\eqref{eq:weq} can be handled in an identical fashion as the previous two, concluding the proof of~\eqref{eq:G22}.  This completes the proof of~\eqref{eq:w222}. 
\end{proof}

Lastly, we prove Corollary~\ref{c:A83}. 

\begin{proof}[Proof of Corollary~\ref{c:A83}]
Recall that 
\ant{
A_\al(s)  =  - \int_s^\I s' F_{s \al}  \dsp
}
It follows that  
\ant{
\na^\al A_\al(s) =  - \int_{s}^\I  s' \Db^\al F_{s \al}  \, \dsp  + \int_s^\I s' [A^\al, F_{s \al}]  \, \dsp
}
and thus, 
\EQ{ \label{eq:naA283}
 \| \na^\al A_\al \|_{ \Ls^\I L^2_t L_{x}^{\frac{8}{3}}}  \lesssim   \int_{0}^\I   \| s' \Db^\al F_{s \al} \|_{L^2_t L^{\frac{8}{3}}_x}  \, \dsp  +  \int_0^\I \|s' [A^\al ,F_{s \al}]\|_{L^2_t L^{\frac{8}{3}}_x}   \, \dsp
}
Using the notation from Section~\ref{s:curv} and recalling that $w = \Db^\al \psi_\al$ we write 
\ant{
 \Db^\al F_{s \al} =  \bfR^{(1)}( \psi^\al; \psi_s, \psi_\al) +  \bfR^{(0)}(  \Db^\al \psi_s, \psi_\al) +  \bfR^{(0)}( \psi_s , w) 
 }
 and we estimate each term on the right as follows: 
 \ant{
 & \|  s' \bfR^{(1)}( \psi^\al; \psi_s, \psi_\al) \|_{L^2_t L^{\frac{8}{3}}_x} \lesssim  \| (s')^{\frac{1}{2}} \psi_s \|_{L^2_tL^8_x}  \| (s')^{\frac{1}{4}} \Psi \|_{L^{\infty}_t L^8_x}^2 }
 Next, 
 \ant{
    \|  \bfR^{(0)}(  \Db^\al \psi_s, \psi_\al) \|_{L^2_t L^{\frac{8}{3}}_x}& \lesssim  \| (s')^{\frac{7}{8}}  \Db_{t, x} \psi_s \|_{L^{\infty}_t L^{\frac{16}{5}}_x} \|(s')^{\frac{1}{8}} \Psi \|_{L^2_t L^{16}_x} \\
  &   \lesssim \| (s')^{\frac{7}{8}}  \na_{t, x} \psi_s \|_{L^{\infty}_t L^{\frac{16}{5}}_x} \|(s')^{\frac{1}{8}} \Psi \|_{L^2_t L^{16}_x} \\
  }
  where the last line follows using \eqref{eq:AL4}.
 For the term involving $w$ we have  
  \ant{
    \|\bfR^{(0)}( \psi_s , w)  \|_{L^2_t L^{\frac{8}{3}}_x}  &\lesssim   \| (s')^{\frac{3}{4}} \psi_s \|_{L^\I_t L^8_x} \| (s')^{\frac{1}{4}} w \|_{L^2_t L^4_x} \\
    & \lesssim  \|  (s')^{\frac{1}{2}} \psi_s \|_{L^\infty_t L_{x}^4}^{\frac{1}{2}} \| s' \na \psi_s \|_{L^\I_t L_{x}^4}^{\frac{1}{2}} \| (s')^{\frac{1}{4}} \na w \|_{L^2_t L^2_x}\\
        & \lesssim  \|  (s')^{\frac{1}{2}} \na \psi_s \|_{L^\infty_t L^2_x}^{\frac{1}{2}} \| s' \na^2 \psi_s \|_{L^\I_t L^2_x}^{\frac{1}{2}} \|  (s')^{\frac{1}{4}} \na w \|_{L^2_t L^2_x}
 }
 For the last term on the right-hand side of~\eqref{eq:naA283} we will use the estimate 
 \ant{
 \|s' [A^\al, F_{s \al}]\|_{L^2_t L^{\frac{8}{3}}_x} & \lesssim \| (s')^{\frac{1}{4}} A \|_{L^\I_t L^8_x} \| (s')^{\frac{1}{2}} \psi_s  \|_{L^2_t L^8_x}  \| (s')^{\frac{1}{4}}  \Psi \|_{L^\I_t L^8_x} \\ 
 & \lesssim  \| A \|_{L^\I_t L^4_x}^{\frac{1}{2}}  \| (s')^{\frac{1}{2}} \na  A \|_{L^\I_t L^4_x}^{\frac{1}{2}}  \| (s')^{\frac{1}{2}} \psi_s  \|_{L^2_t L^8_x}  \| (s')^{\frac{1}{4}}  \Psi \|_{L^\I_t L^8_x} \\
 & \lesssim \eps^2 \| (s')^{\frac{1}{2}} \psi_s  \|_{L^2_t L^8_x}  \| (s')^{\frac{1}{4}}  \Psi \|_{L^\I_t L^8_x}
 } 
 where the last line follows from Sobolev embedding,~\eqref{eq:AL4}, and~\eqref{eq:dkAL2}. 
Plugging all of the preceding estimates back into~\eqref{eq:naA283} we obtain
\ant{
 \| \na^\al A_\al \|_{ \Ls^\I L^2_t L_{x}^{\frac{8}{3}}}  & \lesssim  \int_0^\I  \| (s')^{\frac{1}{2}} \psi_s \|_{L^2_tL^8_x}  \| (s')^{\frac{1}{4}} \Psi \|_{L^{\infty}_t L^8_x}^2 \,   \dsp  \\
 & \quad +  \int_0^\I  \| (s')^{\frac{7}{8}}  \na_{t, x} \psi_s \|_{L^{\infty}_t L^{\frac{16}{5}}_x} \|(s')^{\frac{1}{8}} \Psi \|_{L^2_t L^{16}_x} \, \dsp \\
 & \quad +  \int_0^\I \|  (s')^{\frac{1}{2}} \na \psi_s \|_{L^\infty_t L^2_x}^{\frac{1}{2}} \| s' \na^2 \psi_s \|_{L^\I_t L^2_x}^{\frac{1}{2}} \| (s')^{\frac{1}{4}} \na w \|_{L^2_t L^2_x}
 \, \dsp  \\
 &\quad +  \eps^2 \int_0^\I   \| (s')^{\frac{1}{2}} \psi_s  \|_{L^2_t L^8_x}  \| (s')^{\frac{1}{4}}  \Psi \|_{L^\I_t L^8_x}  \, \dsp \\
 &  \lesssim \| \psi_s \|_{ \cS(I)}^2
}
Here the estimate on the last line follows by several application of Cauchy-Schwarz along with~\eqref{eq:ps165},~\eqref{eq:pa16},~\eqref{eq:pa8}, and~\eqref{eq:w222}.  This concludes the proof. 
\end{proof}

\section{Proofs of Proposition~\ref{p:ap} and Theorem~\ref{t:main1}} \label{s:proof}

The first goal of this section is to complete the proof of Proposition~\ref{p:ap} by closing the bootstrap~\eqref{eq:bs}$\Longrightarrow$\eqref{eq:bsc}.  We then conclude with the proof of Theorem~\ref{t:main1}. 

We begin by establishing the following a priori estimates,  which together with Proposition~\ref{p:pSI} can be used to deduce Proposition~\ref{p:ap}.

\begin{prop}\label{p:s=0}
Let $\psi_s$ denote the heat tension field defined in~\eqref{eq:htf} and let  $\Psi  =  \psi_\al dx^\al$ as above. Then 
\EQ{
\|  \Db   \Psi  \rest_{s=0} \|_{L^\infty_t L^2_x( I \times \Hp^4)}  \lesssim \|\psi_s \|_{S(I)}.
}
where the implicit constant above is independent of $C_0$ in~\eqref{eq:bs}. 
\end{prop} 

Before beginning the proofs of Proposition~\ref{p:s=0} and Proposition~\ref{p:ap}, we first note that once these results have been established we can also prove uniform estimates on higher derivatives of $\Psi$ assuming appropriate control on higher Sobolev norms of the initial data $(u_0, u_1)$.  As in Corollary~\ref{c:hr} the key point here is that $\| \psi_s \|_{ \cS(I)}$ is a controlling norm for all of the dynamic variables. We state the following corollary of Proposition~\ref{p:pSI}, Corollary~\ref{c:hr} and Proposition~\ref{p:s=0}, which will be needed in the proof of Theorem~\ref{t:main1}. We omit the proof as it is nearly identical to the proof of Proposition~\ref{p:s=0} after commuting the covariant Laplacian, $\Db^b \Db_b$, with~\eqref{eq:covheat} and integrating by parts to control the full third order covariant derivative $\Db^{3} \Psi$ in $L^{2}_{x}$.

\begin{cor} \label{c:hr1} 
Let $(u_0, u_1)$ be initial data as in~\eqref{eq:data},~\eqref{eq:data1} with  
\EQ{
\| (d u_0, u_1) \|_{ H^3 \times H^3( \Hp^4; T\NN)} =: B_0 < \infty
}
and let $(u, \partial_t u)$ be a smooth wave map on a time interval $I \subset \R $ with initial data $(u, \partial_t u) \rest_{t = 0}  = (u_0, u_1)$. Let $\psi_s$ be the associated heat tension field and $\Psi = \psi_\al dx^\al$ be the representation of $d_{t, x} u$ in the caloric gauge as defined as above.  Then there exists $\eps_{4}>0$ with the following property.  Let $\eps_{3}>0$ be as in Corollary~\ref{c:hr} and assume that $\| \psi_s \|_{ \cS(I)} = \eps<   \eps_{3} < \eps_{4}$ is small enough so that the hypothesis of Corollary~\ref{c:hr} are satisfied so that  
 \EQ{ \label{eq:DepsiS1} 
 \|   \De \psi_s \|_{ \cS(I)} \lesssim B_0
}
holds. Then $\eps_{4}>0$ can be chosen small enough so that 
\EQ{ 
 \|  \Db^3 \Psi \rest_{s=0} \| _{ L^\infty_t L^2_x( I \times \Hp^4) } \lesssim B_0.
}
with an implicit constant that is independent of $I$. 
\end{cor} 

We now turn to the proofs of Proposition~\ref{p:s=0}, Proposition~\ref{p:ap}, and Theorem~\ref{t:main1}. First we assume Proposition~\ref{p:s=0} and use it to deduce Proposition~\ref{p:ap}. 

\begin{proof}[Proof of Proposition~\ref{p:ap} assuming Proposition~\ref{p:s=0}]
Recall that Proposition~\ref{p:ap} follows by showing that there exist constants $\eps_{1}>0$ and $C_0>0$ so that~\eqref{eq:bs} implies~\eqref{eq:bsc} for all $\eps \le \eps_{1}$. Requiring first that $\eps, \eps_{1}>0$ satisfy $ \eps   \le \eps_{1} < \eps_{2}$, where $\eps_{2}>0$ is the constant in Proposition~\ref{p:pSI} we can combine Proposition~\ref{p:pSI} and Proposition~\ref{p:s=0} to obtain 
\EQ{ \label{eq:DPpSI} 
\|  \Db   \Psi  \rest_{s=0} \|_{L^\infty_t L^2_x( I \times \Hp^4)} & \lesssim \| \psi_s \|_{\mathcal{S}(I)} \lesssim   \| s^{\frac{1}{2}} \na_{t, x}\psi_s \rest_{t = 0} \|_{\Ls^\infty \cap \Ls^2 (L^2_x)} \\
& \lesssim  \| (du_0, u_1) \|_{H^1 \times H^1(\Hp^4; T\NN)}   = \eps 
}
for an implicit constant that is independent of the choice of $C_0$ in~\eqref{eq:bs}. Therefore it suffices to show that 
\EQ{ \label{eq:eqnorm} 
  \|   \na d_{t, x} u  \rest_{s=0} \|_{L^\infty_t L^2_x( I \times \Hp^4)} \lesssim  \|  \Db   \Psi  \rest_{s=0} \|_{L^\infty_t L^2_x( I \times \Hp^4)}
}
for an implicit constant independent of $C_0$. Finally to prove~\eqref{eq:eqnorm} we note
that by using the schematic $e(\Db\Psi) =  D d_{t,x}u  =  \na d_{t,x} u + S(u)( du, d_{t,x}u),$ we can show that for any time interval $J  \subset I$ we have 
\ant{
 \|   \na  d_{t, x}u  \rest_{s=0} \|_{L^\infty_t L^2_x( J \times \Hp^4)} \lesssim  \|  \Db   \Psi  \rest_{s=0} \|_{L^\infty_t L^2_x( J \times \Hp^4)} +  \|   \na  d_{t, x} u  \rest_{s=0} \|_{L^\infty_t L^2_x( J \times \Hp^4)}^2.
 }
The estimate \eqref{eq:eqnorm} now follows from the usual continuity argument. To conclude we choose $C_0>0$ sufficiently large to account for the constants in~\eqref{eq:DPpSI} and~\eqref{eq:eqnorm} and then $\eps_{1}>0$ sufficiently small relative to $C_0$. This completes the proof. 
%
%
%
%
\end{proof} 

Before turning to the proof of Proposition~\ref{p:s=0} we state two auxiliary lemmas which will be required for the proof.
\begin{lem} \label{l:s=0 temp 1}Let $\psi_s$  be as in the statement of Proposition \ref{p:s=0}. Then
\EQ{ \label{eq:Db123} 
& \| s^{\frac{1}{2}}  \Db_{t, x} \psi_s \|_{ \Ls^\infty \cap \Ls^2 L^{\infty}_tL^2_x}  \lesssim  \| \psi_s \|_{ \cS(I)},\\
& \| s \Db \Db_{t, x} \psi_s \|_{ \Ls^\infty \cap \Ls^2 L^{\infty}_tL^2_x}  \lesssim  \| \psi_s \|_{ \cS(I)},\\
&  \| s^{\frac{3}{2}} \Db^2 \Db_{t, x} \psi_s \|_{ \Ls^\infty \cap \Ls^2 L^{\infty}_tL^2_x}  \lesssim  \| \psi_s \|_{ \cS(I)}.
}
\end{lem}

\begin{proof}
This is a consequence of Proposition~\ref{p:preg8}, as well as bounds on $A$ in \eqref{eq:AL4} and Lemma~\ref{l:PA}. In order to replace $\nabla$ in \eqref{eq:preg8} by $\Db$  we note that for any smooth $E_{n}$-valued tensor $\Phi$ we have  $\Db \Phi =   \na \Phi +  A \Phi$, and thus  by Sobolev embedding, 
\ali{\label{s0 temp 2}
\| \Db_{t, x} \Phi\|_{\Ls^2L_t^\infty L_x^2} &\lesssim \|\nabla_{t, x} \Phi\|_{\Ls^2L_t^\infty L_x^2}+\|A\|_{\Ls^\infty L_t^\infty L_x^4}\|\nabla \Phi\|_{\Ls^2L_t^\infty L_x^2} \\
& \lesssim \|\nabla_{t, x} \Phi\|_{\Ls^2L_t^\infty L_x^2},
} 
Now apply~\eqref{s0 temp 2} to $\Phi = \psi_{s}, \Db_{t, x} \psi_s$, and $\Db \Db_{t, x} \psi_s$, followed by~\eqref{eq:AL4},~\eqref{eq:Ainf} and Proposition~\ref{p:preg8}.
\end{proof} 

\begin{lem} \label{l:s=0 temp 2} Let   $\psi_s$ and $\Psi$ be as in the statement of Proposition~\ref{p:s=0}. Then, 
\EQ{
 \|  s^{\frac{1}{2}} \Db^2 \Psi \|_{\Ls^2  L^\I_t L^2_x} \lesssim  \| \psi_s \|_{\calS(I)}.
}
\end{lem} 
\begin{proof}
First note that for any $a, b \in \{1,  \dots, 4 \}$ and $ \gamma \in \{0, 1, \dots, 4 \}$ we have 
\ant{
\p_s \Db_a \Db_b \psi_\gamma = F_{s a}  \Db_b \psi_\ga  + F_{s b} \Db_a \psi_\ga  +  \Db_a F_{s b} \psi_\ga + \Db_a \Db_b \Db_\ga \psi_s
}
from which it follows that 
\ant{
s^{\frac{1}{2}}\Db_a \Db_b \psi_\gamma (s)  &=   - s^{\frac{1}{2}} \int_s^\I  \Db_a \Db_b \Db_\ga  \psi_s \ \,  ds'   -  s^{\frac{1}{2}}\int_s^\I F_{s a}  \Db_b \psi_\ga\, ds'    \\
& \quad - s^{\frac{1}{2}}\int_s^\I F_{s b} \Db_a \psi_\ga \, ds' -s^{\frac{1}{2}} \int_s^\I \Db_a F_{s b} \psi_\ga   \, ds'
} 
Therefore, using~\eqref{eq:Fdu} and~\eqref{eq:DF} we have 
\ant{
 \| s^{\frac{1}{2}}\Db^2  \Psi  \|_{ \Ls^2 L^\I_t L^2_x}^2   &\le \int_0^\I  \bigg[  \int_s^\I  \Big( \frac{s}{s'} \Big)^{\frac{1}{2}}  \| (s')^{\frac{3}{2}}  \Db^2  \Db_{t, x} \psi_s \|_{L^\I_t L^2_x} \, \dsp \bigg]^2 \, \ds \\
 & +  \int_0^\I  \bigg[  \int_s^\I  \Big( \frac{s}{s'} \Big)^{\frac{1}{2}}  \| (s')^{\frac{3}{2}}  \bfR^{(0)}( \psi_s,  \underline{\Psi}) \Db  \Psi  \|_{L^\I_t L^2_x} \, \dsp \bigg]^2 \, \ds \\
  & +  \int_0^\I  \bigg[  \int_s^\I  \Big( \frac{s}{s'} \Big)^{\frac{1}{2}}  \| (s')^{\frac{3}{2}} \bfR^{(1)}( \underline{\Psi};  \psi_s, \underline{\Psi}) \Psi \|_{L^\I_t L^2_x} \, \dsp \bigg]^2 \, \ds \\
 & +  \int_0^\I  \bigg[  \int_s^\I  \Big( \frac{s}{s'} \Big)^{\frac{1}{2}}  \| (s')^{\frac{3}{2}}  \bfR^{(0)}(  \Db  \psi_s, \underline{\Psi}) \Psi \|_{L^\I_t L^2_x} \, \dsp \bigg]^2 \, \ds \\
 & +  \int_0^\I  \bigg[  \int_s^\I  \Big( \frac{s}{s'} \Big)^{\frac{1}{2}}  \| (s')^{\frac{3}{2}}  \bfR^{(0)}(    \psi_s, \Db\underline{\Psi}) \Psi \|_{L^\I_t L^2_x} \, \dsp \bigg]^2 \, \ds 
}
The lemma now follows from an application of Schur's test,~\eqref{eq:Db123},~\eqref{eq:pa8}~\eqref{eq:Dpa4}, and~\eqref{eq:painf}. 
\end{proof} 

\begin{proof}[Proof of Proposition~\ref{p:s=0}]
Recall that 
\ant{
\|  \Db   \Psi  \rest_{s=0} \|_{L^\infty_t L^2_x( I \times \Hp^4)}^2 = \|  \Db   \underline{\Psi}  \rest_{s=0} \|_{L^\infty_t L^2_x( I \times \Hp^4)}^2 + \|  \Db   \psi_t  \rest_{s=0} \|_{L^\infty_t L^2_x( I \times \Hp^4)}^2
}
We prove 
\EQ{ \label{eq:Dbx} 
 \|  \Db   \underline{\Psi}  \rest_{s=0} \|_{L^\infty_t L^2_x( I \times \Hp^4)} \lesssim \| \psi_s\|_{ \cS(I)}
 }
 as the proof of the same bound for $\|  \Db   \psi_t  \rest_{s=0} \|_{L^\infty_t L^2_x( I \times \Hp^4)}$ is nearly identical. 
Recall from Lemma~\ref{l:pah}  that for each $a  = 1, \dots, 4$,  $\psi_a$ satisfies the covariant heat equation 
\EQ{ \label{eq:covheat} 
\p_s \psi_a = \Db_s \psi_a  = \Db^b \Db_b \psi_a  + 3 \psi_a + \tensor{F}{_{a}^{b}} \psi_b
 }
Using the shorthand $\psi^a = \h^{ab} \psi_b$ and pairing the above with $\Db_c \Db^c \psi^a$ gives 
\ant{ 
\ang{ \Db_s \psi_a, \Db_c \Db^c \psi^a} = \ang{\Db^b \Db_b \psi_a, \Db_c \Db^c \psi^a}  + 3 \ang{\psi_a, \Db_c \Db^c \psi^a} + \ang{\tensor{F}{_{a}^{b}} \psi_b, \Db_c \Db^c \psi^a}
} 
We expand the first term above as follows. 
\ant{
\ang{ \Db_s \psi_a, \Db_c \Db^c \psi^a} &=  \na_c \ang{  \Db_s \psi_a,  \Db^c \psi^a}  - \ang{ \Db_c \Db_s \psi_a,  \Db^c \psi^a} \\
& =  \na_c \ang{  \Db_s \psi_a,  \Db^c \psi^a}  - \ang{ F_{cs} \psi_a,  \Db^c \psi^a} - \frac{1}{2} \p_s \ang{ \Db_c \psi_a,  \Db^c \psi^a}
}
Note furthermore that
\ant{
\ang{\psi_a, \Db_c \Db^c \psi^a} = \nb_c \ang{\psi_a, \Db^c \psi^a} - \ang{\Db_c \psi_a, \Db^c \psi^a}
}
From the previous three formulas, we have
\ant{
 - \frac{1}{2} \p_s \ang{ \Db_c \psi_a,  \Db^c \psi^a} &=  - 3 \ang{\Db_c \psi_a, \Db^c \psi^a} + \ang{ \Db_b \Db^b \psi_a, \Db_c \Db^c \psi^a}  + \ang{ F_{cs} \psi_a,  \Db^c \psi^a}  \\
&  \quad + \ang{\tensor{F}{_a^b} \psi_b, \Db_c \Db^c \psi^a} 
   -\na_c \ang{  \Db_s \psi_a,  \Db^c \psi^a} + 3 \na_c \ang{\psi_a, \Db^c \psi^a} 
 }
 Integrating over $(s_1, s_2) \times \Hp^4$ gives,  
\EQ{ \label{eq:DbxP12} 
\frac{1}{2}\| \Db  \underline{\Psi}  (s_1) \|_{L^2_x}^2 
&= \frac{1}{2}\| \Db  \underline{\Psi}  (s_2)\|_{L^2_x}^2 - 3\int_{s_1}^{s_2} s \nrm{\Db \underline{\Psi}}_{L^2_x}^2 \, \dvol_{\bfh} \, \ds   \\
& \quad + \int_{s_1}^{s_2}  \| s^{\frac{1}{2}}\Db^b \Db_b \underline{\Psi}   \|_{L^2_x}^2 \,  \ds \\
&  \quad +  \int_{s_1}^{s_2}  \int_{\Hp^4} s\ang{ F_{cs} \psi_a,  \Db^c \psi^a} \,\dvol_{\bfh} \,  \ds \\
&  \quad +  \int_{s_1}^{s_2}  \int_{\Hp^4} s \ang{\tensor{F}{_a^b} \psi_b, \Db_c \Db^c \psi^a} \, \dvol_{\bfh} \, \ds 
}
First note that the second term on the right-hand side is non-positive, and thus can be ignored when proving an upper bound for the left-hand side.
Next we claim that for each fixed $t \in I$ we have
\EQ{ \label{eq:DbxPs2}
\lim_{s_2 \to \infty} \| \Db  \underline{\Psi}  (s_2, t)\|_{ L^2_x}^2  = 0
}
This qualitative statement is a quick consequence of Poincar\'e's inequality, the identity $e \bfD \underline{\Psi} = \nb d u + S(u)(d u, du)$ and Proposition~\ref{p:preg}. Indeed, we have
\begin{align*}
	\nrm{\bfD \underline{\Psi}(s, t)}_{L^{2}_{x}}
	\aleq & \nrm{\nb d u(s, t)}_{L^{2}_{x}} + \nrm{S(u)(du, du)(s, t)}_{L^{2}_{x}} \\
	\aleq & s^{-\frac{1}{2}} \nrm{s^{\frac{1}{2}} \nb^{2} d u(s, t)}_{L_{x}^{2}} + s^{-1} \nrm{s^{\frac{1}{2}} \nb^{2} d u(s, t)}_{L_{x}^{2}}^{2}
\end{align*}
where on the last line we used the Sobolev inequality followed by the Poincar\'e inequality. By Proposition~\ref{p:preg}, the right-hand side vanishes as $s \to \infty$ as desired.

Therefore, by letting $s_2 \to \infty$ in~\eqref{eq:DbxP12}, then taking the supremum over $s_1$, and then the supremum over $t  \in I$ yields the inequality 
\ali{\label{s0 temp 1}
\frac{1}{2}\| \Db  \underline{\Psi}   \|_{\Ls^\infty L_t^\infty L^2_x}^2 & \le   \int_{0}^{\infty}  \| s^{\frac{1}{2}}\Db ^2 \underline{\Psi}  (s) \|_{ L_t^\infty L^2_x}^2 \,  \ds \\
&  \quad +  \sup_{t\in I}\int_{0}^{\infty}   \int_{\Hp^4}   \abs{s\ang{ F_{cs} \psi_a,  \Db^c \psi^a}} \,\dvol_{\bfh} \, \ds \\
&  \quad +   \sup_{t\in I}\int_{0}^{\infty}   \int_{\Hp^4}   \abs{ s\ang{\tensor{F}{_a^b}\psi_b, \Db_c \Db^c \psi^a}} \, \dvol_{\bfh} \, \ds
}
From~\eqref{s0 temp 1} we claim the estimate 
\EQ{ \label{Dbx1} 
\| \Db  \underline{\Psi}  \|_{\Ls^\infty L_t^\infty L^2_x}^2 \lesssim  \|\psi_s\|_{S(I)}^2+\eps\|\Db \underline{\Psi}  \|_{\Ls^\infty L_t^\infty L_x^2}^2
}
Note that~\eqref{Dbx1} implies~\eqref{eq:Dbx} since the second term on the right above can be absorbed into the left-hand side. 

To prove~\eqref{Dbx1} we estimate each of the terms on the right-hand side of~\eqref{s0 temp 1}. To begin the first, main term is handled by Lemma~\ref{l:s=0 temp 2}, 
\ant{
\int_{0}^{\infty}  \| s^{\frac{1}{2}}\Db ^2 \underline{\Psi}  (s) \|_{ L_t^\infty L^2_x}^2 \,  \ds \lesssim \| \psi_s\|_{ \cS(I)}^2
}
For the second term we note that 
\ant{
F_{cs} \psi_a =  \bfR^{(0)}( \psi_c, \psi_s) \psi_a
}
and thus
\ant{
 \abs{s\ang{ F_{cs} \psi_a,  \Db^c \psi^a}}  \lesssim  \abs{ \underline{\Psi}  }^2 |s^{\frac{1}{2}} \psi_s |  | s^{\frac{1}{2}} \Db  \underline{\Psi}  | 
}
It follows that 
\EQ{
 \sup_{t\in I}\int_{0}^{\infty}   \int_{\Hp^4}  & \abs{s\ang{ F_{cs} \psi_a,  \Db^c \psi^a}} \,\dvol_{\bfh} \, \ds  \\
 & \lesssim  \| \underline{\Psi}   \|_{\Ls^\infty L_t^\infty L_x^4}^2 \int_0^\infty \| s^{\frac{1}{2}} \psi_s \|_{L_t^\infty L_x^4} \| s^{\frac{1}{2}} \Db \underline{\Psi}  \|_{L_t^\infty L_x^4} \ds \\
 &  \lesssim 
    \eps^2 \| \psi_s \|_{\cS(I)}^2 
}
where we have used Sobolev embedding and~\eqref{eq:eqna} to get the factor of $\eps^2$ in the last line above and~\eqref{eq:Dpa4}.

To bound the third term on the right-hand side of \eqref{s0 temp 1}, we note that
\ant{
\tensor{F}{_a^b} \psi_b = \bfR^{(0)}(\psi_a, \psi^b) \psi_b
}
Therefore, 
\ant{
 \abs{ s\ang{ \tensor{F}{_a^b} \psi_b, \Db_c \Db^c \psi^a}} \lesssim \abs{ \underline{\Psi} }^3 \abs{ \Db^2 \underline{\Psi} }
 }
 It follows that 
 \begin{align} \notag
  \sup_{t\in I}\int_{0}^{\infty}   \int_{\Hp^4}  & \abs{ s\ang{\tensor{F}{_a^b} \psi_b, \Db_c \Db^c \psi^a}} \, \dvol_{\bfh} \, \ds  \\
& \lesssim   \| \underline{\Psi}  \|^2_{ \Ls^\I L^\I_t L^4_x} \sup_{t \in I}  \int_0^\infty  \| s^{\frac{1}{2}} \underline{\Psi} (s, t) \|_{L^\I_x}  \|  s^{\frac{1}{2}} \Db^2 \underline{\Psi}  \|_{L^\I_t L^2_x}  \, \ds  \notag\\
& \lesssim   \| \na  \underline{\Psi}  \|^2_{ \Ls^\I L^\I_t L^2_x}  \| s^{\frac{1}{2}} \underline{\Psi}   \|_{ L^\I_t \Ls^2 L^\I_x}  \|  s^{\frac{1}{2}} \Db^2 \underline{\Psi}  \|_{\Ls^2 L^\I_t L^2_x} \label{tired2} 
 \end{align}
%
By \eqref{eq:AL4}, we have
\ant{
\nrm{\nb \underline{\Psi}}_{L^\infty_{\ds} L^\infty_t L^2_x}
\lesssim \nrm{\Db \underline{\Psi}}_{L^\infty_{\ds} L^\infty_t L^2_x}
}
  Using the previous inequality together with~\eqref{eq:painf},  Lemma~\ref{l:s=0 temp 1}, and Proposition~\ref{p:pSI}, \eqref{tired2} can be estimated by
  \ant{
  \| \na  \underline{\Psi}  \|^2_{ \Ls^\I L^\I_t L^2_x}  \| s^{\frac{1}{2}} \underline{\Psi}   \|_{ L^\I_t \Ls^2 L^\I_x}  \|  s^{\frac{1}{2}} \Db^2 \underline{\Psi}  \|_{\Ls^2 L^\I_t L^2_x} & \lesssim \eps \| \psi_s \|_{ \cS(I)}  \| \Db  \underline{\Psi}   \|_{ L^\I_t \Ls^\I L^2_x}^2  \\
  & \lesssim  \eps^2  \| \Db  \underline{\Psi}   \|_{ L^\I_t \Ls^\I L^2_x}^2
  }
  which can be absorbed into the left-hand side of~\eqref{s0 temp 1} as long as $\eps>0$ is small enough. This completes the proof of~\eqref{eq:Dbx} and the proof of Proposition~\ref{p:s=0}. 
\end{proof}

Now that we have established the a priori estimates in Proposition~\ref{p:ap} we can complete the proof of Theorem~\ref{t:main1}. The key point is that the conclusion of Proposition~\ref{p:pSI} now holds unconditionally. 

\begin{proof}[Proof of Theorem~\ref{t:main1} and of Theorem~\ref{t:main}]

We begin by proving the global existence and regularity statement in Theorem~\ref{t:main1}. This is a standard argument now that we have established the a priori estimates in Proposition~\ref{p:ap} and Proposition~\ref{p:pSI} and so we give only a brief sketch.  

Let $\eps_{0} = \min \set{\eps_{1}, \eps_{2}, \eps_{3}, \eps_{4}}$, where the epsilons on the right-hand side are as in Proposition~\ref{p:ap}, Proposition~\ref{p:pSI}, Corollary~\ref{c:hr} and Corollary~\ref{c:hr1}, respectively.
Let $(u_0, u_1)$ be any smooth initial data set satisfying~\eqref{eq:data},~\eqref{eq:data1} and with 
\EQ{\label{eq:dataeps0} 
 \| (d u_0, u_1) \|_{ H^1 \times H^1( \Hp^4; T\NN)} =  \eps < \eps_{0}
 }
 Our assumption that the data is smooth and constant outside a  compact subset of $\Hp^4$ means that we can further assume that, say,  
\EQ{ \label{eq:h4data}
  \| (d u_0, u_1) \|_{ H^3 \times H^3( \Hp^4; T\NN)}   = B_0 <  \infty.
  }
 Using the usual argument based on energy estimates one can find a positive time $T>0$ and a unique solution $(u, \partial_t u) \in C( [0, T); H^4 \times H^3( \Hp^4; T\NN))$. Moreover, the time of existence $T>0$ depends only $B_0$ from~\eqref{eq:h4data}.  One then arrives at a finite time blow-up criterion: if the solution $(u, \partial_t u)$ cannot be smoothly extended past a finite time $T^*< \infty$, then 
 \EQ{ \label{eq:bu} 
  \limsup_{t \nearrow T^*}\| (d u, \partial_t u)(t) \|_{ H^3 \times H^3( \Hp^4; T\NN)}  = + \infty.
  }
We will show that~\eqref{eq:bu} is impossible for $T_*<\infty$ under the hypothesis of Theorem~\ref{t:main1}. Assume for contradiction that $(u, \partial_t u)$ with data as in~\eqref{eq:dataeps0} cannot be extended as a smooth solution beyond a time $0<T_*< \infty$ and denote by  $I$ the interval $I = [0, T_*)$. Since $\eps_{0}>0$ in~\eqref{eq:dataeps0} is small enough so that Proposition~\ref{p:ap} holds, we know by Proposition~\ref{p:pSI} that  
\EQ{ \label{eq:pseps} 
 \| \psi_s \|_{ \cS(I)} \lesssim \eps \lesssim \eps_{0}.
}
Note that Proposition~\ref{p:ap} along with~\eqref{eq:pseps} imply that all of the additional estimates from Sections~\ref{s:pr},~\ref{s:wave}, and~\ref{s:proof} hold true. Moreover, since $\eps_{0}$ is small enough so that the hypothesis of Corollary~\ref{c:hr}  and Corollary~\ref{c:hr1} are satisfied, we obtain the uniform bounds 
\EQ{
 \| \De \psi_s \|_{ \cS(I)} \lesssim B_0, \quad \| \Db^3 \Psi \rest_{s=0} \|_{L^\I_t L^2_x( I \times \Hp^4)} \lesssim B_0.
 }
 But then we can argue as in the proof of Proposition~\ref{p:ap}, using the smallness  
 \ant{
 \| (d u, \partial_t u) \rest_{s=0} \|_{L^\I_t (I; H^1 \times H^1( \Hp^4))} \lesssim \eps_{0}
 }
  to conclude that in fact 
 \EQ{
  \sup_{t \in I} \| (d u, \partial_t u)(t) \|_{ H^3 \times H^3( \Hp^4; T\NN)}  \lesssim B_0,
 }
 which contradicts~\eqref{eq:bu}. The same argument applied to negative times shows that any smooth solution $(u, \partial_t u)$ with sufficiently small data as in the statement of Theorem~\ref{t:main1} is defined and  globally regular on $I = \R$.   
 
 Since Proposition~\ref{p:ap} holds now with $I = \R$, the statements~\eqref{eq:psap} and~\eqref{eq:Depap} follow from Proposition~\ref{p:pSI} and Corollary~\ref{c:hr}. Moreover, the scattering statement~\eqref{eq:pscat} is a standard consequence of the argument used to prove Proposition~\ref{p:pSI}. 
 
 It remains to establish~\eqref{eq:upwdec}. Let $u(s, t, x)$ denote the harmonic map heat flow resolution of the wave map $u(t, x)$ and we write  $u(0, t, x) = u(t, x)$. By Lemma~\ref{l:hmhf} we know that each fixed $(t, x) \in  \R \times  \Hp^4$ we have 
 \ant{
   u(s, t, x) \to u_\infty \mas s \to \infty.
 }
 It follows that for each fixed $(t, x) \in  \R \times  \Hp^4$ 
 \ant{
  \abs{ u(t, x)  -  u_\infty} =  \abs{ u(0, t, x) - u_{\infty}}  =  \abs{ \int_0^\infty  \p_s u(s, t, x)  \, ds}.
}
Recalling that $\p_s u = e \psi_s$, we see that 
\ant{
 \lim_{t \to \infty} \| u(t)  -  u_\infty \|_{L^\I_x}  \le   \lim_{t \to \infty} \int_0^\infty  \| \p_s u(s, t) \|_{L^\infty_x}  \, ds  =   \lim_{t \to \infty}  \int_0^\I  \|  \psi_s(s, t) \|_{L^\I_x} \, ds.
}
From the above we observe that the proof of~\eqref{eq:upwdec} is now a consequence of the following two statements and the dominated convergence theorem. 
\begin{itemize}
\item[(a)] There exists an integrable function $f(s) \ge 0$, $ \int_0^\I f(s)  \, ds < \infty$ so that 
\EQ{
 \| \p_s u(s, t) \|_{L^\infty_x} = \| \psi_s (s, t) \|_{L^\infty_x}  \le f(s),
 }
 uniformly in $t \in \R$. 
 \item[(b)] For each fixed $s \in (0, \infty)$ we have 
 \EQ{
  \| \psi_s(s, t) \|_{L^\I_x}  \to 0 \mas t \to \infty.
  }
  \end{itemize} 
 
To prove (a) we define $f(s)$ simply by $f(s):=   \| \psi_s(s) \|_{L^\infty_{t, x}}$. First, we estimate $f(s)$ on the interval $s \in (0, 1]$. The key ingredient in this region is the higher regularity estimate~\eqref{eq:Depap},  which breaks the scaling in $s$. Indeed, by Gagliardo-Nirenberg we have 
 \EQ{
  \| \psi_s( s) \|_{L^\infty_{t, x}} &\lesssim \|\psi_s(s) \|^{\frac{1}{2}}_{ L^\infty_t L^8_x}  \| (-\De)^{\frac{1}{2}}\psi_s(s) \|^{\frac{1}{2}}_{ L^\infty_t L^8_x} \\
  & \lesssim s^{-\frac{3}{4}} \|s^{\frac{1}{2}}\na \psi_s(s) \|^{\frac{1}{4}}_{ L^\infty_t L^2_x} \| s \na^2 \psi_s(s) \|_{L^\I_t L^2_x}^{\frac{1}{2}} \| s^{\frac{1}{2}} \na  \De\psi_s(s) \|^{\frac{1}{4}}_{ L^\infty_t L^2_x}   \\
  & \lesssim s^{-\frac{3}{4}} \| \psi_s \|_{ \cS(\R)}^{\frac{3}{4}} \| \De \psi_s \|_{ \cS(\R)}^{\frac{1}{4}} \lesssim s^{-\frac{3}{4}} ,
  }
  where the last two lines follows from the definition of the norm $ \cS(\R)$,~\eqref{eq:psap}, \eqref{eq:Depap}, and~\eqref{eq:preg8}. Therefore, 
  \ant{
  \int_0^1  \| \psi_s( s) \|_{L^\infty_{t, x}} \, ds  \lesssim  \int_{0}^1 s^{-\frac{3}{4}} \, ds  \lesssim 1.
  } 
 For the interval $s \in [1, \infty)$ we note that~\eqref{eq:psdec} provides, say,  the bound 
 \ant{
 \| \psi_s ( s) \|_{L^\infty_{t,x}}  \lesssim  s^{-2} \mas s \to \infty 
 }
 so that 
 \ant{ 
 \int_1^\infty \| \psi_s ( s) \|_{L^\infty_{t,x}} \, ds \lesssim 1,
 }
which completes the proof of (a). 
 
  Next, we prove (b). Fix $s_0>0$. By an approximation argument, we may assume that the scattering profile $\phi_+(s_0)$ in~\eqref{eq:pscat} is smooth with compactly supported initial data $ \vec \phi_+(s_0, 0)$. For such a free wave, i.e., for $\phi_+(s_0, t)$,  we can deduce from the Gagliardo-Nirenberg inequality and the global dispersive estimate~\eqref{eq:longd} for the free propagator that 
  \EQ{
   \| \phi_+(s_0, t) \|_{L^\I_x} \to 0 \mas t \to \infty.
  }
  Next, again using Gagliardo-Nirenberg, we have 
  \begin{align*}
   \| &\psi_s(s_0, t) -  \phi_+(s_0, t) \|_{L^\infty_x}  \\ &\lesssim \| \na( \psi_s(s_0, t) -  \phi_+(s_0, t)) \|_{L^2_x}^{\frac{1}{2}} \Big(  \| \na^3 \psi_s(s_0) \|_{L^\infty_t L^2_x} +  \| \na^3 \phi_+(s_0) \|_{L^\infty_t L^2_x} \Big)^{\frac{1}{2}}  \\
   &\lesssim C_{s_0, \phi_+} \| \na( \psi_s(s_0, t) -  \phi_+(s_0, t)) \|_{L^2_x}^{\frac{1}{2}} \longrightarrow 0 \mas t  \to \infty,
   \end{align*} 
where we have used~\eqref{eq:preg8} to control the first term in parenthesis \emph{for fixed $s_0>0$}, and then~\eqref{eq:pscat} to obtain the convergence in the last line. This completes the proof of Theorem~\ref{t:main1} and thus of Theorem~\ref{t:main} as well. \end{proof}
\bibliographystyle{plain}
\bibliography{researchbib}

 \bigskip

\centerline{\scshape Andrew Lawrie, Sung-Jin Oh}
\smallskip
{\footnotesize
 \centerline{Department of Mathematics, The University of California, Berkeley}
\centerline{970 Evans Hall \#3840, Berkeley, CA 94720, U.S.A.}
\centerline{\email{ alawrie@math.berkeley.edu, sjoh@math.berkeley.edu}}
} 

 \medskip

\centerline{\scshape Sohrab Shahshahani}
\medskip
{\footnotesize
 \centerline{Department of Mathematics, The University of Michigan}
\centerline{2074 East Hall, 530 Church Street
Ann Arbor, MI  48109-1043, U.S.A.}
\centerline{\email{shahshah@umich.edu}}
} 

\end{document}